\documentclass[10pt]{article} 
\usepackage[usenames,dvipsnames]{xcolor}
\usepackage{enumerate}

\RequirePackage[OT1]{fontenc}
\usepackage{underscore}
\usepackage[portrait,margin=3.5cm]{geometry}
\usepackage{mathrsfs}
\usepackage[colorlinks,citecolor=blue,urlcolor=blue,linkcolor=blue,linktocpage=true]{hyperref}
\usepackage{amssymb,amsthm,amsfonts,amsbsy,latexsym}
\usepackage[]{amsmath}

\usepackage[textsize=small]{todonotes}       
\usepackage{graphicx}
\usepackage[english]{babel}
\usepackage{subfigure}
\usepackage{appendix}

\usepackage{enumerate}

\usepackage{algorithmicx}
\usepackage{algorithm}

\usepackage{color}              
\usepackage{graphicx}
\usepackage[]{amsmath}
\usepackage{amsthm}
\usepackage{amssymb}
\usepackage{hyperref}
\usepackage{comment}
\usepackage{microtype}

\definecolor{MyDarkblue}{rgb}{0,0.08,0.50}
\definecolor{Brickred}{rgb}{0.65,0.08,0}

\newtheorem{theorem}{Theorem}[section]
\newtheorem{lemma}[theorem]{Lemma}
\newtheorem{proposition}[theorem]{Proposition}
\newtheorem{corollary}[theorem]{Corollary}

\newtheorem{remark}[theorem]{Remark}

\newcommand{\E}[1]{\mathbb{E}\left[#1\right]}

\newcommand{\set}[1]{\left\{#1\right\}}

\newcommand{\s}[2]{\sum_{#1}^{#2}}

\newcommand{\cC}{\mathcal{C}}
\newcommand{\cE}{\mathcal{E}}\newcommand{\cF}{\mathcal{F}}
\newcommand{\cG}{\mathcal{G}}

\newcommand{\cP}{\mathcal{P}}
\newcommand{\cS}{\mathcal{S}}\newcommand{\cT}{\mathcal{T}}\newcommand{\cU}{\mathcal{U}}
\newcommand{\cW}{\mathcal{W}}
\newcommand{\cY}{\mathcal{Y}}

\newcommand{\Var}{{\rm Var}}

\newcommand{\e}{{\mathrm e}}

\numberwithin{equation}{section}

\newcommand{\indicator}[1]{1_{\set{#1}}}

\newcommand{\Poi}{\mathrm{Poi}}

\renewcommand{\emptyset}{\varnothing}

\newcommand*{\be}{\begin{equation}}
\newcommand*{\ee}{\end{equation}}
\newcommand*{\ba}{\begin{aligned}}
\newcommand*{\ea}{\end{aligned}}
\newcommand*{\barr}{\begin{array}{c}}
\newcommand*{\earr}{\end{array}}

\def\namedlabel#1#2{\begingroup
    #2%
    \def\@currentlabel{#2}%
    \phantomsection\label{#1}\endgroup
}

\renewcommand{\P}[1]{\mathbb{P}\left(#1\right)}

\newcommand{\bigO}{\mathcal{O}}
\newcommand{\IND}[1]{1_{#1}}

\newcommand{\lr}[1]{\left(#1\right)}

\renewcommand{\c}[1]{\mathcal{#1}}

\newcommand{\PHI}{\Phi^{(1)}}
\newcommand{\PHItwo}{\Phi^{(2)}}
\newcommand{\PHIthree}{\Phi^{(3)}}

\newcommand{\NUstar}{\nu^*}
\newcommand{\RHO}{\frac{a+b}{2}\PHItwo}
\newcommand{\spm}{\{ + , - \}}
\newcommand{\Pois}[1]{\text{Poi}\lr{#1}}
\newcommand*{\BA}{\be \ba}
\newcommand*{\EA}{\ea \ee}
\newcommand*{\VAR}{\text{Var}}
\newcommand*{\BLM}{\cite{BoLeMa15}}

\newcommand*{\EF}{g} 

\newcommand*{\PHImin}{\phi_{\text{min}}} 
\newcommand*{\PHImax}{\phi_{\text{max}}} 

\newcommand{\bdphi}{\PHImax}

\newcommand{\DCA}[1]{Degree-Corrected Extension of  #1 in \BLM}

\newcommand{\Esub}[2]{\mathbb{E}_{#2}\left[#1\right]}

\newcommand{\Psub}[2]{\mathbb{P}_{#2}\lr{#1}}
\newcommand{\NRM}[1]{{{\left\| #1\right\|}}} 

\newcommand{\calA}{\mathcal{A}}
\newcommand{\calE}{\mathcal{E}}
\newcommand{\calU}{\mathcal{U}}

\newcommand*{\tr}{\text{tr}}

\begin{document}

\title{Non-Backtracking Spectrum of Degree-Corrected Stochastic Block Models}
\author{Lennart Gulikers\footnote{Microsoft Research - INRIA Joint Centre and \'Ecole Normale Sup\'erieure, France. E-mail: lennart.gulikers@inria.fr} , Marc Lelarge\footnote{INRIA Paris and \'Ecole Normale Sup\'erieure, France. E-mail: marc.lelarge@ens.fr}, Laurent Massouli\'e\footnote{Microsoft Research - INRIA Joint Centre, France. E-mail: laurent.massoulie@inria.fr}}
\maketitle

\abstract{
Motivated by community detection, we characterise the spectrum of the non-backtracking matrix $B$ in the Degree-Corrected Stochastic Block Model.

Specifically, we consider a random graph on $n$ vertices partitioned into two asymptotically equal-sized clusters. The vertices have i.i.d. weights $\{ \phi_u \}_{u=1}^n$ with second moment $\PHItwo$. The intra-cluster connection probability for vertices $u$ and $v$ is $\frac{\phi_u \phi_v}{n}a$ and the inter-cluster connection probability is $\frac{\phi_u \phi_v}{n}b$. 

We show that with high probability, the following holds: The leading eigenvalue of the non-backtracking matrix $B$ is asymptotic to $\rho = \frac{a+b}{2} \PHItwo$. The second eigenvalue is asymptotic to  $\mu_2 = \frac{a-b}{2} \PHItwo$ when $\mu_2^2 > \rho$, but asymptotically bounded by $\sqrt{\rho}$ when $\mu_2^2 \leq \rho$. All the remaining eigenvalues are asymptotically bounded by $\sqrt{\rho}$. As a result, a clustering positively-correlated with the true communities can be obtained based on the second eigenvector of $B$ in the regime where $\mu_2^2 > \rho.$

In a previous work we obtained that detection is impossible when $\mu_2^2 < \rho,$ meaning that there occurs a phase-transition in the sparse regime of the Degree-Corrected Stochastic Block Model. 

As a corollary, we obtain that Degree-Corrected Erd\H{o}s-R\'enyi graphs asymptotically satisfy the graph Riemann hypothesis, a quasi-Ramanujan property.

A by-product of our proof is a weak law of large numbers for local-functionals on Degree-Corrected Stochastic Block Models, which could be of independent interest.
}

 \section{Introduction}
The non-backtracking matrix $B$ of a graph $G = (V,E)$ is indexed by the set of its oriented edges $\vec E = \{ (u,v): \{u,v \} \in E \}$. For $e= (e_1, e_2), f = (f_1, f_2)  \in \vec E$, $B$ is defined as
\[ B_{ef} = \IND{e_2 = f_1} \IND{e_1 \neq f_2}. \]
This matrix was introduced by Hashimoto \cite{Ha89} in 1989.

We study the spectrum of $B$ when $G$ is a random graph generated according to the Degree-Corrected Stochastic Block Model (DC-SBM) \cite{KaBrNe11}. We characterise its leading eigenvalues and corresponding eigenvectors when the number of vertices in $G$ tends to infinity. Our motivation stems from community detection problems: experiments in \cite{KrMoMoNeSlZdZh13} show that the spectral method based on the non-backtracking matrix seems to work well on real datasets. We test the robustness of this method and show in particular that, above a certain threshold, the second eigenvector of $B$ is correlated with the underlying communities. 

The DC-SBM \cite{KaBrNe11} is an extension of the \emph{ordinary} Stochastic Block Model (SBM) \cite{HoLaLe83}. The latter model has as a drawback that vertices in the same community are stochastically indistinguishable and it therefore fails to accurately describe networks with high heterogeneity. Compare this to fitting a straight line on intrinsically curved data, which is doomed to miss important information. The DC-SBM is a more realistic model: it allows for very general degree-sequences. 

The special case of the DC-SBM under consideration here is defined as follows: It is a random graph on $n$ vertices partitioned into two asymptotically equal-sized clusters. The vertices have bounded i.i.d. weights $\{ \phi_u \}_{u=1}^n$ with second moment $\PHItwo$.  The intra-cluster connection probability for vertices $u$ and $v$ is $\frac{\phi_u \phi_v}{n}a$ and the inter-cluster connection probability is $\frac{\phi_u \phi_v}{n}b$, for two constants $a,b >0.$ 

Note that those graphs are thus sparse, which is a challenging regime for community detection. Indeed, in the \emph{ordinary} SBM (obtained by putting $\phi_1 = \ldots = \phi_n = 1$), an instance of the graph might not contain enough information to distinguish between the two clusters if the difference between $a$ and $b$ is small. More precisely, reconstruction is impossible when $(a-b) ^2 \leq 2(a+b)$ \cite{MoNeSl15Rec}. Interestingly, positively-correlated reconstruction can be obtained by thresholding the second-eigenvector of $B$ \cite{BoLeMa15} immediately above the threshold (i.e., $(a-b) ^2 > 2(a+b)$). The SBM thus has a phase-transition in its sparse regime.  

Does the DC-SBM exhibit a similar behaviour? We showed in an earlier work \cite{GuLeMa15} that detection is impossible when $ (a-b) ^2 \PHItwo \leq 2(a+b) $. 
In our current work we analyse the regime where $ (a-b) ^2 \PHItwo > 2(a+b) $. We answer the following questions: is detection possible in this regime and if so, can we use again the non-backtracking matrix or do we need to modify it? \emph{A priori this is unclear, because an algorithm solely based on $B$ cannot use any information on the weights as input.} Our main result shows that the spectral method based on the non-backtracking matrix (thus the same method as in \cite{BoLeMa15}) successfully detects communities in the regime $ (a-b) ^2 \PHItwo > 2(a+b) $. Surprisingly, no modification of the matrix, nor information about the weights is needed (compare this to the adjacency matrix, which needs to be adapted to the degree-corrected setting \cite{GuLeMa151}), which shows the robustness of the method. Moreover as in the standard SBM, the algorithm is optimal in the sense that it works all the way down to the detectability-threshold.
 
 Informally, we have the following results: 
With high probability, the leading eigenvalue of the non-backtracking matrix $B$ is asymptotic to $\rho = \frac{a+b}{2} \PHItwo$. The second eigenvalue is asymptotic to  $\mu_2 = \frac{a-b}{2} \PHItwo$ when $\mu_2^2 > \rho$, but asymptotically bounded by $\sqrt{\rho}$ when $\mu_2^2 \leq \rho$. All the remaining eigenvalues are asymptotically bounded by $\sqrt{\rho}$. Further, a clustering positively-correlated with the true communities can be obtained based on the second eigenvector of $B$ in the regime where $\mu_2^2 > \rho$ (i.e., precisely when $(a-b) ^2 \PHItwo > 2(a+b) $).

 A side-result is that Degree-Corrected Erd\H{o}s-R\'enyi graphs asymptotically satisfy the graph Riemann hypothesis, a quasi-Ramanujan property.

In our proof we derive and use a weak law of large numbers for local-functionals on Degree-Corrected Stochastic Block Models, which could be of independent interest.

\subsection{Community detection background}
In this paper we are interested in community detection: The problem of clustering vertices in a graph into groups of "similar" nodes. In particular, the graphs here are generated according to the DC-SBM and the goal is to retrieve the spin (or group-membership) of the nodes based on a single observation of the DC-SBM. 

When the average degree of a vertex grows sufficiently fast with the size of the network (i.e., the average degree is  $\Omega( \log(n)) $), we speak about dense networks. Community-detection is then well understood and we consider instead sparse graphs where the average degree is bounded by a constant. This setting is more realistic as most real networks are sparse, but is at the same time more challenging. Indeed, traditional methods based on the Adjacency or Laplacian matrix working well in the dense case break down when employed in the sparse case. 

In the sparse regime, with high probability, at least a positive fraction of the nodes is isolated. Consequently, one cannot hope to find the community-membership of \emph{all} vertices. We therefore address here the problem of finding a clustering that is positively correlated with the true community-structure. 

In \cite{DeKrMoZd11} it was first conjectured that a detectability phase transition exists in the \emph{ordinary} SBM: When $(a-b)^2 > 2(a+b)$, the belief propagation algorithm would succeed in finding such a positively correlated clustering. Conversely, due to a lack of information, detection would be impossible when $(a-b)^2 \leq 2(a+b)$. 

In \cite{MoNeSl15Rec}, impossibility of reconstruction when $(a-b)^2 \leq 2(a+b)$ is shown for the SBM. This paper builds further on a tree-reconstruction problem in \cite{EvKePeSc00}. 

The authors of \cite{KrMoMoNeSlZdZh13} conjectured that detection using the second eigenvector of $B$ would succeed all the way down to the conjectured detectability threshold. Two variants of this so-called spectral redemption conjecture were proven before the work in \cite{BoLeMa15} appeared: 

In \cite{Ma14} it is shown that detection based on the second eigenvector of a matrix counting self-avoiding paths in the graph leads to consistent recovery when  $\lr{\frac{a-b}{2}}^2 > \frac{a+b}{2}$. 

Independently, in  \cite{MoNeSl15}, the authors prove the positive side of the conjecture by using a constructing based on counting non-backtracking paths in graphs generated according to the SBM.

More recently, in \cite{BoLeMa15} the spectral redemption conjecture is proved. This work moreover determines the limits of community detection based on the non-backtracking spectrum in the presence of an arbitrary number of communities.

Here we extend the work in \cite{BoLeMa15}  to the more general setting of the DC-SBM.

\subsection{Quasi Ramanujan property}
Following the definition introduced in \cite{LuPhSa88}, a $k$-regular graph is Ramanujan if its second largest absolute eigenvalue is no larger than $2\sqrt{k-1}$. In \cite{HoStTe06}, a graph is said to satisfy the graph Riemann hypothesis if $B$ has no eigenvalues $\lambda$ such that $|\lambda| \in (\sqrt{\rho_B} , \rho_B)$, where $\rho_B$ is the Perron-Frobenius eigenvalue of $B$. The graph Riemann hypothesis can be seen as a generalization of the Ramanujan property, because a regular graph satisfies the graph Riemann hypothesis if and only if it has the Ramanujan property \cite{HoStTe06,Mu03}. 

Now, put $a = b = 1$ to obtain a Degree-Corrected Erd\H{o}s-R\'enyi graph where vertices $u$ and $v$ are connected by an edge with probability $\frac{\phi_u \phi_v}{n}$. Our results imply that, with high probability, $\rho_B = \PHItwo + o(1)$, while all other eigenvalues are in absolute value smaller than $\sqrt{\PHItwo} + o(1)$. Consequently, these  Degree-Corrected Erd\H{o}s-R\'enyi graphs asymptotically satisfy the graph Riemann hypothesis.

\subsection{Outline and main differences with ordinary SBM}
We follow the same general approach as in \cite{BoLeMa15}. We focus primarily on the differences and complications here: we often omit or shorten the proof of a statement if it may be proven in a very similar way.

In Section \ref{sec::results} we define the DC-SBM and state the assumptions we make. This is then followed by Theorem \ref{th::4} on the spectrum of B and its consequences for community detection, Theorem \ref{th::5}. 

In Section \ref{sec::preliminaries}, we give the necessary background on non-backtracking matrices. Further, we give an extension of the Bauer-Fike Theorem, that first appeared in \cite{BoLeMa15}.

In Section \ref{sec::main} we give the proof of Theorem \ref{th::4}. It builds on Propositions \ref{prop::19} and \ref{prop::20}. Their proofs are deferred to later sections.

In Section \ref{sec::branching} we consider two-type branching process where the offspring distribution is governed by a Poisson mixture to capture the weights of the vertices. We associate two martingales to this process and extend limiting results by Kesten and Stigum	\cite{KeSt66,KeSt662}. 
Hoeffding's inequality plays an important role here to prove concentrations results for the weights. Further, we define a cross-generational functional on these branching processes that is correlated with the spin of the root. 

In Section \ref{sec::coupling} we state a coupling between local neighbourhoods and the branching process with weights in Section \ref{sec::branching}. We established this coupling in an earlier work \cite{GuLeMa15}, it is technically more involved than the ordinary coupling on graphs with unit weight. It is crucial that the weights in the graph and the branching process are perfectly coupled. We further establish a growth condition on the local neighbourhoods, using a stochastic domination argument that is more involved than its analogue in unweighed graphs.  

In Section \ref{sec::law_large_numbers} we define local functionals that map graphs, together with their spins \emph{and} weights to the real numbers. We establish, using Efron-Stein's inequality, a weak law of large numbers for those functionals, which could be of independent interest. Part of the work here is again hidden in the coupling from \cite{GuLeMa15}. 

In Section \ref{sec::Prop19} we apply those local functionals to establish Proposition \ref{prop::19}.

In Section \ref{sec::norms} we decompose powers of the matrix $B$ as a sum of products. This technique appeared first in \cite{Ma14} for matrices counting self-avoiding paths and was elaborated in \cite{BoLeMa15}. To bound the norm of the individual matrices occurring in the decomposition, we use the trace method initiated in \cite{FuKo81}. In doing so, we need to bound the expectation of products of higher moments of the weights over certain paths. This is a significant complication with respect to the ordinary SBM, see Section \ref{ssec::comparison} for a comparison. 

In Section \ref{sec::detection} we prove that positively correlated clustering is possible based on the second eigenvector of $B$, i.e., Theorem \ref{th::5}. We use the symmetry present in the two-communities setting here, which gets in general broken in models with more than two communities.

Detailed proofs of the statements in Sections \ref{sec::branching}, \ref{sec::coupling}, \ref{sec::law_large_numbers}, \ref{sec::norms} and \ref{sec::detection} can be found in Appendices \ref{App::branching} - \ref{App::detection}.

In each section we give a detailed comparison with the ordinary SBM.

\section{Main Results}
\label{sec::results}
We define our model more precisely and state the two main theorems.

We consider random graphs on $n$ nodes $V = \{1, \ldots, n \}$ drawn according to the Degree-Corrected Stochastic Block Model \cite{KaBrNe11}. The vertices are partitioned into two clusters of sizes $n_+$ and $n_-$ by giving each vertex $v$ a spin $\sigma(v)$ from $\spm$. The vertices have i.i.d. weights $\{\phi_u\}_{u=1}^n$ governed by some law $\nu$ with support in $[ \PHImin , \PHImax ],$ where $0 < \PHImin \leq \PHImax < \infty$ are constants. We denote the second moment of the weights by $\PHItwo$.  An edge is drawn between nodes $u$ and $v$ with probability $\frac{\phi_u \phi_v}{n} a$ when $u$ and $v$ have the same spin and with probability $\frac{\phi_u \phi_v}{n} b$ otherwise. The model parameters $a$ and $b$ are constant.
We assume that for some constant $\gamma \in (0,1]$, 
\be
n_{\pm} = \frac{n}{2} + \bigO(n^{1-\gamma}),
\label{eq::gamma}
\ee
i.e., the communities have nearly equal size. 

The ordinary SBM on two or more communities was first introduced in \cite{HoLaLe83}, which is a generalization of Erd\H{o}s-R\'enyi graphs. 
The Degree-Corrected SBM appeared first in \cite{KaBrNe11}. 
General inhomogeneous random graphs are considered in \cite{BoBeJa07}.
 
 Note that we retrieve the two-communities ordinary SBM by giving all nodes \emph{unit} weight.

Local neighbourhoods in the sparse graphs under consideration are tree-like with high probability. In \cite{GuLeMa15} we showed that these trees are distributed according to a Poisson-mixture two-type branching process, detailed in Section \ref{sec::branching} below. We denote the mean progeny matrix of the branching process by
\be M = \frac{\PHItwo}{2} 
\begin{pmatrix} 
a & b \\
b & a 
\label{def::M}
\end{pmatrix}.\ee 
We introduce the orthonormal vectors 
\be \EF_1 =  \frac{1}{\sqrt{2}}
\begin{pmatrix} 
1  \\
1  
\end{pmatrix} \text{, and } 
\EF_2 =  \frac{1}{\sqrt{2}}
\begin{pmatrix} 
1  \\
-1  
\end{pmatrix},   \ee
together with the scalars 
\be \rho = \mu_1 = \RHO  \text{, and }  \mu_2 = \frac{a-b}{2} \PHItwo. \ee Then, $\EF_k$ ($k=1,2$) are the left-eigenvectors of $M$ associated to eigenvalues $\mu_k$:
\be \EF_k^* M = \mu_k \EF_k^*, \quad k =1,2. \label{eq::EV_M} \ee
Note that $\rho$ and $\mu_2$ are also asymptotically eigenvalues of the expected adjacency matrix conditioned on the weights. 

Indeed, if $A$ denotes the adjacency matrix, and if   $\psi_1$ and $\psi_2$ are the vectors defined for $u \in V$ by $\psi_1(u) = \frac{1}{\sqrt{2}} \phi_u$ and $\psi_2(u) = \frac{1}{\sqrt{2}} \sigma_u \phi_u$, then
$$
\E{A|\phi_1, \ldots, \phi_n} = \frac{a+b}{n}\psi_1 \psi_1^* + \frac{a-b}{n}\psi_2 \psi_2^* - a \frac{1}{n} \text{diag}\{\phi_u^2 \}.
$$
Put $\widehat{\psi}_i = \frac{\psi_i}{\| \psi_i \|_2}$.
Then, by the law of large numbers, for $i = 1,2$,
$$
\left\| \E{A|\phi_i, \ldots, \phi_n} \widehat{\psi}_i - \mu_i \widehat{\psi}_i \right\|_2 \to 0,
$$
in probability, as $n$ tends to $\infty$.

Finally, we define for $k \in \{1,2\}$,
\be \chi_k(e) = \EF_k(\sigma(e_2)) \phi_{e_2}, \quad \text{ for } e \in \vec E. \label{def::chi_k} \ee
We show that the \emph{candidate eigenvectors} 
\be \zeta_k = \frac{B^\ell B^{* \ell} \check{\chi_k}}{\| B^\ell B^{* \ell} \check{\chi_k} \|} \label{def::zeta_k} \ee
are then, for $\ell \sim \log(n)$, asymptotically aligned with the first two eigenvectors of $B$. Note the weight in \eqref{def::chi_k}, which is \emph{not} present in the ordinary SBM. 
\begin{theorem}[\DCA{Theorem $4$}]
\label{th::4}
Let $G$ be drawn according to the DC-SBM such that assumption \eqref{eq::gamma} holds. Assume that $\ell = C_{\text{min}} \log(n),$ with $C_{\text{min}} > 0$ a small constant defined in \eqref{def::C}.

If $\mu_2^2 > \rho$, then, with high probability, the eigenvalues $\lambda_i$ of $B$ satisfy 
$$
| \lambda_1 - \rho | = o(1), | \lambda_2 - \mu_2 | = o(1),  \quad \hbox{ and, for $i \geq 3$,} \quad |\lambda_i|\leq \sqrt{\rho} + o(1).
$$
Further, if, for $k \in \{1,2\}$, $\xi_k$ is a normalized eigenvector associated to $\lambda_k$, then $\xi_k$ is asymptotically aligned with
$\zeta_k$. 
 The vectors $\xi_1$ and $\xi_2$ are asymptotically orthogonal.

If $\rho > 1$, and $\mu_2^2 \leq \rho$, then, with high probability, the eigenvalues $\lambda_i$ of $B$ satisfy 
$$
| \lambda_1 - \rho | = o(1),\quad \hbox{ and, for $i \geq 2$,} \quad |\lambda_i|\leq \sqrt{\rho} + o(1).
$$ 
Further, $\xi_1$ is asymptotically aligned with $\zeta_1$.
\end{theorem} 

Note that $\mu_2^2 > \rho$ implies $\rho > 1$, so that we consider the DC-SBM precisely in the regime where a giant component emerges, see \cite{BoBeJa07}.

In Theorem \ref{th::5} we show that positively correlated clustering is possible based on the second eigenvector of $B$ when above the feasibility threshold. More precisely, let $\widehat{\sigma} = \{ \widehat{\sigma}(v) \}_{v \in V}$ be estimators for the spins of the vertices.  Following \cite{DeKrMoZd11}, we say that $\widehat{\sigma}$ has positive overlap with the true spin configuration $\sigma = \{  \sigma(v) \}_{v \in V}$ if for some $\delta > 0$, with high probability,
$$
\min_{p}  \frac{1}{n} \s{v=1}{n} \IND{\widehat{\sigma}(v) = p \circ \sigma(v) }  > \frac{1}{2} + \delta,
$$
where $p$ runs over the identity mapping on $\spm$ and the permutation that swaps $+$ and $-$. 

\begin{theorem}[\DCA{Theorem $5$}]
\label{th::5}
Let $G$ be drawn according to the DC-SBM such that assumption \eqref{eq::gamma} holds and such that $\mu_2^2 > \rho$. Let $\xi_2$ be the second normalized eigenvector of $B$.

Then, there exists a deterministic threshold $\tau \in \mathbb{R}$, such that the following procedure yields asymptotically positive overlap:
Put for vertex $v \in V$ its estimator $\widehat{\sigma}(v) = +$ 
if $\sum_{e : e_2 = v}  \xi_2 (e) > \frac{\tau}{ \sqrt { n}}$ and put  $\widehat{\sigma}(v) = -$ otherwise. 
\end{theorem}

\subsection{Notation}
We say that a sequence $(E_n)_n$ of events happens with high probability (w.h.p.) if $\lim_{n \to \infty} \P{E_n} = 1$. 

We denote by $\| \cdot \|$ both the euclidean norm for vectors and the operator norm of matrices. I.e., for vectors $x = (x_1, \ldots, x_m)$, and a matrix $A$,
$ \| x \| = \sqrt{ \sum_{u=1}^{m} x_u^2 }, $
and
 $ \| A \| = \sup_{x, \|x\| = 1} \| Ax \|. $

Below we use that the neighbourhoods with a radius no larger than $C_{\text{coupling}} \log_\rho (n)$ can be coupled w.h.p. to certain branching processes, where
\be
C_{\text{coupling}}:= \frac{ \lr{ \frac{1}{3} - \frac{1}{9}\log (4/e) } \wedge  \lr{ \frac{1}{80} \wedge \frac{\gamma}{4} }  }{\log_{\rho}(2(a+b)\PHImax^2)}.\ee
We put,
\be C_{\text{min}} = \frac{1}{10} C_{\text{coupling}}  \label{def::C} \ee
and  consider often neighbourhoods of radius $C_{\text{min}} \log_\rho (n)$.

We denote the $k$-th moment of the weight distribution $\nu$ by $\Phi^{(k)}$. I.e., $\E{\phi_1^k} = \Phi^{(k)}.$ 

The non-backtracking property for oriented edges $e,f \in \vec E$ is denoted by $e \to f$, i.e., $e_2 = f_1$ and $f_2 \neq e_1$.

In proofs, we often use the symbols $c_1, c_2, \ldots$ for suitably chosen constants. 

\section{Preliminaries}
\label{sec::preliminaries}

\subsection{Background on non-backtracking matrix}
We repeat here the most important observations made in \cite{BoLeMa15}. 

Firstly, for any $k \geq 1$, $B^k_{ef}$ counts the number of non-backtracking paths between oriented edges $e$ and $f$. A non-backtracking path is defined as an oriented path between two oriented edges such that no edge is the inverse of its preceding edge, i.e., the path makes no backtrack. 

Another import observation is that $(B^*)_{ef} = B_{fe} = B_{e^{-1} f^{-1}},$ where for oriented edge $e = (e_1,e_2)$, we set $e^{-1} = (e_2,e_1)$. If we introduce the swap notation, for $x \in \mathbb{R}^{\vec E}$,
$$
\check{x}_e = x_{e^{-1}}, \quad e \in \vec{E},
$$
then for any $x,y \in \mathbb{R}^{\vec E}$, and integer $k \geq 0$,
$$
\langle y,B^k x \rangle = \langle B^k \check y, \check x \rangle.
$$
Denote by $P$ the matrix on $\mathbb{R}^{\vec E \times \vec E}$, defined on oriented edges $e,f$ as
$$
P_{ef} = \IND{f = e^{-1}}.
$$
Then, $Px = \check x$, $P^* = P$ and $P^{-1} = P$. Further,
$$
(B^kP)^* = P (B^*)^k = B^k P,
$$
so that we can write the symmetric matrix $B^k P$ in diagonal form: Let $(\sigma_{k,j})_j$ be eigenvalues of $B^k P$ ordered in decreasing order of absolute value, and let $(x_{k,j})_j$ be the corresponding orthonormal eigenvectors. Then,
\be
B^k = (B^k P) P = \sum_j \sigma_{k,j} x_{k,j} x_{k,j}^* P  = \sum_j \sigma_{k,j} x_{k,j} \check{x}_{k,j}^* = \sum_j s_{k,j} x_{k,j} y_{k,j}^*,
\label{eq::svalues} 
\ee
where $s_{k,j} = |\sigma_{k,j}|$ and $y_{k,j} = \text{sign}(\sigma_{k,j}) \check{x}_{k,j}$. Since $P$ is an orthogonal matrix, $(\check{x}_{k,j})_j$ form an orthonormal base for $\mathbb{R}^{\vec E}$ and the term furthest on the right of \eqref{eq::svalues} is thus the spectral value decomposition of $B^k$. Now, if $B$ is irreducible and if $\xi$ denotes the normalized Perron eigenvector of $B$ with eigenvalue $\lambda_1(B) > 0$, we have
$
\lambda_1(B) = \lim_{k \to \infty} (\sigma_{k,1})^{1/k},
$
and
$\lim_{k \to \infty} \| x_{k,1} - \xi \| = 0.$

In \cite{BoLeMa15}, the Bauer-Fike Theorem is extended to prove the spectral claims we make here.

\subsection{Extension of Bauer-Fike Theorem}
Tailored to our needs, we use the following proposition from \cite{BoLeMa15}: 
\begin{proposition}[Special case of Proposition $8$ in \cite{BoLeMa15}]
\label{prop::8}
Let $\ell =  C \log_\rho n$, with $C > 0$. Let $A \in M_n ( \mathbb{R})$,   such that for some  vectors $x_1 = x_{\ell,1}, y_1=y_{\ell,1}, x_2=x_{\ell,2}, y_2= y_{\ell,2} \in \mathbb{R}$, some matrix $R_\ell \in M_n(\mathbb{R})$, and some non-zero constants $\rho > \mu_2$ with $\mu^2_2 > \rho$, 
\be
A^{\ell} = \rho^{\ell} x_{1} y_{1}^* + \mu_2^{\ell} x_{2} y_{2}^* +  R_\ell.
\label{eq::BauerFike}
\ee
Assume there exist $c_0, c_1 > 0$ such that for all $i  \in \{1,2\}$, $\langle y_{i} , x_{i} \rangle \geq c_0$, $\|x_{i} \| \|y_{i} \| \leq c_1$. Assume further that $\langle x_{1} , y_{2} \rangle = \langle x_{2} , y_{1} \rangle = \langle x_1 , x_2 \rangle = \langle y_1 , y_2 \rangle = 0$ and for some $c > 0$
$$\| R_\ell \| <  \rho^{\ell / 2} \log^c(n). $$
Let $(\lambda_i)_{1 \leq i \leq n}$, be the eigenvalues of $A$ with $|\lambda_n| \leq \ldots \leq | \lambda_1|$. Then, 
$$
| \lambda_1 - \rho | = o(1), | \lambda_2 - \mu_2 | = o(1),  \quad \hbox{ and, for $i \geq 3$,} \quad |\lambda_i|\leq \sqrt{\rho} + o(1).
$$
Further, there exist unit eigenvectors $\psi_1, \psi_2$ of $A$ with eigenvalues $\lambda_1$, respectively  $\lambda_2$ such that 
$$
|| \psi_i - \frac{x_{i}}{\|x_{i}\|} || = o(1).
$$
\end{proposition}
\begin{proof}
This is a special case of Proposition $8$ in \cite{BoLeMa15}. In the notation of the latter, we have
$\ell' = \ell -2$, $\theta_1 = \rho$, $\theta_2 = \mu_2$, $\theta=\mu_2$, $\gamma \geq \frac{a+b}{|a-b|}>1$. Further $ \frac{c_0 (c_0 \gamma^k -c_1)_+}{4 c_1} \wedge \frac{c_0^2 }{2(\ell\vee \ell') c_1 }  \sim \frac{1}{\log_\rho n},$ and thus
$$\| R_\ell \| \leq  \log^c(n) \lr{ \frac{\sqrt \rho}{|\mu_2|} }^{\ell} |\mu_2|^\ell =o(1) \frac{1}{\log_\rho n} |\theta|^{\ell}. $$
\end{proof}

To prove the case $\mu_2^2 > \rho$ of Theorem \ref{th::4}, we thus need to find candidate vectors $x_1,x_2,y_1$ and $y_2$ that meet the conditions in Proposition \ref{prop::8} and further verify that the remainder $R_\ell$ has small norm. Note that the last condition is true whenever $\| B^{\ell} x \| \leq \rho^{\ell / 2} \log^c(n)$ for all normalized $x$ in $\text{span} \{y_1, y_2 \}^{\perp}.$

To address the case $\mu_2^2 \leq \rho$ of Theorem \ref{th::4}, we appeal to Proposition $7$ in \cite{BoLeMa15}, which is very similar in spirit to Proposition \ref{prop::8}.

\section{Proof of Theorem \ref{th::4}}
\label{sec::main}
\subsection{The case $\mu_2^2 > \rho$.}
We start with the case $\mu_2^2 > \rho$. We decompose, for some vectors $x_{1}, y_{1},x_{2}$ and $y_{2}$ and matrix $R_\ell$,
$$
B^{\ell} = \rho^{\ell} x_{1} y_{1}^* + \mu_2^{\ell} x_{2} y_{2}^* +  R_\ell,
$$
and we show that the assumptions of Proposition \ref{prop::8} are met.

Let $\ell$ be as in Theorem \ref{th::4} and recall $\chi_k$ and $\zeta_k$ from \eqref{def::chi_k} and \eqref{def::zeta_k}. For ease of notation, we introduce for $k \in \{1,2\}$, 
\be 
\varphi_k = \frac{B^\ell \chi_k}{\| B^\ell \chi_k \|}, \quad \text{ and } \theta_k = \|B^\ell \check \varphi_k\|.
\ee
Then, $\zeta_k = \frac{B^\ell \check \varphi_k}{\theta_k}.$ 

To prove the main theorem, we need the following two propositions. The proofs are deferred to Section \ref{sec::Prop19} and \ref{ssec::Prop20}. The material in Section \ref{sec::Prop19} builds on ingredients from Sections \ref{sec::coupling} - \ref{sec::law_large_numbers}, where we assume that $\mu_2^2 > \rho$, unless stated otherwise. 

\begin{proposition}[\DCA{Proposition $19$}] 
\label{prop::19} Assume that $\mu_2^2 > \rho$. Let $\ell = C \log_\rho n $ with $0 < C < C_{\text{min}}$. For some $b,c>0$, with high probability, 
\begin{enumerate}[(i)]
\item $b | \mu^{\ell}_k | \leq \theta_k  \leq   c | \mu^{\ell}_k |$  if $k \in \{1,2\}$,
\item  $\text{sign}( \mu_k ^{\ell}) \langle \zeta_k , \check \varphi_k\rangle   \geq  b $ if $k  \in \{1,2\}$,
\item $ \left| \langle  \varphi_1 ,   \varphi_2 \rangle \right|  \leq (\log n)^{3}  n^{C -\lr{\frac{\gamma}{2} \wedge \frac{1}{40}}}$, 
\item    $\left| \langle \zeta_j , \check  \varphi_k \rangle \right|   \leq (\log n)^{3}  n^{\frac{3}{2} C -\lr{\frac{\gamma}{2} \wedge \frac{1}{40}}}$   if $k \ne j \in \{1,2\}$.
\item  $\left| \langle \zeta_1 , \zeta_2 \rangle \right|  \leq  (\log n)^{8}  n^{2C -\lr{\frac{\gamma}{2} \wedge \frac{1}{40}}}$.
\end{enumerate}
\end{proposition}

Put $H = \text{span} \{ \check \varphi_1 , \check  \varphi_2 \}$, then
\begin{proposition}[\DCA{Proposition $20$}]
\label{prop::20}
Let $\ell = C \log_\rho n $ with $0 < C < C_{\text{min}}$. For some $c >0$, with high probability, 
\be 
\sup_{ x  \in H^\perp, \|x\| = 1} \| B^{\ell} x \| \leq (\log n)^c \rho^{\ell/2}. 
\ee
\end{proposition}

Put $\bar  \varphi_1 = \check \varphi_1,$ and $\bar  \varphi_2 = \frac{\check \varphi_2 - \langle  \check \varphi_1, \check \varphi_2 \rangle \check \varphi_1}{|| \check \varphi_2 - \langle  \check \varphi_1, \check \varphi_2 \rangle \check \varphi_1 ||},$ then
$\bar \varphi_1$ and $\bar \varphi_2$ are orthonormal and $||\bar \varphi_2 - \check \varphi_2|| = o(\rho^{-\ell/2}),$ due to Proposition \ref{prop::19} (iii). 

Let $\bar \zeta_1$ be the normalized orthogonal projection of $\zeta_1$ on $\text{span} \{\bar \varphi_2 \}^{\perp}$. 
Similarly, let  $\bar \zeta_2$ be the normalized orthogonal projection of $\zeta_2$ on $\text{span} \{\bar \zeta_1, \bar \varphi_1 \}^{\perp}$. 

Then $\langle \bar \zeta_1, \bar \zeta_2 \rangle = 0$ and for $i = 1,2$, $||\bar \zeta_i -   \zeta_i  || = o(\rho^{-\ell/2}),$
as follows from Proposition \ref{prop::19} $(iv)$ and $(v)$.

We set 
\[ D = \theta_1  \bar \zeta_1 \bar   \varphi_1^* +  \theta_2  \bar \zeta_2 \bar   \varphi_2^* = \rho^\ell  \lr{\frac{\theta_1}{\rho^\ell}  \bar \zeta_1} \bar   \varphi_1^* +  \mu_2^\ell  \lr{\frac{\theta_2}{\mu_2^\ell}  \bar \zeta_2} \bar   \varphi_2^*.  \]
Note that,
$$
\| B^{\ell} \bar \varphi_1 \| = \theta_1 = O (   \rho^{\ell}  ),
$$
and
$$
\| B^{\ell} \bar \varphi_2 \| = \| B^{\ell} \lr{ (1+o(1)) \check \varphi_2 + o(1) \bar \varphi_1} \|  = O (   \rho^{\ell}  ).
$$
As a consequence, from Proposition \ref{prop::20},
$$
\| B^{\ell} \| = O (  \rho^{\ell }  ).
$$
Since $D \bar\varphi_i =  B^{\ell} \check \varphi_i +  \theta_i(  \bar \zeta_i  - \zeta_i),$
\[ \| B^{\ell} \bar \varphi_i - D \bar \varphi_i \| \leq    \| B^{\ell} \| \| \bar \varphi_i -  \check \varphi_i \|   + \theta_i  \| \bar \zeta_i - \zeta_i \| = \bigO\lr{\rho^{\ell/2}}. \]

Let $P$ be the orthogonal projection on $H = \text{span} \{ \bar \varphi_1 , \bar \varphi_2 \} =  \text{span} \{ \check \varphi_1 , \check  \varphi_2 \}$,
then $\| B^{\ell}P - D \| = \bigO \lr{ \rho^{\ell/2} }$.

Put $R_\ell = B^\ell - D$. Write for $y \in \mathbb {R}^{\vec E}$ with unit norm, $y = h + h^{\perp},$ with $h \in H$ and $h^{\perp} \in H^{\perp}$, then
\BA
\| R_\ell y \| &= \| B^{\ell} h^{\perp} + (  B^{\ell} -D)  h  \| \\
&\leq \sup_{ x  \in H ^\perp, \|x\| = 1} \| B^{\ell} x \|  + \| B^{\ell} P  - D   \| \\
&= \bigO \lr{ \log^c(n) \rho^{\ell/2}},
\EA
as follows from Proposition \ref{prop::20}.

We finish by applying Proposition \ref{prop::8} with 
$x_1 = \frac{\theta_1}{\rho^\ell}  \bar \zeta_1, y_1 = \bar \varphi_1, x_2 = \frac{\theta_2}{\mu_2^\ell}  \bar \zeta_2,$ and, $y_2 = \bar \varphi_2$.

\subsection{The case $\mu_2^2 \leq \rho$.}
	In case $\mu_2^2 \leq \rho$, Proposition \ref{prop::19} $(i)$ and $(ii)$ continue to hold for $k=1$. Further, Proposition \ref{prop::19} $(iii)$ as well as Proposition \ref{prop::20} continue to hold. We need however the following bound for $k=2$:
\begin{proposition}\label{prop::19::complement}
Assume that $\mu_2^2 \leq \rho$. Let $\ell = C \log_\rho n $ with $0 < C < C_{\text{min}}$. For some $c>0$, with high probability, 
$$ \theta_2 \leq (\log n)^c \rho^{\ell/2}.$$
\end{proposition}
Using this proposition and $||\bar \varphi_2 - \check \varphi_2|| = o(\rho^{-\ell/2}),$ we get
$$ \| B^{\ell} \bar \varphi_2 \|  \leq (\log n)^{c+1} \rho^{\ell/2} . $$
It remains to apply Proposition $7$ from \cite{BoLeMa15}.

\section{Poisson-mixture two-type branching processes}
\label{sec::branching}
The proofs of the statements in this section are deferred to Appendix \ref{App::branching}.

\subsection{A theorem of Kesten and Stigum}
We consider the following branching process starting with a single particle, the root $o$, having spin $\sigma_o \in \spm$ and weight $\phi_o \in [\PHImin, \PHImax]$ (which we often take random). The root is replaced in generation $1$ by $\Pois{\frac{a}{2} \PHI \phi_o}$ particles of spin $\sigma_o$ and $\Pois{\frac{b}{2} \PHI \phi_o}$ particles of spin  $-\sigma_o$. Further, the weights of those particles are i.i.d. distributed following law $\NUstar$, the size-biased version of $\nu$, defined for $x \in [\PHImin, \PHImax]$ by
\be
\label{eq::size_biased}
\NUstar([0,x]) = \frac{1}{\PHI} \int_{\PHImin}^x y \mathrm{d} \nu(y).
\ee
 For generation $t \geq 1$, a particle with spin $\sigma$ and weight $\phi^*$ is replaced in the next generation by $\Pois{\frac{a}{2} \PHI \phi^*}$ particles with the same spin and $\Pois{\frac{b}{2} \PHI \phi^*}$ particles of the opposite sign. Again, the weights of the particles in generation $t+1$ follow in an i.i.d. fashion the law $\NUstar$. The offspring-size of an individual is thus a \textbf{Poisson-mixture}.

We use the notation $Z_t=
\begin{pmatrix} 
Z_t(+)  \\
Z_t(-)  
\end{pmatrix}$ for the population at generation $t \geq 1$, where $Z_t(\pm)$ is the number of type $\pm$ particles in generation $t$. We let  $(\cF_t)_{t\geq 1}$ denote the natural filtration associated to $(Z_t)_{t\geq 1}$.

We associate two matrices to the branching process, namely $M$ defined in \eqref{def::M},
and, for a root with weight $\phi_o$,
\be M_{\phi_o} = \frac{\PHI \phi_o}{\PHItwo} M. \ee 
 Then, $M$ is the \emph{transition} matrix for generations $t \geq 1$ and later:
\be
\E{Z_{t+1}| Z_t} = M Z_t, \quad \text{for all } t \geq 1,
\label{eq::Trans_normal}
\ee 
and $M_{\phi_o}$ describes the transition from the root to the first generation:
\be
\E{Z_{1}| Z_0, \phi_o} = M_{\phi_o} Z_0,
\label{eq::Trans_root}
\ee 
where, by assumption $Z_0 = \begin{pmatrix} 
\IND{\sigma_o = +}  \\
\IND{\sigma_o = -}  
\end{pmatrix}$.
Note that the difference between the root and later generations stems from the fact that the root's weight is deterministic in the conditional expectation, whereas the weight of a particle in any later generation has expectation $\frac{\PHItwo}{\PHI}$.

 Recall from \eqref{eq::EV_M} that $\EF_k$ ($k=1,2$) are the left-eigenvectors of $M$ associated to eigenvalues $\mu_k$:
\be \EF_k^* M = \mu_k \EF_k^*, \quad k =1,2. \ee
Note that $M_{\phi_o}$ has the same left-eigenvectors as $M$, while the corresponding eigenvalues are given by
\be \mu_{k,\phi_o} = \frac{\PHI \phi_o }{\PHItwo} \mu_k , \quad k =1,2.  \ee

Theorem \ref{th::21} shows that a Kesten-Stigum theorem applies to the "classical" branching process obtained after restricting the above process to generations $1$ and later. Corollary \ref{CO::21}, then, joins this classical branching process to the transition from the root to generation $1$.

We further consider the vector $\Psi_t = (\Psi_t(+), \Psi_t(-))$, containing sums of the weights, 
\be
\label{eq::Psi_t}
\Psi_t(\pm) = \sum_{u \in Y_t} \IND{\sigma_u = \pm} \phi_u,
\ee
where $Y_t$ is the set of particles at distance $t$ from the root, and where $\phi_u$ and $\sigma_u$ denote the weight respectively spin of a particle $u$. Note that $\Psi_t = Z_t$ in case of unit weights. 

The martingale Theorem \ref{th::21_other_martingale} is not present in \cite{BoLeMa15}. We need it to bound the variance of the cross-generational functional defined in Section \ref{ssec::cross}.

\begin{theorem}[\DCA{Theorem $21$}]
\label{th::21}
Assume that $\mu_2^2 > \rho$. Put $\cF_t = \{ Z_s \}_{s \leq t}$. For any $k = 1,2$,
$$
\lr{X_k(t) :=  \frac{\langle \EF_{k} , Z _t \rangle}{\mu^{t-1}_k} - \langle  \EF_k, Z_1  \rangle}_{t \geq 1},
$$
is an $\cF_t$-martingale converging a.s. and in $L^2$ such that for some $C > 0$ and all $t \geq 1$, $\E{ X_k (t) } = 0$ and $\E {  X_k^2 (t) | Z_1 } \leq C \| Z_1 \|_1$. 
\end{theorem}
\begin{corollary}
\label{CO::21}
Assume that $\mu_2^2 > \rho$. For $k =1,2$, with the weight $\phi_o = \psi_o$ of the root fixed, the sequence of random variables $(Y_{k,\psi_o}(t))_{t \geq 1} = \lr{\frac{\langle \EF_{k} , Z _t \rangle}{\mu^{t-1}_k \mu_{k,\psi_o}} }_{t \geq 1} $ converges \emph{almost surely} and in $L^2$ to a random variable $Y_{k,\psi_o}(\infty)$ with $\E{Y_{k,\psi_o}(\infty) | \sigma_o} = \EF_k(\sigma_o).$ Further, the $L^2$-convergence takes place uniformly over all $\psi_o$.
\end{corollary}
\begin{theorem}
\label{th::21_other_martingale}
Assume that $\mu_2^2 > \rho$. Put $\cG_t = \{ \Psi_s \}_{s \leq t}$. For any $k = 1,2$,
$$
\lr{X_k(t) :=  \frac{\langle \EF_{k} , \Psi_t \rangle}{\mu^{t-1}_k} - \langle  \EF_k, \Psi_1  \rangle}_{t \geq 1},
$$
is an $\cG_t$-martingale converging a.s. and in $L^2$ such that for some $C > 0$ and all $t \geq 1$, $\E{ X_k (t) } = 0$ and $\E {  X_k^2 (t) | Z_1 } \leq C \| Z_1 \|_1$. 
\end{theorem}

\subsection{Quantitative version of the Kesten-Stigum theorem}
We now quantify the growth of the population size. The latter is defined as
\[ S_t = \| Z_t \|_1, \quad t \geq 0, \]
i.e., the number of individuals in generation $t \geq 0$. Given $S_t$, for $t \geq 1$ we have
\be S_{t+1} = \text{Poi} \lr{ 	\s{l=1}{S_t} X_t^{(l)} }, \label{eq::S_k_plus_1} \ee
where $\lr{X_t^{(l)}}_l$ are i.i.d. copies of $\frac{a+b}{2} \PHI \phi^*$, where $\phi^*$ follows law $\NUstar$.

 Note that in the \emph{ordinary} Stochastic Block Model (i.e., when all vertices have unit weight), the argument of the Poisson random variables in \eqref{eq::S_k_plus_1} is deterministic, contrary to the general case under consideration here.  Using \eqref{eq::Trans_normal} recursively in conjunction with \eqref{eq::Trans_root}, it follows that
\[ \E{S_t | \phi_o} = \frac{ \PHI \phi_o }{\PHItwo} \rho^t, \quad \forall t \geq 1. \]

In the following lemma we show that deviations from this average are small. In fact, there exists a constant $C$ such that for each $t \geq 0$, $S_t$ is asymptotically stochastically dominated by an Exponential random variable with mean $C \rho^t$. An important ingredient in the proof below is Hoeffding's inequality, which we use to derive a concentration result for the parameter of the Poisson variable in \eqref{eq::S_k_plus_1}.

\begin{lemma}[\DCA{Lemma $23$}]
Assume $S_0 =1$.  There exist $c, c'>0$ such that for all $s \geq 0$, 
$$
\P{ \forall k\ge 1, S_{k}  \leq  s \rho^k  }  \geq 1 -  c' e^{ - c s}. 
$$
\label{lm::23}
\end{lemma}

From Theorem \ref{th::21} and Corollary \ref{CO::21}, we know that the different components (expressed in the basis of eigenvectors of $M$) grow exponentially with rate $\rho$, respectively $\mu_2$. We now quantify the error. Recall $\Psi_t$ from \eqref{eq::Psi_t}.

\subsubsection{The case $\mu_2^2 > \rho$.}
\begin{theorem}[\DCA{Theorem $24$}]
\label{th::24}
Assume that $\mu_2^2 > \rho$. Let $\beta >0$, $Z_0 = \delta_x$ and $\phi_o = \psi_o$ 
be fixed.
 There exists $C  = C(x,\beta) > 0$  such that with probability at least $1 - n^{-\beta}$, for all $k \in \{1,2\}$, all $0 \leq s < t \leq C_{\text{min}} \log(n)$,  with $0 \leq s < t $, 
$$
|  \langle \EF_{k} , Z _s \rangle -  \mu^{s-t}_{k} \langle  \EF_k, Z_t  \rangle | \leq C  (s +1) \rho ^{s / 2} ( \log n )^{3/2}, 
$$ 
and,
$$
|  \langle \EF_{k} , \Psi_s \rangle -  \mu^{s-t}_{k} \langle  \EF_k, \Psi_t  \rangle | \leq C  \rho ^{s / 2} ( \log n )^{5/2}. 
$$
\end{theorem}

\subsubsection{The case $\mu_2^2 \leq \rho$.}

\begin{theorem}
\label{th::24_special}
Assume that $\mu_2^2 \leq \rho$. Let $\beta >0$, $Z_0 = \delta_x$ and $\phi_o = \psi_o$ 
be fixed.
 There exists $C  = C(x,\beta) > 0$  such that with probability at least $1 - n^{-\beta}$, for all $t \geq 1$,
$$
|  \langle \EF_{2} , \Psi_t \rangle | \leq C t^2 \rho ^{t / 2} ( \log n )^{2}, 
$$
and,
$$
  \E{ | \langle \EF_{2} , \Psi_t \rangle|^2} \leq C t^3 \rho ^{t }. 
$$
\end{theorem}

\subsection{$B^{\ell} B^{*\ell} \check \chi_k$ on trees: a cross generation functional}
\label{ssec::cross}
Recall our claim that $B^{\ell} B^{*\ell} \check \chi_k$ are asymptotically aligned with the eigenvectors of $B$. In the DC-SBM, the local-neighbourhood of a vertex has with high probability a tree-like structure described by the branching process above. In this section we analyse $B^{\ell} B^{*\ell} \check \chi_k$ on trees.

To this end we define a cross-generational functional slightly different from its analogue in \cite{BoLeMa15} due to the presence of weights:
\begin{equation}\label{eq::Q}
Q_{k,\ell} = \sum_{(u_0,\ldots,u_{2\ell+1})\in{\mathcal P}_{2\ell+1}} \EF_k ( \sigma (u_{2 \ell +1}) ) \phi_{u_{2 \ell +1}} ,
\end{equation}
where ${\mathcal P}_{2\ell+1}$ is the set of paths $(u_0,\ldots,u_{2\ell+1})$ (of length $2\ell + 1$)  in the tree starting from $u_0 = o$
with both $(u_0, \ldots,
u_{\ell})$ and $(u_{\ell}, \ldots, u_{2 \ell +1})$  non-backtracking and 
$u_{\ell -1}  = u_{\ell +1}$. Note that these paths thus make a back-track exactly once at step $\ell +1$. 

 Explicitly, we have 
\be Q_{1,\ell} = \sum_{(u_0,\ldots,u_{2\ell+1})\in{\mathcal P}_{2\ell+1}} \frac{1}{\sqrt{2}} \phi_{u_{2 \ell +1}}, \label{eq::Q1_l} \ee 
 and, 
\be Q_{2,\ell} = \sum_{(u_0,\ldots,u_{2\ell+1})\in{\mathcal P}_{2\ell+1}} \frac{1}{\sqrt{2}} \sigma (u_{2 \ell +1}) \phi_{u_{2 \ell +1}}.   \label{eq::Q2_l} \ee

Consider a tree $\cT'$ and a leaf $e_1$ on it that has unique neighbour, say, $o$. Then, if $e$ is the oriented edges from $e_1$ to $o$ and if $B_{\cT'}$ denotes the non-backtracking matrix defined on $\cT'$, 
\be \lr{ B_{\cT'}^{\ell} B_{\cT'}^{*\ell} \check \chi_k }(e) =  Q_{k,\ell} + \EF_k(\sigma(e_1)) \phi_{e_1} \|Z_\ell\|_1, \label{eq::Q_on_tree} \ee
where $Q_{k,\ell}$ and $Z_\ell$ are defined on the tree $\cT$ with root $o$ obtained after removing vertex $e_1$ from $\cT'$.

In the sequel we analyse $Q_{k,\ell}$ on the branching process defined above, starting with a single particle, the root $o$. Let $V$ indicate the particles of the random tree. Denote the spin of a particle $v \in V$ by $\sigma_v \in \spm$ and its weight by $\phi_v \in S$. 

 For $t \geq 0$, let
$Y^v_t$ denote the set of particles, \emph{including} their spins and weights, of generation $t$ from $v$ in  the subtree of particles with common ancestor $v \in V$.
Let $Z^v_t  = ( Z^{v,+}_t,  Z^{v,-}_t)$ denote the number of $\pm$ vertices in generation $t$; i.e., $ Z^{v,\pm}_t = \sum_{u \in Y^v_t} \IND{  \sigma(u) = \pm}$. Finally, let
$\Psi^{v}_t =( \Psi^{v,+}_t, \Psi^{v,-}_t),$ with $\Psi^{v,\pm}_t = \sum_{u \in Y^v_t} \IND{  \sigma(u) = \pm} \phi_u$.

 We  rewrite $Q_{k,\ell}$ into a more manageable form: First observe that every path in ${\mathcal P}_{2\ell+1}$, after reaching $u_{\ell + 1}$, climbs back to a depth $t$ from which it then again moves down the tree (that is, in the direction away from the root). Let us call the vertex at level $t$ (to which the path climbs back before descending again) $u$. Then, (if $t \neq 0$) there are two children of $u$, say $v$ and $w$ such that $w$ lies on the path between $u$ and $u_{\ell +1}$ and $v$ is in between $u$ and $u_{2\ell + 1}$. For such fixed $v$ and $w$ in $Y_1^u$, only the children $u_{2\ell + 1} \in Y_t^v$ determine the contribution of a path to \eqref{eq::Q}, regardless of the choice of $u_{\ell + 1} \in Y^w_{\ell-t-1}$. Hence, for such fixed $u$ and $v,w \in Y_1^u$ and $u_{2\ell + 1}$, there are 
$|Y^w_{\ell-t-1}| = S^w_{\ell-t-1}$ paths giving the same contribution to \eqref{eq::Q}:
\begin{equation}\label{eq:defQkl}
Q_{k,\ell} =  \sum_{t=0} ^{\ell -1} \sum_{ u \in Y^o_{t}} L_{k,\ell}^u,
\end{equation}
where, for $|u|=t\geq 0$,
\be
L^u_{k,\ell} = \sum_{w\in Y^u_1} S^w_{\ell-t-1}\left( \sum_{v\in
    Y^u_1\backslash \{w\}} \langle \EF_k, \Psi^v_t\rangle \right).
\label{eq::L_k_ell}
\ee

The following theorem is an extension of Theorem $25$ in \cite{BoLeMa15}. The important observation is that, again, for $Z_0 = \delta_{\tau}$ fixed, $\lr{Q_{2,\ell}/ \mu_2 ^{2 \ell}}_\ell$ converges to a random variable with mean a constant times $\tau$, that is, the spin of the root. Its proof uses both martingale theorems stated above. We use the second martingale statement, which is not present in the ordinary SBM, to bound the variance of $Q_{k,\ell}$: 

\begin{theorem}[\DCA{Theorem $25$}]
\label{th::25}
Assume that $\mu_2^2 > \rho$. Let $Z_0 = \delta_x$ and $\phi_o = \psi_o$ be fixed.  For $k \in \{1,2\}$, $ \lr{Q_{k,\ell}/ \mu_k ^{2 \ell}}_\ell$ converges in $L^2$ as $\ell$ tends to infinity to a random variable with mean $ \frac{\PHIthree}{\PHItwo} \frac{\rho}{\mu_k^2 - \rho} \mu_{k,\psi_o}  \EF_k (x)  $. Further, the $L^2$-convergence takes place uniformly for all $\psi_o$. 
\end{theorem}

\subsubsection{The case $\mu_2^2 \leq \rho$.}
\begin{theorem}
\label{th::25_special}
Assume that $\mu_2^2 \leq \rho$. Let $Z_0 = \delta_x$ and $\phi_o = \psi_o$ be fixed. There exists a constant $C$ such that $\E{Q_{2,\ell}^2} \leq C \rho^{2\ell} \ell^5$.
\end{theorem}

\subsection{Orthogonality: Decorrelation in branching process}
Again, as in \cite{BoLeMa15}, $Q_{1,\ell}$ and $Q_{2,\ell}$ are uncorrelated when defined on the branching process above. The proof presented here is simpler than the corresponding one in \cite{BoLeMa15} and uses that for the two communities-case, $Q_{1,\ell}$ and $Q_{2,\ell}$ are \emph{explicitly} known.

The orthogonality of the candidate eigenvectors (i.e., $(iii) - (v)$ in Proposition \ref{prop::19}) follows from this fact, see Proposition \ref{prop::37} $(ii),(iii)$ and Proposition \ref{prop::38} $(ii)$ below.

\begin{theorem}[\DCA{28}]
\label{th::28}
Assume that the spin $\sigma_o$ of the root is drawn uniformly from $\spm$. Then for any $\ell \geq 0$, 
$$
\E{Q_{1,\ell}Q_{2,\ell}| \cT} = 0.
$$
\end{theorem}

\section{Coupling of local neighbourhood }
\label{sec::coupling}
The proofs of the statements in this section are deferred to Appendix \ref{App::coupling}.
\subsection{Coupling}
Here we establish the connection between neighbourhoods in the DC-SBM and the branching process in Section \ref{sec::branching}. We established this coupling in an earlier paper \cite{GuLeMa15}  using an exploration process that we repeat below. Compared to the ordinary SBM, vertices are now weighted, so that two facts need to be verified: At each step of the exploration process, unexplored vertices have a weight drawn from a distribution close in total variation distance to $\nu$. Detected vertices on their turn follow a law close to $\NUstar$. 

We distinguish between two different concepts of neighbourhood: the classical neighbourhood that is rooted at a vertex and another neighbourhood that starts with an edge. For the latter, we need the following concept of \emph{oriented} distance $\vec d$, which for $e,f \in \vec E(V)$ is defined as
$$
\vec d ( e, f) = \min_{\gamma} \ell( \gamma )
$$
where  the minimum is taken over all self-avoiding paths $\gamma = (\gamma_0, \gamma_1, \cdots , \gamma_{\ell+1} )$  in $G$ such that $(\gamma_0, \gamma_1) = e$, $(\gamma_{\ell} , \gamma_{\ell+1} ) = f$ and for all $1 \leq k \leq \ell+1$, $\{ \gamma_k , \gamma_{k+1} \} \in E$.
and where for such a path $\gamma$, $\ell(\gamma) = \ell$. Note that $\vec d (e, f) = \vec d (f^{-1}, e^{-1})$, i.e., $\vec d$ is not symmetric.

We introduce the vector $Y_t(e) = (Y_t(e) (i))_{i \in \spm}$ where, for $i \in \spm$,
\begin{equation}\label{eq:defYti}
Y_t (e)(i) = \left| \left\{ f \in \vec E : \vec d ( e, f) = t  , \sigma(f_2) =  i \right\} \right| ,
\end{equation}
 we denote the number of vertices at oriented distance $t$ from $e$ by
$$
S_t (e) = \|Y_t(e) \|_1 = \left| \left\{ f \in \vec E : \vec d ( e, f) = t \right\} \right|,
$$
and we define vector $\Psi_t(e) = (\Psi_t(e) (i))_{i \in \spm}$ where, for $i \in \spm$,
\begin{equation}\label{eq:def_Psi_t_i}
\Psi_t (e)(i) = \sum_{ f \in \vec E : \vec d ( e, f) = t} \IND{\sigma(f_2) =  i}  \phi_{f_2}.
\end{equation}
We denote the classical neighbourhood of radius $r$ rooted at vertex $v$ by $(G,v)_r$ and the neighbourhood around oriented edge $e = (e_1,e_2)$ by $(G,e)_r$. With the definitions above, we then have, $(G,e)_r = (G',e_2)_r$, where $G'$ is the graph $G$ with edge $\{e_1,e_2\}$ removed. In particular,
$$
S_t (e) = S'_t(e_2),
$$
where $S'_t$ is $S_t$ defined on $G'$. 

The two branching processes that describe the neighbourhoods are almost identical, the only difference lies in the weight of the root: In the classical branching processes, the weight is drawn according to distribution $\nu$. In the branching process starting at an edge oriented towards, say, $o$, the root $o$ has weight governed by $\NUstar$. See Proposition \ref{prop::31} below.

As a corollary we obtain an analogue of Theorem \ref{th::24} for local neighbourhoods: the components of $\Psi_t(e)$ grow exponentially, see Corollary \ref{cor::32}.

We bound the growth of $S_t$ in Lemma \ref{lm::29}. We use a coupling argument to show that the weights of the unexplored vertices and selected vertices are stochastically dominated by variables following law $\nu$, respectively $\NUstar$. This argument is not needed in the ordinary SBM.

Following \cite{MoNeSl15}, we need to verify that certain problematic structures, namely \emph{tangles}, are excluded with high probability. 
We say that a graph $H$ is tangle-free if all its $\ell-$ neighbourhoods contain at most one cycle. If there is at least one  $\ell-$ neighbourhood in $H$ that contains more than one cycle, we call $H$ tangled. Note that in the sequel we shall often suppress the dependence on $\ell$ and simply call a graph tangle-free or tangled; the $\ell$ dependence is then tacitly assumed. 

Following standard arguments we establish in Lemma \ref{lm::30} that the graph is with high probability $\log(n)$-tangle free. 

We prepare by recalling the exploration process in \cite{GuLeMa15} starting at a vertex:

At time $m=0$, choose a vertex $\rho$ in $V(G)$, where $G$ is an instant of the DC-SBM. Initially, it is the only active vertex:  $\mathcal{A}(0) = \{\rho\}$. All other vertices are neutral at start: $\mathcal{U}(0) = V(G) \setminus \{\rho\}$. No vertex has been explored yet: $\mathcal{E}(0) = \emptyset$. 

At each time $m \geq 0$ we arbitrarily pick an active vertex $u$ in $\mathcal{A}(m)$ that has shortest distance to $\rho$, and explore all its edges in $\{ uv: v \in \mathcal{U}(m) \}$: if $uv \in E(G)$ for $v \in \mathcal{U}(m)$, then we set $v$ active in step $m+1$, otherwise it remains neutral. 

At the end of step $m$, we designate $u$ to be explorated. 

Thus,
\[ \calE(m+1) = \calE(m) \cup \{u\}, \]
\[ \calA(m+1) = \lr{ \calA(m) \setminus \{u\}  } \cup \lr{ \mathcal{N}(u) \cap  \calU(m) }, \]
and,
\[ \calU(m+1) = \calU(m) \setminus \mathcal{N}(u). \]

\begin{proposition}[\DCA{Proposition $31$}]
\label{prop::31}
Let $\ell = C \log_\rho (n)$, with $C < C_{\text{coupling}}$.
Let $\rho \in V$ and $e = (e_1,e_2) \in \vec E$. 
Let $(T,o)$ be the branching process with root $o$ defined in Section \ref{sec::branching}, where the root has spin $\sigma(v)$ and weight governed by $\nu$. Similarly, Let $(T',o)$ be that same branching process, when the root has spin $\sigma(e_2)$ and weight governed by $\NUstar$. 
Then, the total variation distance between the law of $(G,v)_\ell$ and $(T,o)_\ell$ goes to zero as $1 - n^{-\lr{ \frac{\gamma}{2} \wedge \frac{1}{40} }}$. The same is true for the difference between the law of $(G,e)_\ell$ and $(T',o)$. 
\label{prop::31}
\end{proposition} 
\begin{remark}
Note that with the event $(G,v)_\ell = (T,o)_\ell$, we mean that the graph and tree are equal, \textbf{including} their \textbf{spins and weights}. See \cite{GuLeMa15} for more details.
\end{remark}

\begin{corollary}[\DCA{Corollary $32$}]
\label{cor::32}
Assume $\mu_2^2 > \rho$. Let $\ell  = C \log_\rho n$ with $0 < C < C_{\text{coupling}}$. For $e \in \vec E(V)$, we define the event $\cE (e)$ that for all $0 \leq t <  \ell $ and $k \in \{1,2\}$: 
$|  \langle \EF_k , \Psi_t (e) \rangle - \mu_k^{t - \ell}  \langle \EF_k , \Psi_\ell (e) \rangle | \leq (\log n)^3 \rho^{t/2} $.
Then,  with high probability, the number of edges $e \in \vec E$ such that  $\cE(e)$ does not hold is at most $\log (n) \  n^{1  - (\frac{\gamma}{2} \wedge \frac{1}{40})}$.\end{corollary}
\begin{lemma}[\DCA{Lemma $29$}]
\label{lm::29}
There exist $c, c' >0$ such that for all $s \geq 0$ and for any $w \in [n] \cup \vec E(V)$,    
$$
\P{   \forall t \ge 0: S_{t} (w)   \leq  s \bar \rho_n^t    } \geq 1 -  c e^{ - c' s}. 
$$
Consequently, for any $p\geq 1$, there exists $c'' >0$ such that $$\E{ \max_{v \in [n] , t \geq 0} \lr{  \frac{ S_{t} (v) }{ \bar \rho^t_n} }^p } \leq c'' (\log n)^p.$$
 \end{lemma}

\begin{lemma}[\DCA{Lemma 30}]
\label{lm::30}
Let $\ell = C \log_{\rho} (n)$, with $0 < C < C_{\text{coupling}}.$
Then, w.h.p., at most $\rho^{2\ell} \log(n)$ vertices have a cycle in their $\ell$ - neighbourhood. Further, w.h.p., the graph is $\ell$ - tangle-free. 
\end{lemma} 

\subsection{Geometric growth}
Here we show that for $k \in \{1,2\}$, $\langle B^\ell \chi_k , \delta_e \rangle$ grows nearly geometrically in $t$ with rate $\mu_k$. 
Corollary \ref{cor::34} then establishes a bound for $r \leq \ell$ on $\sup_{\langle B^\ell \chi_k , x \rangle = 0, \| x\| = 1 } \| \langle B^r \chi_k , x \rangle \|$ crucial for the norm bounds in Section \ref{sec::norms}.

\begin{proposition}[\DCA{Proposition $33$}]
\label{prop::33} Assume $\mu_2^2 > \rho$.
Let $\ell = C \log_{\rho} (n)$, with $0 < C < C_{\text{coupling}} \wedge \lr{\frac{1}{2}- \lr{\frac{\gamma}{4} \wedge \frac{1}{80}}} = C_{\text{coupling}}$.
For $e \in \vec E(V)$, let $\vec E_{\ell}$ be the set of oriented edges such that either $(G,e_2)_{\ell}$ is not a tree or the event $\mathcal{E}(e)$ (defined  in Corollary \ref{cor::32})  does not hold. Then, w.h.p. for $k \in \{1,2\}$:
\begin{enumerate}[(i)]
\item 
$| \vec E_{\ell} | 
\ll  (\log n)^2 n^{1-\frac{\gamma}{2} \wedge \frac{1}{40}}$,
\item for all $e \in \vec E \backslash \vec E_{\ell}$, $0 \leq r \leq \ell$,  
\[
|\langle B^r \chi_k , \delta_e \rangle -  \mu_k ^{r-\ell} \langle B^\ell \chi_k , \delta_e \rangle |  \leq   (\log n )^4 \rho^{r/2},  
\]
\item
for all $e \in \vec E_{\ell}$, $0 \leq r \leq \ell$,
\[|\langle B^r \chi_k , \delta_e \rangle |   \leq   (\log n )^2  \rho^{r}.
\]
\end{enumerate} 
\end{proposition}

\begin{corollary}[\DCA{Corollary $34$}]
\label{cor::34}
Let $\ell = C \log_{\rho} (n)$, with $0 < C < C_{\text{coupling}} \wedge \lr{ 1-\frac{\gamma}{2} \wedge \frac{1}{40}} \wedge \lr{\frac{\gamma}{4} \wedge \frac{1}{80}} = C_{\text{coupling}}$.
W.h.p. for any $0 \leq r \leq \ell-1$ and $k \in \{1,2\}$:
$$
\sup_{\langle B^\ell \chi_k , x \rangle = 0, \| x\| = 1 } \| \langle B^r \chi_k , x \rangle \| \leq  (\log n)^5 n^{1/2} \rho^{r/2}.
$$
\end{corollary} 

\section{A weak law of large numbers for local functionals on the DC-SBM}
\label{sec::law_large_numbers}
The proofs of the statements in this section are deferred to Appendix \ref{App::law_large_numbers}.

Here we show that a weak law of large numbers applies for local functionals defined on \emph{weighted coloured} random graphs generated according to the DC-SBM. 

By a \emph{weighted  coloured} graph we mean a graph $G = (V,E)$ together with maps $\sigma: V \to \spm$ and $\phi: V \to [\PHImin,\PHImax]$. For $v \in V$, we identify $\sigma(v)$ as the spin of $v$ and $\phi(v)$ as its weight. We denote by $\cG^*$ the set of \emph{rooted weighted coloured} graphs. We denote an element of $\cG^*$ by $(G,o)$: $G = (V,E)$ is then a weighted coloured graph and $o \in V$ is some distinguished vertex.  A function $\tau: \cG^* \to \mathbb{R}$ is said to be $\ell-$ local if $\tau(G,o)$ depends only on $(G,o)_\ell$. 

To derive the claimed weak law when $G$ is drawn according to the DC-SBM, we prepare with a variance bound for $\sum_{v=1}^n    \tau ( G, v)$, see Proposition \ref{prop::35}. The bound follows from the law of total variance,
\[ \ba \Var \lr{\sum_{v=1}^n    \tau ( G, v)} &= \E{\Var\lr{ \left. \sum_{v=1}^n    \tau ( G, v)\right|\phi_1, \ldots, \phi_n}} \\ &\quad + \Var \lr{ \left. \E{\sum_{v=1}^n    \tau ( G, v) \right| \phi_1, \ldots, \phi_n} }, \ea \]
together with an application of Efron-Stein's inequality to both terms on the right. Note that $\E{\left. \sum_{v=1}^n    \tau ( G, v) \right| \phi_1, \ldots, \phi_n}$ is a constant in the \emph{ordinary} SBM, whereas here it needs a careful analysis. 

The sample average $\frac{1}{n} \sum_{v=1}^n    \tau ( G, v)$ concentrates then around $\E{\tau(T,o)}$, where $(T,o)$ is the branching process from Section \ref{sec::branching}, with root $o$ having spin 
drawn uniformly from $\spm$ and weight governed by $\nu$, see Proposition \ref{prop::36}. The coupling, and in particular the matching of the weights, plays an important role in its proof.

In the next section we apply the latter proposition to some specific functionals. 

\begin{proposition}[\DCA{Proposition $35$}]
\label{prop::35}
Let $G$ be drawn according to the DC-SBM. 
There exists $c  >0$ such that if $\tau, \varphi : \cG^* \to \mathbb{R}$ are $\ell$-local, $| \tau(G,o) | \leq  \varphi (G,o)$  and $\varphi$ is non-decreasing by the addition of edges, then 
$$
\emph{Var} \lr{  \sum_{v=1}^n    \tau ( G, v) } \leq c n   \rho^{2\ell } \lr{ \E{ \max_{v \in [n]} \varphi^4 (G,v)}}^{1/ 2}.
$$ 
\end{proposition}

\begin{proposition}[\DCA{Proposition $36$}]
\label{prop::36}
Let $G$ be drawn according to the DC-SBM. Let $(T,o)$ be the branching process from Section \ref{sec::branching}, with root $o$ having spin  
drawn uniformly from $\spm$ and weight governed by $\nu$.
Let 
$\ell = C \log_\rho(n)$, with $C <   C_{\text{coupling}}$. There exists $c  >0$ such that if $\tau, \varphi : \cG^* \to \mathbb{R}$ are $\ell$-local, $| \tau(G,o) | \leq  \varphi (G,o)$  and $\varphi$ is non-decreasing by the addition of edges, then 
\BA
&\E{ \left|\frac 1 n   \sum_{v=1}^n    \tau ( G, v)  - \E{ \tau ( T, o)}  \right|}   \\   
&\leq c_2 n^{-\lr{ \frac{\gamma}{2} \wedge \frac{1}{40} }} \lr{ \E{ \max_{v \in [n]}  \varphi^4 (G,v) }^{1/4} \vee \E{ \varphi^2 (T,o) }^{1/2} } + \bigO(n^{-\gamma})
\EA
\end{proposition}

\subsection{Application with some specific local functionals}
\label{ssec::app_local}
Here we consider $\langle B^\ell \chi_1, B^\ell \chi_2 \rangle, $  $\langle B^{2\ell} \chi_k, B^\ell \chi_j \rangle, $  and  $\langle B^\ell B^{* \ell} \chi_1, B^\ell B^{* \ell} \chi_2 \rangle, $ quantities occurring in Proposition \ref{prop::19}. 

Explicitly, $B^\ell \chi_k (e) = \sum_f B^\ell_{ef} \EF_k(\sigma(f_2))\phi_{f_2}, $ where we recall that $B^\ell_{ef}$ is the number of non-backtracking walks from $e$ to $f$. Now, if the oriented $\ell-$ neighbourhood of $e$ is a \emph{tree}, then $B^\ell \chi_k (e) = \langle \EF_k, \Psi_\ell(e) \rangle$. With this intuition in mind, we analyse likewise expressions in Proposition \ref{prop::37} below. 

Inspired by \eqref{eq::Q_on_tree}, which expresses $B^\ell B^{* \ell} \chi_k$ on \emph{trees} in terms of the operator $Q_{k,\ell}$, we extend the latter to an operator defined on general graphs. First, 
for  $e \in \vec E(V)$ and $t \geq 0$, set $\cY_t (e) = \{ f \in \vec E : \vec d (e,f) = t \}$. Then, for $k \in \{1,2\}$, we set
\begin{equation}\label{eq::P}
P_{k,\ell} (e) = \sum_{t = 0} ^{\ell -1} \sum_{f \in \cY_{t} (e) } L_k(f),
\end{equation}
with  $$L_k(f) = \sum_{(g,h) \in \cY_1 (f) \backslash \cY_t(e); g \ne h  } \langle \EF_k , \tilde \Psi_t(g)\rangle \tilde S_{\ell - t -1} (h),$$
where $\tilde \Psi_t(g)$, $\tilde S_{\ell - t -1}(h) = \|\tilde Y_{\ell - t-1}(h)\|_1$ are the variables $ \Psi_t(g)$, respectively $S_{\ell - t -1}(h)$, defined on the graph $G$ where all  edges in $(G,e_2)_t$ have been removed. Note that,  if $(G,e)_{2\ell}$ is a tree, then $\tilde \Psi_s(g) = \Psi_s(g)$  for $s \leq 2 \ell -t$.
Compare $P_{k,\ell}$ to $Q_{k,\ell}$ in \eqref{eq::Q} and $L_k(f)$ to $L_{k,\ell}û$ in \eqref{eq::L_k_ell}. 

Finally, define
\be
\label{eq::S_k_ell}
S_{k,\ell} (e) = S_{\ell}(e) \EF_k ( \sigma (e_1)) \phi_{e_1}.
\ee
We then have an extension of \eqref{eq::Q_on_tree}, when $(G,e_2)_{2\ell}$ is a tree:
\be
\label{eq::P_on_tree}
B^{\ell} B^{*\ell} \check \chi_k (e) = P_{k,\ell} (e) + S_{k,\ell} (e). 
\ee
We analyse \eqref{eq::P_on_tree} in Proposition \ref{prop::38} below.

 \subsubsection{The case $\mu_2^2 > \rho$.}

\begin{proposition}[\DCA{Proposition $37$}]
\label{prop::37}
Assume that $\mu_2^2 > \rho$. Let $\ell =  C \log_\rho n$ with $0 < C < C_{\text{coupling}}$. 
\begin{enumerate}[(i)]
\item \label{zz1} 
For any $ k\in \{1,2\}$, there exists $c'_k >0$ such that, in probability, 
$$
\frac 1 { n} \sum_{e \in \vec E}  \frac{\langle \EF_k , \Psi_{\ell} (e) \rangle^2}{ \mu_k ^{2 \ell}}  \to c'_k.
$$
\item \label{zz2} 
For any $ k\in \{1,2\}$, there exists $c''_k >0$ such that, in probability, 
$$
\frac 1 { n} \sum_{e \in \vec E}  \frac{\langle \EF_k , Y_{\ell} (e) \rangle^2}{ \mu_k ^{2 \ell}}  \to c''_k.
$$
\item \label{zz3} 
$$
\E{  \left|\frac 1 {n} \sum_{e \in \vec E} \langle \EF_1 , \Psi_{\ell} (e) \rangle  \langle \EF_2 , \Psi_{\ell} (e) \rangle \right| } \leq  (\log n)^{3}  n^{2C -\lr{\frac{\gamma}{2} \wedge \frac{1}{40}}} + n^{-\gamma}   .
$$
\item  \label{zz4} For any $ k \ne j \in \{1,2\} $, 
$$
\E{ \left| \frac 1 { n} \sum_{e \in \vec E} \langle \EF_k , \Psi_{2\ell} (e) \rangle    \langle \EF_j , \Psi_{\ell} (e) \rangle \right|}   \leq  (\log n)^{3}  n^{3C -\lr{\frac{\gamma}{2} \wedge \frac{1}{40}}} + n^{-\gamma}.
$$
\item  \label{zz5} For any $ k  \in \{1,2\}$, in probability
$$
\frac{1}{n} \sum_{e \in \vec E}   \frac{\langle \EF_k , \Psi_{2\ell} (e) \rangle    \langle \EF_k , \Psi_{\ell} (e) \rangle}{\mu_k^{3\ell}  } \to  c'''_k.
$$
\end{enumerate}
\end{proposition}

\begin{proposition}[\DCA{Proposition $38$}]
\label{prop::38}
Assume that $\mu_2^2 > \rho$. Let $\ell = C \log_\rho n$ with $C < C_{\text{coupling}}$.  
\begin{enumerate}[(i)]
\item \label{yy1} 
For any $ k\in \{1,2\}$, there exists $c''''_k >0$ such that in probability
$$
\frac 1 {n} \sum_{e \in \vec E}  \frac{P^2_{k,\ell} (e) }{ \mu_k ^{4 \ell}}  \to  c''''_k.
$$
\item 
$$
\E{ \left| \frac 1 {n} \sum_{e \in \vec E} ( P_{1,\ell} (e) + S_{1,\ell} (e) ) ( P_{2,\ell}(e) + S_{2,\ell} (e))  \right| }  
\leq  (\log n)^{8}  n^{4C -\lr{\frac{\gamma}{2} \wedge \frac{1}{40}}}
$$
\end{enumerate}
\end{proposition}

\subsubsection{The case $\mu_2^2 \leq \rho$.}
Most of the above claims continue to hold if $\mu_2^2 \leq \rho$. We treat the exceptions here.
\begin{proposition}
\label{prop::37_special}
Assume that $\mu_2^2 \leq \rho$. Let $\ell =  C \log_\rho n$ with $0 < C < C_{\text{coupling}}$. There exists some $c > 0$, such that w.h.p.,
\[ \frac{1}{n} \sum_{e \in \vec E} \frac{\langle \EF_2 , \Psi_\ell(e) \rangle^2}{\rho^\ell} \geq c. \]
\end{proposition}

\begin{proposition}
\label{prop::38_special}
Assume that $\mu_2^2 \leq \rho$. Let $\ell = C \log_\rho n$ with $C < C_{\text{coupling}}$.  
There exists $c > 0$ such that w.h.p.,
$$
\frac 1 {n} \sum_{e \in \vec E}  \frac{P^2_{2,\ell} (e) }{ \rho^{2 \ell} \log^5 (n)}  \leq c.
$$
\end{proposition}

\section{Proof op Propositions \ref{prop::19} and \ref{prop::19::complement} }
\label{sec::Prop19}
We introduce for $k \in \{1,2\}$ the vector $N_{k,\ell}$, defined on $e \in \vec E$ as
$$
N_{k,\ell}(e) = \langle \EF_k, \Psi_\ell(e) \rangle.
$$
If $(G,e_2)_\ell$ is a tree, then
$$
N_{k,\ell}(e) = \langle B^\ell \chi_k, \delta_e \rangle,
$$
and we have a similar expression for $B^{\ell} B^{* \ell} \check \chi_k $ in \eqref{eq::P_on_tree}. 
Now, at most $\rho^{2\ell} \log(n)$ vertices have a cycle in their $\ell$-neighbourhood (see Lemma \ref{lm::30}). Therefore:

\begin{lemma}[\DCA{Lemma $39$}]
\label{lm::39}
Let $\ell = C \log_\rho n $ with $0 < C < C_{\text{min}}$.
Then, w.h.p. 
$\| B^{\ell} \chi_k - N_{k,\ell} \| =  O \lr{  (\log n )^{5/2} \rho ^{2 \ell} } = o \lr{  \rho^{\ell/2} \sqrt n } $, $ \| B^{\ell} B^{* \ell} \check \chi_k - P_{k,\ell} - S_{k,\ell} \| =  O ( (\log n)^4 \rho^{4\ell} )  $ and $ \| B^{\ell} B^{* \ell} \check \chi_k - P_{k,\ell} \| = O ( \rho^{\ell} \sqrt n )$.
\end{lemma}
\begin{proof}
The proof of Lemma 39 in \cite{BoLeMa15} can be easily adapted to the current setting. The key idea is pointed out above. It thus remains to bound $|(B^{\ell} \chi_k - N_{k,\ell})(e)|$ and $|( B^{\ell} B^{* \ell} \check \chi_k - P_{k,\ell})(e)|$ on edges $e$ for which $(G,e_2)_\ell$ is not a tree. For this, use that with high probability the graph is $2\ell$-tangle free so that there are at most two non-backtracking paths between $e$ and any edge at distance $\ell$. 
\end{proof}
We can thus in our calculations replace $B^{\ell} \chi_k$ by $N_{k,\ell}$ and $B^{\ell} B^{* \ell} \check \chi_k$ by $P_{k,\ell}$. From Propositions \ref{prop::37} and \ref{prop::38}, Proposition \ref{prop::19} then follows:
\begin{proof}[Proof of Proposition \ref{prop::19}]
This proof follows the corresponding proof in \cite{BoLeMa15}. We give the key observations:
$(i)$ From Proposition \ref{prop::37} $(i)$, $\|N_{k,\ell}\| \sim \sqrt{n} \mu_k^\ell$ and from Proposition \ref{prop::38} $(i)$, $\|P_{k,\ell}\| \sim \sqrt{n} \mu_k^{2\ell}.$

$(ii)$ From Proposition \ref{prop::37} $(v)$, $| \langle N_{k,\ell}, N_{k,2\ell} \rangle | \sim n \mu_k^{3\ell}$.

$(iii)$ From Proposition \ref{prop::37} $(iii)$, $| \langle N_{1,\ell}, N_{2,\ell} \rangle | \sim (\log n)^{3}  n^{3C -\lr{\frac{\gamma}{2} \wedge \frac{1}{40}}}$.
 
$(iv)$ From Proposition \ref{prop::37} $(iv)$, $| \langle N_{k,2\ell}, N_{j,\ell} \rangle | \sim (\log n)^{3}  n^{4C -\lr{\frac{\gamma}{2} \wedge \frac{1}{40}}}$.

$(v)$ From Proposition \ref{prop::38} $(ii)$, $| \langle P_{1,\ell} + S_{1,\ell}, P_{2,\ell} + S_{2,\ell} \rangle | \sim (\log n)^{8}  n^{5C -\lr{\frac{\gamma}{2} \wedge \frac{1}{40}}}$.
\end{proof}

Proposition \ref{prop::19::complement} follows similarly from the case $\mu_2^2 \leq \rho$ treated in Section \ref{ssec::app_local}:

\begin{proof}[Proof of Proposition \ref{prop::19::complement}]
This follows from Propositions \ref{prop::37_special} and \ref{prop::38_special} in conjunction with Lemma \ref{lm::39}.
\end{proof}

\section{Norm of non-backtracking matrices}
\label{sec::norms}
The proofs of the statements in this section are deferred to Appendix \ref{App::norms}.

In this section the product over an empty set is defined to be one.

It is convenient to extend matrix $B$ and vector $\chi_k$ to the set of directed edges on the \emph{complete} graph,  $\vec E_K(V) = \{ ( u ,v ) : u \ne v   \in V\}$: For $e , f\in \vec E_K(V)$, $B_{ef}$ is then extended to 
\be
B_{ef} = A_e A_f 1_{e_2 = f_1} 1_{e_1 \ne f_2},
\ee
where $A$ is the adjacency matrix. For each $e\in \vec E_K(V)$ we set $\chi_k(e) = \EF_k(\sigma(e_2))\phi_{e_2}$. 

For integer $k \geq 1$, $e , f\in \vec E_K(V)$, we let $\Gamma^k_{ef}$
be the set of non-backtracking walks $ \gamma = (\gamma_0, \ldots, \gamma_{k} )$  of length $k$ from $(\gamma_0, \gamma_1) = e$ to $(\gamma_{k-1},\gamma_k) = f$ on the \emph{complete} graph with vertex set $V$. 

By induction it follows that
\be
(B^{k} )_{e f} = \sum_{\gamma \in \Gamma^{k+1} _{e f}} \prod_{s=0}^{k} A_{\gamma_{s} \gamma_{s+1}}.
\label{eq::B_k_ef}
\ee
Indeed, note that $\prod_{s=0}^{k} A_{\gamma_{s} \gamma_{s+1}}$ is one when $
\gamma$ is a path in $G$ and zero otherwise. 

To each walk $\gamma = (\gamma_0,\ldots, \gamma_k)$, we associate the graph $G(\gamma) = ( V(\gamma), E(\gamma) ) $, with the set of vertices 
$V(\gamma) = \{ \gamma_{i}, 0 \leq i \leq k \}$ and the set of edges 
$E (\gamma) = \{ \{ \gamma_{i} , \gamma_{i+1}\}, 0 \leq i \leq k-1 \}$.

From Lemma \ref{lm::30}, the graphs following the DC-SBM are tangle-free with high probability. Hence, it makes sense to consider the subset $F^{k+1}_{e f} \subset \Gamma^{k+1}_{ef}$  of tangle-free non-backtracking walks on the \emph{complete} graph. Indeed, if $G$ is tangle-free, we need only consider the tangle-free paths in the summation \eqref{eq::B_k_ef}:
\be
(B^{(k)} )_{e f} = \sum_{\gamma \in F^{k+1} _{e f}} \prod_{s=0}^{k} A_{\gamma_{s} \gamma_{s+1}},
\label{eq::B_k_ef_TF}
\ee
and $B^k = B^{(k)}$ for $1 \leq k \leq \ell$.
 
Define for $u \neq v$ the \emph{centred} random variable 
\be
\underline A_{u v}  = A_{uv} - \frac{\phi_u \phi_v}{n} W_{\sigma_u \sigma_v},
\label{eq::A_uv_centered}
\ee
where
\[ W = \left( \begin{array}{cc}
a & b \\ 
b & a \end{array} \right). \]
\textbf{Compare this to the SBM \emph{without} degree-corrections in Section $10.1$ of \cite{BoLeMa15}: $\phi_u = 1$ for all $u$ in the latter model.} 

Using $\underline A$ we shall attempt to center $B^k$ when the underlying graph $G$ is tangle-free through considering
\be
 \Delta^{(k)}_{ef} = \sum_{\gamma \in F^{k+1} _{e f}}   \prod_{s=0}^{k}  \underline A_{\gamma_{s} \gamma_{s+1}}.
\label{eq::Delta_ef}
\ee
Further, we set 
\be 
\Delta^{(0)}_{ef} = \IND{ e = f }\underline A_{e} \quad \hbox{ and } \quad B^{(0)}_{ef} = \IND{ e = f }A_e.
\label{eq::Delta_0}
\ee
To decompose \eqref{eq::B_k_ef_TF}, following a decomposition that appeared first in \cite{Ma14}, we use 
\[
\prod_{s=0}^\ell x_s = \prod_{s=0}^\ell y_s  + \sum_{t=0}^{\ell}\prod_{s=0}^{t-1} y_s  ( x_t - y_t) \prod_{s=t+1}^{\ell} x_s,
\]
with $x_s = A_{\gamma_{s} \gamma_{s+1}}$ and $y_s = \underline A_{\gamma_{s} \gamma_{s+1}}$ on a path $\gamma \in F^{k+1} _{e f}$:
\[
\prod_{s=0}^\ell A_{\gamma_{s} \gamma_{s+1}} = \prod_{s=0}^\ell \underline A_{\gamma_{s} \gamma_{s+1}}  + \sum_{t=0}^{\ell}\prod_{s=0}^{t-1} \underline A_{\gamma_{s} \gamma_{s+1}}  \lr{ \frac{\phi_{\gamma_t} \phi_{\gamma_{t+1}}}{n} W_{\sigma_{\gamma_{t}} \sigma_{\gamma_{t+1}}}} \prod_{s=t+1}^{\ell} A_{\gamma_{s} \gamma_{s+1}}.
\]
Summing over all $\gamma \in F^{\ell+1} _{e f}$ then gives
\be
\ba 
B^{(\ell)} _{e f} &= \sum_{\gamma \in F^{\ell+1} _{e f}}     \prod_{s=0}^{\ell} \underline A_{\gamma_{s} \gamma_{s+1}}  \\ 
&\quad +\sum_{t = 0}^\ell \sum_{\gamma \in F^{\ell+1} _{e f}}     \prod_{s=0}^{t-1} \underline A_{\gamma_{s} \gamma_{s+1}} \lr{ \frac{\phi_{\gamma_t} \phi_{\gamma_{t+1}}}{n} W_{\sigma_{\gamma_{t}} \sigma_{\gamma_{t+1}}}} \prod_{s=t+1}^\ell A_{\gamma_{s} \gamma_{s+1}} \\
&= \Delta^{(\ell)}_{ef}  +\sum_{t = 0}^\ell \sum_{\gamma \in F^{\ell+1} _{e f}}     \prod_{s=0}^{t-1} \underline A_{\gamma_{s} \gamma_{s+1}} \lr{ \frac{\phi_{\gamma_t} \phi_{\gamma_{t+1}}}{n} W_{\sigma_{\gamma_{t}} \sigma_{\gamma_{t+1}}}} \prod_{s=t+1}^\ell A_{\gamma_{s} \gamma_{s+1}}.   
\ea
\label{eq::B_ef_telescopic}
\ee
Consider the two products in the summation over $F^{\ell+1}_{e f}$ on the right of \eqref{eq::B_ef_telescopic}: We can, for $1 \leq t \leq \ell - 1$, replace the summation over $F^{\ell+1}_{e f}$ by summing over all pairs 
$\gamma' = (\gamma_0, \ldots, \gamma_t)  \in F^{t}_{eg}$ and $\gamma'' = ( \gamma_{t+1}, \ldots , \gamma_{\ell +1}) \in F^{\ell- t}_{g' f}$ for some $g,g' \in \vec E (V)$ such that there exists a non-backtracking path with one intermediate edge, on the \emph{complete} graph, between oriented edges $g$ and $g'$ (we denote this property by $g \stackrel{2}{\to}  g' $). However caution is needed, as this summation also includes \emph{tangled} paths, namely those in the sets $\{ F^{\ell+1}_{t, ef} \}_{t=0}^{\ell}.$ Where, for $1 \leq t \leq \ell - 1$, $F^{\ell+1}_{t, ef}$ is defined as the collection of all \emph{tangled} paths $\gamma = (\gamma_0, \ldots, \gamma_{\ell +1}) =  (\gamma',\gamma'') \in  \Gamma^{\ell+1}_{ef}$ with $\gamma'$ and $\gamma''$ as above. For $t=0$, $F_{0,ef}^{\ell+1}$ consists of all non-backtracking \emph{tangled} paths $(\gamma',\gamma'')$ with $\gamma' = (e_1)$ and $\gamma'' \in F^{\ell}_{g' f}$ for any $g'$ such that $g'_1 = e_2$. For $t = \ell$, $F_{\ell,ef}^{\ell+1}$ is the set of non-backtracking \emph{tangled} paths $(\gamma',\gamma'')$ such that $\gamma'' = (f_2)$ and $\gamma' \in F^{\ell}_{e g}$ for some $g \in \vec E(V)$ with $g_2 = f_1$. We rewrite \eqref{eq::B_ef_telescopic} as
\be
B^{(\ell)}  =  \Delta^{(\ell)}   +  \frac 1 n K B^{(\ell-1)}  +  \frac 1 n  \sum_{t = 1} ^{\ell-1}  \Delta^{(t-1)} K^{(2)}  B^{(\ell - t -1)}  +  \frac 1 n \Delta^{(\ell-1)} \widehat{K}  -   \frac 1 n    \sum_{t = 0}^{\ell} R^{(\ell)}_t, \label{eq::B_dec} \ee
where for $e,f \in E_K$,
\be 
K_{e f} = \IND{e \to  f} \phi_{e_1} \phi_{e_2} W_{\sigma(e_1) \sigma(e_2)},  \label{eq::K}
\ee
the \emph{weighted} non-backtracking matrix on the \emph{complete} graph (recall that $e \to f$ represents the non-backtracking property),
\be 
\widehat{K}_{e f} = \IND{e \to  f} \phi_{f_1} \phi_{f_2} W_{\sigma(f_1) \sigma(f_2)},  \label{eq::widehat_K}
\ee
\be 
K^{(2)}_{ef}  = \IND{ e \stackrel{2}{\to}  f }\phi_{e_2} \phi_{f_1}  W_{\sigma(e_2) \sigma(f_1)}  
,  \label{eq::K2}
\ee
and where
\be 
(R_t ^{(\ell)} )_{ef}   =  \sum_{\gamma \in F^{\ell+1} _{t,e f}} \prod_{s=0}^{t-1} \underline A_{\gamma_{s} \gamma_{s+1}}  \phi_{\gamma_t} \phi_{\gamma_{t+1}} W_{\sigma(\gamma_t) \sigma(\gamma_{t+1})} \prod_{s=t+1}^\ell A_{\gamma_{s} \gamma_{s+1}}. \label{eq::R}
\ee
Indeed, 
\be
\ba 
\lr{ \sum_{t = 1}^{\ell-1}  \Delta^{(t-1)} K^{(2)}  B^{(\ell - t -1)}}_{ef} &=  \sum_{t = 1}^{\ell-1} \sum_{g,g'}  \Delta^{(t-1)}_{eg} K^{(2)}_{gg'}  B^{(\ell - t -1)}_{g'f} \\
&= \sum_{t = 1}^{\ell-1} \sum_{g,g'} \sum_{\gamma' \in F^{t}_{eg}} \sum_{\gamma'' \in F^{\ell-t}_{g'f}} \prod_{s=0}^{t-1} \underline A_{\gamma'_{s} \gamma'_{s+1}}  \IND{ g \stackrel{2}{\to}  g' } \phi_{\gamma'_t} \phi_{\gamma''_{0}}  \\
&\quad \cdot W_{\sigma(\gamma'_t) \sigma(\gamma''_{0})} \prod_{s=0}^{\ell-t-1} A_{\gamma''_{s} \gamma''_{s+1}} , \label{eq::Delta_K2_B}
\ea
\ee

\be
\ba 
\lr{ K  B^{(\ell  -1)}}_{ef} = \sum_g \sum_{\gamma'' \in F^{\ell}_{gf}} \IND{e \to g} \phi_{e_1} \phi_{e_2} W_{\sigma(e_1) \sigma(e_2)} A_{e_2, g_2} \prod_{s=1}^{\ell-2} A_{\gamma''_{s} \gamma''_{s+1}} A_{f_1 f_2},
\ea
\ee
and,
\be
\ba 
\lr{   \Delta^{(\ell  -1)} \widehat{K} }_{ef} = \sum_g \sum_{\gamma' \in F^{\ell}_{eg}}  
\underline {A}_{e_1 e_2} 
\prod_{s=1}^{\ell-2} \underline A_{\gamma'_{s} \gamma'_{s+1}} \underline A_{g_1 f_1} 
\IND{g \to f}
\phi_{f_1} \phi_{f_2} W_{\sigma(f_1) \sigma(f_2)} 
\ea
\ee
that is exactly the splitting described just below \eqref{eq::B_ef_telescopic}, where we also pointed out the need to compensate for \emph{tangled} paths occuring in \eqref{eq::Delta_K2_B}, which is precisely the role of $R_t ^{(\ell)}$ in \eqref{eq::B_dec}. 

To bound \eqref{eq::B_dec}, we introduce
\be
\overline{W} = \frac{2}{\PHItwo} \lr{ \rho \chi_1 \check \chi_1^* + \mu_2 \chi_2 \check \chi_2^*} = \lr{ \phi_{e_2} \phi_{f_1}  W_{\sigma(e_2) \sigma(f_1)} }_{ef} ,
\label{eq::W_bar}
\ee
and,
\be
L = K^{(2)}- \overline{W}.
\label{eq::L}
\ee
Note the presence of weights in \eqref{eq::W_bar}, hence our choice for the candidate eigenvectors. 

Further, we set for $1 \leq t \leq \ell -1$, 
\be
S_t^{(\ell)} = \Delta^{(t-1)} L B^{(\ell - t -1)}.
\label{eq::S_t_ell}
\ee
We then have: 
\begin{proposition}[\DCA{Proposition $13$}]
\label{prop::13}
If $G$ is tangle-free and $x\in\mathbb{C}^{\vec E(V)}$ with norm smaller than one, we have
\be
\ba
 \| B^{\ell} x \| &\leq  \|  \Delta^{(\ell)} \|   +    \frac{ 1}{  n}  \| K  B^{(\ell-1)} \|   +    \frac 1  n  \sum_{j=1,2} \frac{2 \mu_j}{\PHItwo}  \sum_{t = 1} ^{\ell-1}    \| \Delta^{(t-1)} \chi_j  \| || \langle \check \chi_j , B^{\ell-t-1} x \rangle ||  \nonumber \\
&  \quad +  \;  \frac{1}{n}    \sum_{t = 1} ^{\ell-1} \| S_t^{(\ell)} \|   +  \PHImax^2(a \vee b) \| \Delta^{(\ell-1)}  \|   +  \frac{1}{n}   \sum_{t = 0}^\ell \| R^{(\ell)}_t \|.
\ea
\ee
\end{proposition}
\begin{proof}
Due to the tangle-freeness, $B^{\ell} = B^{(\ell)}$. Further $K^{(2)} = L + \overline{W}$ and $||K|| \leq \PHImax^2(a \vee b) n$.
\end{proof}

In appendix \ref{App::norms} we prove the following bounds on the matrices in Proposition \ref{prop::13}:
\begin{proposition}[\DCA{Proposition $14$}]
\label{prop::14}
Let $\ell = C \log_\rho n$ with $C < 1$.
With high probability, the following norm bounds hold for all $k$, $0 \leq k \leq \ell$, and   $i=1,2$:
\begin{eqnarray}
\label{prop:normDelta}
\| \Delta^{(k)} \| \leq ( \log n) ^{10} \rho^{k /2},
\\
\label{prop:Dscalar}
\|  \Delta^{(k)} \chi_i  \|    \leq (\log n) ^{5}  \rho^{k / 2} \sqrt n,
\\
\label{prop:normR}
\| R^{(\ell)}_k \|    \leq (\log n) ^{25}  \rho^{\ell  - k /2},\\
\label{prop:normB2} 
 \|  K B^{(k)} \|    \leq \sqrt n (\log n) ^{10}  \rho^{k},
\end{eqnarray}
and the following bound holds for all $k$, $1 \leq k \leq \ell-1$:
\begin{equation}
\label{prop:normS}
\| S^{(\ell)}_k \|    \leq \sqrt n (\log n) ^{20}  \rho^{\ell  - k /2}.
\end{equation}
\end{proposition}

\subsection{Proof of Proposition \ref{prop::20}}
\label{ssec::Prop20}
From Propositions \ref{prop::13} and \ref{prop::14}, the geometric growth in Corollary \ref{cor::34} together with the tangle-freeness due to Lemma \ref{lm::30}, the proof of Proposition \ref{prop::20} follows:

Let $j \in \{1,2\}$. If, for some vector $x$, $\langle \check \varphi_j,x \rangle = 0$, then $\langle B^\ell \chi_j, \check x \rangle = 0$. Therefore, using Corollary \ref{cor::34},
\BA
\sup_{ \|x\| = 1, \langle \check \varphi_j,x \rangle = 0}  \langle  \check \chi_j, B^{\ell-t-1} x \rangle 
& = \sup_{ \|x\| = 1, \langle B^\ell \chi_j, \check x \rangle = 0}  \langle  B^{\ell-t-1} \chi_j,  \check x \rangle  \\
& = \sup_{ \|\check x\| = 1, \langle B^\ell \chi_j, \check x \rangle = 0}  \langle  B^{\ell-t-1} \chi_j,  \check x \rangle \\
&\leq \log^2(n) n^{1/2} \rho^{\frac{\ell - t -1}{2}}.  
\EA

With high probability, the graph is $\ell-$ tangle free (Lemma \ref{lm::30}). Thus, invoking Propositions \ref{prop::13}  and \ref{prop::14}, with high probability,
\BA
\sup_{ x  \in H^\perp, \|x\| = 1} \| B^{\ell} x \| & \leq 
\log^{10}(n)  \rho^{\frac{\ell}{2}} 
+   n^{-1/2} \log^{10}(n)  \rho^{\ell - 1} \\
&\quad +   c_1 \log^{8}(n)  \rho^{\frac{\ell}{2}}
 +   n^{-1/2} \log^{21}(n)  \rho^{\ell} \\
&\quad +   c_2 \log^{10}(n)  \rho^{\frac{\ell}{2}}
 +   n^{-1} \log^{26}(n)  \rho^{\ell} \\
&\leq \log^{c}(n)  \rho^{\frac{\ell}{2}},
\EA 
since $C<1$.

\subsection{Comparison with the Stochastic Block Model in \cite{BoLeMa15}}
\label{ssec::comparison}
Putting $\phi_u=1$ for all $u$, we retrieve exactly  \textbf{the same bounds as in the Stochastic Block Model}, that is equations $(30)-(34)$ in \cite{BoLeMa15}.

Below we use the trace method and therefore path counting combinatorial arguments to establish Proposition \ref{prop::14}. In particular, we bound the expectation of expressions of the form 
\be \E{\prod_{i=1}^{2m}  \prod_{s=1}^{k} \underline  A_{\gamma_{i,s-1} \gamma_{i,s}}}, \label{eq::example} \ee 
for certain paths $\gamma = ( \gamma_1, \ldots, \gamma_{2m})$ with $\gamma_i = (\gamma_{i,0}, \cdots, \gamma_{i,k})\in V^{k+1}$, where $\underline{A}$ is defined in \eqref{eq::A_uv_centered}.

In bounding \eqref{eq::example} the following term occurs:
\[ \prod_{u \in V(\gamma)}  \Phi^{(d_u)}, \]
where $(d_u)_u$ are the degrees of the vertices in a specific tree (or forest) spanning the path $\gamma$. See, for instance, \eqref{eq::Egamma_bound} and \eqref{eq::Egamma_bound_R} below.
Here lies a major complication with respect to the Stochastic Block Model: those terms are not present in the latter model.  	 
In \eqref{eq::prod_phi} and \eqref{eq::prod_phi_R} we find
\[ \prod_{u=1}^{|V(\gamma)|}  \Phi^{(d_u)} 
\leq C_2^{\sum_{u: d_u > 2}(d_u-2)} \lr{\Phi^{(2)}}^{|V(\gamma)|-n_{\cC}},
\]
where $C_2 > 1$ is some constant and where $n_{\cC} \geq 1$ is the number of components on the path $\gamma$. To compare this term with powers of $\PHItwo$ (which are present in powers of $\rho = \frac{a+b}{2} \PHItwo$), we bound 
$\sum_{u: d_u > 2}(d_u-2)$, see in particular Lemma \ref{lm::bound_overshoot} and Lemma \ref{lm::bound_overshoot_forest}.

\section{Detection: Proof of Theorem \ref{th::5}}
\label{sec::detection}
The proofs of the statements in this section are deferred to Appendix \ref{App::detection}.

We need the following special case of a lemma in \cite{BoLeMa15}:
\begin{lemma}[Special case of Lemma $40$ in \cite{BoLeMa15}]
\label{lm::40}
Assume that there exists a function $F : V \to \{0,1\}$ such that in probability, for any $i \in \spm$,
$$
 \lim_{n\to \infty} \frac{1}{n} \sum_{v = 1}^n  \IND{\sigma(v) = i}  F(v) = \frac{ f(i)}{2}, 
$$
where $f: \spm \to [0,1]$ is such that $f(+) > f(-)$. Then, assigning to each vertex a label $\widehat{\sigma}(v) = +$ if $F(v) = 1$ and $\widehat{\sigma}(v) = -$ if $F(v) = 0$, yields asymptotically positive overlap with the true spins.
\end{lemma}

Recall the eigenvector $\xi_2$ from Theorem \ref{th::4}. Below we use the function $F: v \mapsto \IND{\sum_{e : e_2 = v}  \xi_2 (e) > \frac{\tau}{ \sqrt { n}} }$ or $F: v \mapsto \IND{\sum_{e : e_2 = v}  \xi_2 (e) \leq \frac{\tau}{ \sqrt { n}} }$ for some fixed parameter $\tau$. We verify also that $\xi_2$ is aligned with $P_{2,\ell}$. It is therefore useful to introduce the vector $I_\ell$, defined element-wise by 
\be
I_{\ell} (v) = \sum_{e \in \vec E : e_2 = v} P_{2,\ell} (e),
\label{eq::I_ell}
\ee
for $v \in V$. 

Further, put \[\widehat{c} = \frac{a+b}{2} \frac{(\PHI)^2  \PHIthree}{\PHItwo} \frac{\rho}{\mu_2^2 - \rho} \mu_{2}\] 

The following lemma shows that $I_{\ell}$ is correlated with the spins:
\begin{lemma}[\DCA{Lemma $41$}]
\label{lm::41}
Let $\ell = C \log_\rho n$ with $C < C_{\text{coupling}}$ and $i \in \spm$. There exists a  random variable $Y_{i}$ such that $\E{Y_{i}} = 0$, $\E{ |Y_{i}| } < \infty$ and for any continuity point $t$ of the distribution of $Y_{i}$, in $L^2$, $$
\frac 1 {n} \sum_{v = 1}^n \IND{\sigma(v) = i }  \IND{  I_{\ell} (v) \mu_2 ^{- 2\ell}  -  \widehat{c} \EF_2 (i)   \geq t } \to  \frac{1}{2}  \P{  Y_{i}  \geq t }.
$$
\end{lemma}

Recall from Theorem \ref{th::4} that the eigenvector $\xi_2$ is asymptotically \emph{aligned} with 
\be  \frac{ B^{\ell} B^{* \ell} \check \chi_2 }{\| B^{\ell} B^{* \ell} \check \chi_2 \|}, \label{eq::vec_prop_2} \ee 
where $\ell \sim \log_\rho(n)$. Hence, for some unknown sign $\omega$, the vector $\xi'_2 = \omega \xi_2$ is asymptotically \emph{close} to \eqref{eq::vec_prop_2}.
From Lemma \ref{lm::39} we know that $B^{\ell} B^{* \ell} \check \chi_2$ and $P_{2,\ell}$ are asymptotically close.   Consequently, properly renormalizing $\xi'_2$ will make it asymptotically close to $P_{2,\ell}$, so that we can replace $P_{2,\ell}$ in \eqref{eq::I_ell} by $\xi'_2$. That is, we set for $v \in V$,
$$
I(v) = \sum_{e : e_2 = v}  s \sqrt { n} \xi'_2 (e),
$$
with $s = \sqrt{c''''_2}$ the limit in Proposition \ref{prop::38}. Then,
$I$ and $I_\ell / \mu_2^{2 \ell}$ are close, which leads to the following lemma:

\begin{lemma}[\DCA{Lemma $42$}]
\label{lm::42}
Let $i \in \spm$ and $\widehat{Y}_{i}$ be as in Lemma \ref{lm::41}. For any continuity point $t$ of the distribution of $\widehat{Y}_{i}$, in $L^2$,
 $$
\frac 1 {n} \sum_{v = 1}^n \IND{\sigma(v) = i }  \IND{  I(v) -  \widehat{c} \EF_2 (i)   \geq t } \to  \frac{1}{2}  \P{  \widehat{Y}_{i}  \geq t }.
$$
\end{lemma}

Put for $i \in \spm$, $X_i = \widehat{Y}_{i} + \widehat{c} g_2(i) = \widehat{Y}_{i} + \frac{1}{\sqrt{2}}\widehat{c} i$. Then, for all $t \in \mathbb{R}$ that are continuity points of the distribution of $X_i$, the following convergence holds in probability
$$
\frac 1 {n} \sum_{v = 1}^n \IND{\sigma(v) = i }  \IND{  I(v) > t } \to  \frac{1}{2}  \P{  X_{i}  > t }.
$$
Since $\E{X_+} > 0$, the argument below $(90)$ in \cite{BoLeMa15} establishes the existence of a continuity point $t_0 \in \mathbb{R}$ such that $\P{X_+ > t_0} > \P{X_- > t_0}$.

Further, we note that $X_+$ is in distribution equal to $-X_-$, a fact that we use below.

We are now in a position to apply Lemma \ref{lm::40} and thereby finishing the proof of Theorem \ref{th::5}:

If $\omega = 1$, then we define $F$, for $v \in V$, by
\[ F(v) = \IND{\sum_{e : e_2 = v}  \xi_2 (e) > \frac{t_0}{s \sqrt { n}} } = \IND{  I(v) > t_0 }.\]
Then,
$$
 \lim_{n\to \infty} \frac{1}{n} \sum_{v = 1}^n  \IND{\sigma(v) = +}  F(v) = \frac{1}{2}  \P{  X_{+}  > t_0 } =: \frac{f(+)}{2}, 
$$
and,
$$
 \lim_{n\to \infty} \frac{1}{n} \sum_{v = 1}^n  \IND{\sigma(v) = -}  F(v) = \frac{1}{2}  \P{  X_{-}  > t_0 } =: \frac{f(-)}{2}, 
$$
so that $f(+) > f(-)$ and Lemma \ref{lm::40} applies.

If, however, $\omega = -1$, then we define $F$, for $v \in V$, by
\[ F(v) = \IND{\sum_{e : e_2 = v}  \xi_2 (e) \leq \frac{t_0}{s \sqrt { n}} } = \IND{  -I(v) \leq t_0 }.\]
Then, this time, 
$$
 \lim_{n\to \infty} \frac{1}{n} \sum_{v = 1}^n  \IND{\sigma(v) = +}  F(v) 
 =  \lim_{n\to \infty} \frac{1}{n} \sum_{v = 1}^n  \IND{\sigma(v) = +}  \IND{  I(v) > -t_0 }
  = \frac{1}{2}  \P{  X_{+}  > -t_0 } =: \frac{f(+)}{2}, 
$$
since $-t_0$ is a continuity point of $X_+$, which follows from the fact that $X_+$ is in distribution equal to $-X_-$ and $t_0$ is a continuity point of $X_-$. 

Similarly,
$$
 \lim_{n\to \infty} \frac{1}{n} \sum_{v = 1}^n  \IND{\sigma(v) = -}  F(v) 
  = \frac{1}{2}  \P{  X_{-}  > -t_0 } =: \frac{f(-)}{2}. 
$$
Now, 
\[ f(+) = \P{  X_{+}  > -t_0 } = 1 - \P{  X_{-}  > t_0 } > 1 - \P{  X_{+}  > t_0 } = \P{  X_{-}  > -t_0 } = f(-), \]
exactly the setting of Lemma \ref{lm::40}.

\appendix
\section{Proofs of Section \ref{sec::branching}}
\label{App::branching}
\begin{proof}[Proof of Theorem \ref{th::21}]
For $1\leq q<t$, we have
\[
Z _t-M^{t-s} Z _s =\sum_{u=s}^{t-1} M^{t-u-1}(Z_{u+1}-M Z_u),
\]
consequently, as $\EF_k^*M=\mu_k \EF_k^*$,
\be \label{eq::Z_t_recursive}
\frac{\langle \EF_k ,  Z_t  \rangle}{\mu^{t-1}_k} = \frac{\langle \EF_k
  ,  Z_q \rangle}{\mu_k^{q-1}} 
+\sum_{u=q}^{t-1} \frac{\langle \EF_k , Z_{u+1}-MZ_u \rangle }{\mu_k^{u}},
\ee
compare to (55) in \cite{BoLeMa15}.
Hence, $(X_k(t))_{t \geq 1}$ is an $\cF_t$-martingale with mean $0$. We shall invoke Doob's martingale convergence theorem to prove the assertion. That is, we shall show that for some $C > 0$ and all $t \geq 1$, 
$$
\E{ X^2_k (t) | Z_1  } \leq C \| Z_1 \|_1. 
$$ 

Let, for $i,j \in \spm$, $Z_{s+1}(i,j)$ denote the number of type $i$ individuals in generation $s+1$ which descend from from a type $j$ particle in the $s$-th generation. Then,
\be 
\E{  \|Z_{s+1}-MZ_s \|_2^2 |Z_s } 
= \sum_{i,j \in \spm} \E{  \lr{Z_{s+1} (i,j) - 
    M_{ij} Z_s (j) }^2 |Z_s(j) }.   \label{eq::Z_s_norm2}
\ee

We calculate first, for some integer $z \geq 0$, 
\be 
\ba 
\E{  \lr{Z_{s+1} (i,j) - M_{ij} Z_s (j) }^2 |Z_s(j) = z} 
&= \E{ \left. \lr{ \s{l=1}{z} (Y_l (i,j) - M_{ij}) }^2 \right| Z_s(j) = z} \\
&= \s{l=1}{z} \E{  \lr{  Y_l (i,j) - M_{ij} }^2 }, \label{eq::Z_s_ij}
\ea
\ee
where $\lr{Y_l (i,j)}_{l=1}^z$ are i.i.d. copies of Poi$\lr{\frac{\IND{i=j}a + \IND{i \neq j}b}{2}  \PHI \phi^{*}}$, where $\phi^*$ follows the biased law $\NUstar$. 

 Put $c_1 = \max_{i,j \in \spm} \E{  \lr{  Y_l (i,j) - M_{ij} }^2 } < \infty$. Then, plugging \eqref{eq::Z_s_ij} into \eqref{eq::Z_s_norm2}, we obtain 
\[ \E{  \|Z_{s+1}-MZ_s \|_2^2 |Z_s } \leq 2 c_1  \| Z_s \|_1.\]
Consequently,
\be \ba 
\E{  \|Z_{s+1}-MZ_s \|_2^2 |Z_1 }  
&= \E{  \E{ \|Z_{s+1}-MZ_s \|_2^2 |Z_s } | Z_1} \\
&\leq 2c_1 \E{ \| Z_s \|_1 | Z_1} \\
&= 2c_1 \rho^{s-1} \| Z_1 \|_1.
\ea \ee
Combining the above with \eqref{eq::Z_t_recursive} for $q=1$, we obtain
\be \ba 
\E{ X^2_k (t) | Z_1  }  
&= \sum_{s=1}^{t-1}  \frac{ \E { \langle \EF_k , (Z_{s+1}-MZ_s ) \rangle ^2 | Z_1 } }{\mu_k^{2s}} \\
&\leq \| \EF_k \|_2^2 \sum_{s=1}^{t-1}  \frac{ \E{  \|Z_{s+1}-MZ_s \|_2^2 |Z_1 } }{\mu_k^{2s}} \\
&\leq 2c_1 \| \EF_k \|_2^2 \sum_{s=1}^{\infty} \frac{\rho^{s-1}}{\mu_k^{2s}} \| Z_1 \|_1.
\ea \ee
The assertion now follows upon noting that 
\[ C:= 2c_1 \max_{k \in \spm} \| \EF_k \|_2^2 \sum_{s=1}^{\infty} \frac{\rho^{s-1}}{\mu_k^{2s}} < \infty,\]
since $\rho < \mu_{k}^2$.
\end{proof}
\begin{proof}[Proof of Corollary \ref{CO::21}]
From Theorem \ref{th::21} we know that there exists a random variable $X_k(\infty)$ such that
\[ X_k(t) :=  \frac{\langle \EF_{k} , Z _t \rangle}{\mu^{t-1}_k} - \langle  \EF_k, Z_1  \rangle \overset{\text{a.s.}} \to X_k(\infty),\]
as $t \to \infty$. 
Now,
\[ \langle \EF_{k} , Z _1 \rangle = \mu_{k,\psi_o} \langle \EF_{k} , Z _0 \rangle + \langle \EF_{k} , Z _1 - M_{\psi_o} Z_0 \rangle. \]
We combine this with the definition of $X_k(t)$ to obtain
\[ \frac{\langle \EF_{k} , Z _t \rangle}{\mu_{k,\psi_o} \mu^{t-1}_k} = \langle \EF_{k} , Z _0 \rangle + \frac{ \langle \EF_{k} , Z _1 - M_{\psi_o} Z_0 \rangle }{\mu_{k,\psi_o}} + \frac{X_k(t)}{\mu_{k,\psi_o}}, \]
where the right hand side is seen to converge in both senses to the random variable 
\[Y_{k,\psi_o}(\infty) = \langle \EF_{k} , Z _0 \rangle + \frac{ \langle \EF_{k} , Z _1 - M_{\psi_o} Z_0 \rangle }{\mu_{k,\psi_o}} + \frac{X_k(\infty)}{\mu_{k,\psi_o}}. \]
Indeed,
\[ \left| \frac{\langle \EF_{k} , Z _t \rangle}{\mu_{k,\psi_o} \mu^{t-1}_k} -  Y_{k,\psi_o}(\infty) \right| \leq \frac{1}{\mu_{k,\PHImin}} \left| X_k(t) - X_k(\infty) \right|,  \]
for all $\psi_o$. The uniform convergence follows, since
\BA & \E{ \left. | X_k(t) - X_k(\infty) |^2  \right|  	\phi_0 = \psi_o} \\
&= \s{z=0}{\infty} \E{ \left. \left| X_k(t) - X_k(\infty) \right|^2  \right| \|Z_1\| = z} \ \P{\|Z_1\| = z |\phi_0 = \psi_o} \\
&\leq e^{\frac{a+b}{2}\PHI (\PHImax - \PHImin)} \E{ \left| X_k(t) - X_k(\infty) \right|^2  |\phi_0 = \PHImax}  \EA

\end{proof}
\begin{proof}[Proof of Theorem \ref{th::21_other_martingale}]
For $1\leq q<t$, we have again
\be \label{eq::Z_t_recursive_}
\frac{\langle \EF_k ,  \Psi_t  \rangle}{\mu^{t-1}_k} = \frac{\langle \EF_k
  ,  \Psi_q \rangle}{\mu_k^{q-1}} 
+\sum_{u=q}^{t-1} \frac{\langle \EF_k , \Psi_{u+1}-M \Psi_u \rangle }{\mu_k^{u}}.
\ee
Since $\E{\Psi_{u+1}|\Psi_u} = M \Psi_u$, $(X_k(t))_{t \geq 1}$ is an $\cG_t$-martingale with mean $0$. We show again that for some $C > 0$ and all $t \geq 1$, 
$$
\E{ X^2_k (t) | Z_1  } \leq C \| Z_1 \|_1. 
$$ 

Let, for $i,j \in \spm$, $\Psi_{s+1}(i,j)$
denote the sum over the weights of type $i$ individuals in generation $s+1$ which descend from a type $j$ particle in the $s$-th generation. Then,
\be 
\E{  \|\Psi_{s+1}-M \Psi_s \|_2^2 |Z_s } 
= \sum_{i,j \in \spm} \E{  \lr{\Psi_{s+1} (i,j) - 
    M_{ij} \Psi_s (j) }^2 |Z_s(j) }.   \label{eq::Z_s_norm2_}
\ee

We calculate first, for some integer $z \geq 0$, 
\be 
\ba 
\E{  \lr{\Psi_{s+1} (i,j) - M_{ij} \Psi_s (j) }^2 |Z_s(j) = z} 
&= \E{ \left. \lr{ \s{l=1}{z} \lr{\s{l'=1}{Y_l (i,j)} \phi^i_{ll'} - M_{ij} \phi^j_l } }^2 \right| Z_s(j) = z} 
\ea
\ee
where $\phi^i_{ll'}$ and $\phi^j_l$ are all independent and governed by the biased law $\NUstar$, and where  $\lr{Y_l (i,j)}_{l=1}^z$ are i.i.d. copies of Poi$\lr{\frac{\IND{i=j}a + \IND{i \neq j}b}{2}  \PHI \phi^{*}}$, with $\phi^*$ governed by $\NUstar$. Thus the summands indexed by $l$ are independent. We have
\[ \E{ \left. \s{l'=1}{Y_l (i,j)} \phi^i_{ll'} - M_{ij} \phi^j_l \right| Z_s(j) }
= \frac{\IND{i=j}a + \IND{i \neq j}b}{2}  \PHI \frac{\PHItwo}{\PHI}  \frac{\PHItwo}{\PHI} - M_{ij} \frac{\PHItwo}{\PHI} = 0  \]
Therefore,
\be 
\ba 
\E{  \lr{\Psi_{s+1} (i,j) - M_{ij} \Psi_s (j) }^2 |Z_s(j) = z} 
&= \s{l=1}{z} \E{  \lr{  \s{l'=1}{Y_l (i,j)} \phi^i_{ll'} - M_{ij} \phi^j_l  }^2 }. \label{eq::Z_s_ij_}
\ea
\ee
 Put $c_1 = \max_{i,j \in \spm} \E{  \lr{  \s{l'=1}{Y_l (i,j)} \phi^i_{ll'} - M_{ij} \phi^j_l }^2 } < \infty$. Then, plugging \eqref{eq::Z_s_ij_} into \eqref{eq::Z_s_norm2_}, we obtain 
\[ \E{  \|\Psi_{s+1}-M \Psi_s \|_2^2 |Z_s } \leq 2 c_1  \| Z_s \|_1.\]
\end{proof}

\begin{proof}[Proof of Lemma \ref{lm::23}]
For $k \geq 1$, put 
$$
\epsilon_k = \rho^{-k/2} \sqrt{k} \quad  \hbox{ and } \quad f_k = \prod_{\ell=1}^k ( 1 + \epsilon_{\ell}).
$$
Due to convergence of $(f_k)_k$, there exist constants $c_0,c_1 >0$ such that for all $k \geq 1$,  
\begin{equation}
c_0 \leq f_k \leq c_1 \quad  \hbox{ and } \quad \epsilon_k \leq c_1,
\label{eq::bounds_on_f}
\end{equation}
exactly as $(57)$ in \cite{BoLeMa15}.

Recall the law of $S_{k+1}$ from \eqref{eq::S_k_plus_1}. 
We shall firstly derive a concentration result for $\s{l=1}{S_k} X_k^{(l)}$, by using Hoeffding's inequality. Note that by definition $X_k^{(l)} \in \frac{a+b}{2} \PHI [\PHImin, \PHImax]$. Put $\gamma =  (\frac{a+b}{2} \PHI)^2 (\PHImax - \PHImin)^2$, then Hoeffding's equality reads
\[ \P{ \left| \s{l=1}{n} X_k^{(l)} - n \rho \right| \geq t } \leq 2 \exp \lr{- \frac{2 t^2}{n \gamma}}. \] 
Hence, in particular,  
\be \P{ \left| \s{l=1}{s f_k \rho^k} X_k^{(l)} - s f_k \rho^k \rho \right| \geq s f_k \rho^k \rho \frac{\epsilon_{k+1}}{2} } \leq 2 \exp \lr{- \frac{f_k \rho(k+1)}{2 \gamma}s} \leq 2 \exp \lr{-c_2 s}, 
\label{eq::15}
\ee
for some $c_2 > 0$, due to \eqref{eq::bounds_on_f}.
We use the last result to obtain
\be 
\ba
&\P{S_{k+1} > s f_{k+1}\rho^{k+1} | S_k \leq s f_k \rho^k} \\ 
&\leq \P{\text{Poi}\lr{\s{l=1}{s f_k \rho^k} X_k^{(l)}} > s f_{k+1}\rho^{k+1}}\\
&\leq \P{\text{Poi}\lr{ s f_k \rho^{k+1} \lr{1 +  \frac{\epsilon_{k+1}}{2}} } > s f_{k+1}\rho^{k+1}} \lr{1 - 2 \text{e}^{-c_2 s} } \\
&\quad + 2\text{e}^{-c_2 s}.
\label{eq::S_k_plus_1_bdd}
\ea
\ee
We bound 
\[ \ba s f_{k+1}\rho^{k+1} 
&= s f_{k}\rho^{k+1} \lr{1 + \frac{\epsilon_{k+1}}{2}} \frac{1 + \epsilon_{k+1}}{1 + \frac{\epsilon_{k+1}}{2}} \\
&\geq s f_{k}\rho^{k+1} \lr{1 + \frac{\epsilon_{k+1}}{2}} (1 + c_3 \epsilon_{k+1}),
\ea \]
where $c_3 = \frac{1}{2}\frac{1}{1 + \max_l \epsilon_l / 2} > 0$.
Combining the last estimate with \eqref{eq::S_k_plus_1_bdd} and the inequality 
\[ \P{\Pois{\lambda} \geq \lambda s} \leq e^{-\lambda I(s)}, \]
where 
\be
I: x \mapsto \left\{ \begin{array}{ll}
         x \text{log} x - x + 1 & \mbox{if $x > 0$};\\
        \infty & \mbox{if $x \leq 0$},\end{array} \right.  
\label{eq::17}
\ee
entails that
\[ \P{S_{k+1} > s f_{k+1}\rho^{k+1} | S_k \leq s f_k \rho^k} \leq \exp \lr{-s f_k \rho^{k+1} \lr{1 +  \frac{\epsilon_{k+1}}{2}} I(1 + c_3 \epsilon_{k+1})}  + 2\text{e}^{-c_2 s}. \]
It remains to bound $I(1 + c_3 \epsilon_k)$ from below. But, due to the form of $I$, 
there exists a $\theta > 0$ such that for $x \in [0, c_3 \max_k \epsilon_k]$, $I(1+x) \geq \theta x^2$.
Consequently
\[ \P{S_{k+1} > s f_{k+1}\rho^{k+1} | S_k \leq s f_k \rho^k} \leq 3 \e^{-c_4 s k}, \]
for some constant $c_4 > 0$. Hence,
\[ \P{ \exists k : S_{k}  > s c_1 \rho^{k}  } \leq \sum_{k=1}^{\infty} 3 e^{ -   c_4s k  } = \frac{3 }{1 - e^{ -   c_4s  }} e^{ -   c_4s  },\]
from which the statement follows.
\end{proof}

\begin{proof}[Proof of Theorem \ref{th::24}]
We claim that there exist constants $c, c' > 0$ such that for any $s \geq 0$
\be  
\P{ \| Z_{t+1}  -  M Z_t  \|_2   > s \| Z_t \|^{1/2}_1 \bigm| \cF_t } 
\leq  c'  \e^{-c ( s \wedge s^2 )}.
\label{eq::M_Z_t_dev}
\ee
To prove \eqref{eq::M_Z_t_dev}, we shall employ Hoeffding's inequality 
to establish a concentration result for 
\be 
\lambda^+ = \frac{\PHI}{2} \lr{a  \s{i=1}{Z_t^+} \Phi_i^+ +   b\s{i=1}{Z_t^-} \Phi_i^- },
\label{eq::Lambda_plus}
\ee
and,
\be 
\lambda^- = \frac{\PHI}{2} \lr{b  \s{i=1}{Z_t^+} \Phi_i^+ +   a\s{i=1}{Z_t^-} \Phi_i^- }
\label{eq::Lambda_min}
\ee
around their respective means $y^+ = \Esub{\lambda^+}{*}$ and $y^- = \Esub{\lambda^-}{*}$,
where $(\Phi_i^{\pm})_i$ are i.i.d. random variables with law $\NUstar$, and where $\Esub{\cdot}{*} = \E{\cdot|Z_t}.$
This in conjunction with the classical tail bound for $Y  \overset{d} = \Poi{(\lambda)}$:
\be
\P{  |Y - \lambda|  >\lambda s } \leq  2 e^{ - \lambda \delta(s)},
\label{eq::Poi_tail}
\ee
where $\delta: x \mapsto I(1-x) \wedge I(1+x)$, with $I$ defined in \eqref{eq::17}, shall allow us to prove concentration of 
$ \left( \begin{array}{c}
Z_{t+1}^+  \\
Z_{t+1}^- \end{array} \right)
= 
\left( \begin{array}{c}
\Pois{\lambda^+}  \\
\Pois{\lambda^-} \end{array} \right)
$ 
around 
$
\Esub{\left( \begin{array}{c}
Z_{t+1}^+  \\
Z_{t+1}^- \end{array} \right)}{*} 
= \left( \begin{array}{c}
y_+  \\
y_- \end{array} \right) 
= M Z_t.
$

Let $t^+, t^- > 0$. Then, Hoeffding's inequality gives 
\be \Psub{ \left| \s{i=1}{Z_t^{\pm}} \Phi_i^{\pm} - Z_t^{\pm} \frac{\PHItwo}{\PHI} \right| \geq t^{\pm} }{*} \leq 2 \exp \lr{- \frac{2 (t^{\pm})^2}{Z_t^{\pm} \gamma}}, \label{eq::Hoeffding_Psi} \ee 
where $\gamma = (\PHImin - \PHImax)^2$, and where $\Psub{\cdot}{*} = \P{\cdot|Z_t}.$
\\Hence,
\BA
&\Psub{|\lambda^+ - y^+| \leq \frac{\PHI}{2} \lr{a t^+ + b t^-} }{*} 
\\ &\geq  \Psub{ \left| \s{i=1}{Z_t^{+}} \Phi_i^{+} - Z_t^{+} \frac{\PHItwo}{\PHI} \right| \leq  t^{+},  \left| \s{i=1}{Z_t^{-}} \Phi_i^{-} - Z_t^{-} \frac{\PHItwo}{\PHI} \right| \leq  t^{-} }{*} \\
&\geq \lr{1 - 2 \exp \lr{- \frac{2 (t^{+})^2}{Z_t^{+} \gamma}}}
\lr{1 - 2 \exp \lr{- \frac{2 (t^{-})^2}{Z_t^{-} \gamma}}}.
\EA
Plugging $t^+ = \frac{s \sqrt{y^+}}{\sqrt{3} \PHI a}$ and $t^- = \frac{s \sqrt{y^-}}{\sqrt{3} \PHI b}$ into the last equation leads to
\BA
&\Psub{|\lambda^+ - y^+| \leq \frac{s}{2} \| y \|^{1/2}_1 }{*} \\
&\geq \lr{1 - 2 \exp \lr{- \frac{4/3 }{(\PHI)^2 a^2 \gamma  } \frac{y^{+}}{Z_t^{+}} s^2}}
\lr{1 - 2 \exp \lr{- \frac{4/3 }{(\PHI)^2 a^2 \gamma  } \frac{y^{-}}{Z_t^{-}} s^2}} \\
&\geq \lr{1 - 2e^{-c_0 s^2}}^2 \\
&\geq 1 - 4 e^{-c_0 s^2},
\EA
for some constant $c_0 > 0$, since $\frac{y^{\pm}}{Z_t^{\pm}}$ is bounded away from zero by some constant. 

We use the last inequality to obtain
\BA \Psub{  Z_{t+1}^+ -  y^+   > s \| y \|^{1/2}_1 }{*}
&\leq \Psub{\Pois{y^+ + \frac{s}{2} \| y \|^{1/2}_1} - \lr{y^+ + \frac{s}{2} \| y \|^{1/2}_1} > \frac{s}{2} \| y \|^{1/2}_1}{*} \\
&\quad + 4 e^{-c_0 s^2}.
\label{eq::bound_Zt_min_y}
\EA
We continue by invoking \eqref{eq::Poi_tail},
\[ \ba &\Psub{\Pois{y^+ + \frac{s}{2} \| y \|^{1/2}_1} - \lr{y^+ + \frac{s}{2} \| y \|^{1/2}_1} > \frac{s}{2} \| y \|^{1/2}_1}{*} \\
&\leq 2 \exp \lr{-(y^+ + \frac{s}{2} \| y \|^{1/2}_1) \delta \lr{ \frac{\frac{s}{2} \| y \|^{1/2}_1}{y^+ + \frac{s}{2} \| y \|^{1/2}_1} }  }. \ea \]
We note the existence of a $\theta > 0$ such that for all $x \in [0,1]$, $\delta(x) \geq \theta x^2$, so that
\[ (y^+ + \frac{s}{2} \| y \|^{1/2}_1) \delta \lr{ \frac{\frac{s}{2} \| y \|^{1/2}_1}{y^+ + \frac{s}{2} \| y \|^{1/2}_1} } \geq \frac{\theta \frac{s^2}{4}\| y \|_1}{y^+ + \frac{s}{2} \| y \|^{1/2}_1} \geq c_2 (s^2 \wedge s), \]
for some constant $c_2 > 0$, because $y^+ + \frac{s}{2} \| y \|^{1/2}_1 \leq \max \{ 2y^+,  s \| y \|^{1/2}_1 \}$.

 Similarly, to bound  $\Psub{  Z_{t+1}^+ -  y^+   \leq - s \| y \|^{1/2}_1 }{*}$ from above, we need to estimate
\be \ba  &\Psub{\Pois{y^+ - \frac{s}{2} \| y \|^{1/2}_1} - \lr{y^+ - \frac{s}{2} \| y \|^{1/2}_1} \leq  -\frac{s}{2} \| y \|^{1/2}_1}{*} \\ &\leq 2 \exp \lr{-(y^+ - \frac{s}{2} \| y \|^{1/2}_1) \delta \lr{ \frac{\frac{s}{2} \| y \|^{1/2}_1}{y^+ - \frac{s}{2} \| y \|^{1/2}_1} }  },
\label{eq::Poi_Z}
\ea \ee
when $y^+ > \frac{s}{2} \| y \|^{1/2}_1$ (if $y^+ < \frac{s}{2} \| y \|^{1/2}_1$, then $Z_{t+1}^+ - y^+ > -\frac{s}{2} \| y \|^{1/2}_1$, so that \\ $\Psub{  Z_{t+1}^+ -  y^+   \leq - s \| y \|^{1/2}_1 }{*}=0$). 

We distinguish between two cases: Firstly, when $y^+ - \frac{s}{2} \| y \|^{1/2}_1 > \frac{s}{2} \| y \|^{1/2}_1,$ we have
\be (y^+ - \frac{s}{2} \| y \|^{1/2}_1) \delta \lr{ \frac{\frac{s}{2} \| y \|^{1/2}_1}{y^+ - \frac{s}{2} \| y \|^{1/2}_1} } 
\geq \frac{\theta \frac{s^2}{4}\| y \|_1}{y^+ - \frac{s}{2} \| y \|^{1/2}_1} 
\geq 
\theta \frac{\| y \|_1}{y^+} \frac{s^2}{4}
\geq c_3 s^2, 
\label{eq::Delta_arg_smaller}
\ee
for some constant $c_3$, due to our observation above.
\\ 
Secondly, in case  $y^+ - \frac{s}{2} \| y \|^{1/2}_1 < \frac{s}{2} \| y \|^{1/2}_1,$  we use the existence of a $\theta' > 0$  such that for all $x \geq 1$, $\delta(x) \geq \theta' x$:
\be (y^+ - \frac{s}{2} \| y \|^{1/2}_1) \delta \lr{ \frac{\frac{s}{2} \| y \|^{1/2}_1}{y^+ - \frac{s}{2} \| y \|^{1/2}_1} } 
\geq \theta'  \frac{\| y \|^{1/2}_1}{2} s \geq c_4 s, 
\label{eq::Delta_arg_larger}
\ee
for some constant $c_4 > 0$.

 Combining \eqref{eq::bound_Zt_min_y} - \eqref{eq::Delta_arg_larger}, leads to
\BA
\P{  |Z_{t+1}^+ -  y^+|   > s \| y \|^{1/2}_1 } &\leq 2\lr{ e^{-c_2 (s^2 \wedge s)} +  e^{-  c_4 s}  + e^{-c_3 s^2} } + 8 e^{-c_0 s^2} \\
&\leq c_5  e^{-c_6 (s^2 \wedge s)}.
\EA

An identical bound holds (after possibly redefining the values of $c_5$ and $c_6$) for $|Z_{t+1}^- -  y^-|$.

Finally, noting that $\| y \|_1 = \rho \| Z_t \|_2$, we have 
\BA
\P{ \| Z_{t+1}  -  M Z_t  \|_2   > s \| Z_t \|^{1/2}_1 \bigm| \cF_t } 
&\leq \P{  |Z_{t+1}^+ -  y^+|   \geq  \frac{s}{\sqrt{2}} \| Z_t \|^{1/2}_1 \bigm| \cF_t}  \\
&\quad + \P{  |Z_{t+1}^- -  y^-|   \geq  \frac{s}{\sqrt{2}} \| Z_t \|^{1/2}_1 \bigm| \cF_t} \\
&\leq c'  e^{-c (s^2 \wedge s)},
\EA
that is exactly claim \eqref{eq::M_Z_t_dev}. 

We are now in a position to derive a similar bound as $(59)$ in \cite{BoLeMa15}:
\begin{equation}\label{eq:4U}
\P{ \forall t \ge 1:  \| Z_{t+1}  -  M Z_t  \|_2  \leq   u (t+1)  \log n  \| Z_t \|_1^{1/2}   } \geq 1 - c'\sum_{t\geq 1}  e^{-c  u  t \log n} \geq 1 - C' n^{ - C u}.   
\end{equation}
Recalling \eqref{eq::Z_t_recursive}, we have, for $s \geq 1$,
$$
|  \langle \EF_{k} , Z _s \rangle -  \mu^{s-t}_{k} \langle  \EF_k, Z_t  \rangle | \leq \mu_k ^{s-1} \| \EF_k \|_2  \sum_{u=s}^{t-1}   \frac{ \| Z_{u+1}  -  M Z_u  \|_2}{ \mu_k ^{u}}\cdot
$$
From Equation (\ref{eq:4U}) we know that, for all $u \geq 1$, 
\be \| Z_{u+1}  -  M Z_u  \|_2 \leq c_9  (\log n )  ( u+1)  \| Z_u \|_1^{1/2}, \label{eq::Zu} \ee
where $c_9$ is so large that \ref{eq::Zu} holds with probability $1 - n ^{-\beta}$. Further, $\| Z_h \|_1$ itself is bounded by Lemma \ref{lm::23}:
\be \| Z_h \|_1 \leq c_{10}( \log n )  \rho^h, \ee
also with probability at least $1 - n ^{-\beta}$. 

With the same probability,  for $k \in \{1,2\}$,
\BA
|  \langle \EF_{k} , Z _s \rangle -  \mu^{s-t}_{k} \langle  \EF_k, Z_t  \rangle | 
&\leq c_{11} (\log n )^{3/2}  \mu_k ^ {s-1}  \sum_{u=s}^{t-1}  (u+1)  \frac{   \sqrt \rho  }{ \mu_k }^{u} \\
&\leq c_{12}  (\log n )^{3/2} (s+1)  \rho^{s/2}.
\label{eq::g_k_Z_s}
\EA
The proof the last claim, write
\be   \langle \EF_{k} , \Psi_s \rangle -  \mu^{s-t}_{k} \langle  \EF_k, \Psi_t  \rangle  =  \frac{\PHItwo}{\PHI} \lr{  \langle \EF_{k} , Z _s \rangle -  \mu^{s-t}_{k} \langle  \EF_k, Z_t  \rangle } + \epsilon_s -  \mu^{s-t}_{k} \epsilon_t,  \label{eq::g_k_Psi}   \ee
where, for $s \geq 1$,
\[ \epsilon_s = g_k(+) \lr{ \Psi_s(+) - Z_s^+ \frac{\PHItwo}{\PHI} } + g_k(-) \lr{\Psi_s(-) - Z_s^- \frac{\PHItwo}{\PHI}}. \]
We bound $\epsilon_t$ using \eqref{eq::Hoeffding_Psi},
\begin{equation}\label{eq::epsilon_t}
\P{ \forall t \ge 1:  \epsilon_t  \leq   t  \log n  \| Z_t \|_1^{1/2}   } \geq 1 - c_{13} \sum_{t\geq 1}  e^{-c_{14}   t^2 \log^2 n} \geq 1 - C' n^{ - C u}.   
\end{equation}
So that, with probability $1 - n ^{-\beta}$,
 \[|\epsilon_s -  \mu^{s-t}_{k} \epsilon_t| \leq c_{15} \log^{5/2}(n) \lr{\rho^{s/2} + |\mu_k|^{s-t}\rho^{t/2}} \leq c_{16} \log^{5/2}(n) \rho^{s/2},\]
 since $|\mu_k|> \rho^{1/2}.$
\end{proof}

\begin{proof}[Proof of Theorem \ref{th::24_special}]
We have, 
\be \| \Psi_{u+1}  -  M \Psi_u  \|_2 \leq \frac{\PHItwo}{\PHI} \| Z_{u+1}  -  M Z_u  \|_2 + \| \Psi_{u+1}  -  \frac{\PHItwo}{\PHI} Z_{u+1}  \|_2 + \|  M \lr{ \Psi_u -\frac{\PHItwo}{\PHI} Z_u  }  \|_2. 
\label{eq::bound_1}\ee
We use \eqref{eq::Hoeffding_Psi}, to obtain that for any $\beta > 0$ (similar to \eqref{eq:4U})

\begin{equation}
\P{ \forall t \ge 1:  \| \Psi_{t}  -  \frac{\PHItwo}{\PHI} Z_t  \|_2  \leq   t \log n  \| Z_t \|_1^{1/2}   } \geq  1 -  n^{ - \beta}.   
\label{eq::bound_2}
\end{equation}
Combing \eqref{eq::bound_1}, \eqref{eq::bound_2} and \eqref{eq:4U}, gives that with probability $1 -  n^{ - \beta}$, for all $u \geq 1$,
\be \| \Psi_{u+1}  -  M \Psi_u  \|_2 \leq c_2 u \log n \| Z_u \|_1^{1/2}. \label{eq::543} \ee
We can now apply the argument at the end of Theorem $24$ in \cite{BoLeMa15}. 
The second claim follows by using the last part of the proof of Theorem $24$ in \cite{BoLeMa15}, where the variable $U$ needs to be replaced by
\[ U = \sup_{t \geq 1} \frac{\| \Psi_{t+1}  -  M \Psi_t  \|_2}{t \| Z_t \|_1^{1/2}}.  \]
It is important here that $\E{U^4} = \bigO(1)$, which is ensured by \eqref{eq::543}.
\end{proof}

\begin{proof}[Proof of Theorem \ref{th::25}]
We start by calculating the expectation and variance of $\sum_{ u \in Y^o_{t}} L_{k,\ell}^u$ conditional on $\cF_t$ (defined in Theorem \ref{th::21}). We use this to show that, as $\ell \to \infty$, uniformly for all $\psi_o$,
\be
\frac{\bar Q_{k,\ell}}{\mu_k^{2 \ell}} \to \frac{ \PHIthree}{\PHItwo} \frac{\rho}{\mu_k^2 - \rho} \mu_{k,\psi_o} Y_{k,\psi_o}(\infty),
\label{eq::Conv_Q}
\ee
almost surely and in $L^2$, where $Y_{k,\psi_o}(\infty)$ is defined in Corollary \ref{CO::21}, and where 
\[ \bar Q_{k,\ell}  = \sum_{t=0} ^{\ell-1}   \mathbb{E}_{\cF_t}  \sum_{ u \in Y^o_{t}} L_{k,\ell}^u.  \]
The latter is reminiscent of 
\[ Q_{k,\ell} =  \sum_{t=0} ^{\ell -1} \sum_{ u \in Y^o_{t}} L_{k,\ell}^u,\]
and we show that $\bar Q_{k,\ell}$ and $Q_{k,\ell}$ are in fact close in $L^2$-distance:
\[ \| \bar Q_{k,\ell} - Q_{k,\ell} \| = o(\mu_k^{2\ell}).  \]
Consider for $t \geq 0$ and $\ell \geq t + 2$, 
\BA
\mathbb{E}_{\cF_t, Y_t^o, Y_1^u} L_{k,\ell}^u  &= 
\sum_{ (v,w) \in Y_1 ^u , v \ne w} \mathbb{E}_{\cF_t, Y_t^o, Y_1^u} S^w_{\ell - t -1} 
\mathbb{E}_{\cF_t, Y_t^o, Y_1^u}  \langle \EF_k , \Psi^v_ t  \rangle \\
&= \sum_{ (v,w) \in Y_1 ^u , v \ne w} \rho_w \rho^{\ell - t -2}  \phi_v \mu_k^t \langle \EF_k , Z^v_0  \rangle
\label{eq::L_u}
\EA
where $\rho_w = \frac{a+b}{2} \PHI \phi_w$, with $\phi_w$  a random variable that follows law $\NUstar$. The second equality in \eqref{eq::L_u} follows after calculating
\[ \E{\Psi^v_ t | Y^v_0} = \frac{\PHItwo}{\PHI}  \E{Z^v_ t | Y^v_0} = \frac{\PHItwo}{\PHI}  \frac{\PHI \phi_v}{\PHItwo} M^t Z_0^v = \phi_v M^t Z_0^v , \]
where the factor $\frac{\PHI \phi_v}{\PHItwo}$ accounts for the fact that the "parental" vertex $v$ has \emph{deterministic} type $\phi_v$ (and transitions are thus given by $M_{\phi_v} =  \frac{\PHI \phi_v}{\PHItwo} M$), whereas vertices in the later generations have i.i.d. weights (for which $M$ is the transition matrix).
Now, 
\BA
\mathbb{E}_{\cF_t, Y_t^o} L_{k,\ell}^u 
&= \mathbb{E}_{\cF_t, Y_t^o} \mathbb{E}_{\cF_t, Y_t^o, Y_1^u} L_{k,\ell}^u \\
&= \mathbb{E}_{\cF_t, Y_t^o} \sum_{ (v,w) \in Y_1 ^u , v \ne w} \rho_w \rho^{\ell - t -2}  \phi_v \mu_k^t \langle \EF_k , Z^v_0  \rangle \\
&= \rho^{\ell - t -2} \mu_k^t \mathbb{E}_{\cF_t, Y_t^o} |Y_1^u|(|Y_1^u|-1) \mathbb{E}_{\cF_t, Y_t^o} \rho^* \mathbb{E}_{\cF_t, Y_t^o}   \phi^*  \langle \EF_k , \begin{pmatrix} 
\IND{\sigma^* = +}  \\
\IND{\sigma^* = -}  
\end{pmatrix} \rangle, 
\label{eq::L_u_children_first}
\EA
where $\phi^*$ has law $\NUstar$, $\rho^*$ is an i.i.d. copy of $\frac{a+b}{2} \PHI \phi^*$ and $\sigma^* = \sigma_u$ with probability $\frac{a}{a+b}$, and $\sigma^* = -\sigma_u$ with probability $\frac{b}{a+b}$ (further, $\rho^*, \phi^*$ and $\sigma^*$ are independent).  

We thus have
\be \mathbb{E}_{\cF_t, Y_t^o} L_{k,\ell}^u =  \rho^{\ell - t -2} \mu_k^t  \cdot \rho_u^2 \cdot \rho \cdot \frac{\PHItwo}{\PHI} \cdot (\EF_k(1) c(\sigma_u,+) + \EF_k(2)c(\sigma_u,-)),
\label{eq::L_u_children}
\ee
where $\rho_u = \frac{a+b}{2} \PHI \phi_u$ (with $\phi_u$ the weight of $u$) and for $(x,y) \in \spm \times \spm,$ $c(x,y) = \frac{a}{a+b}$ if $x=y$ and  $c(x,y) = \frac{b}{a+b}$ otherwise. 

Now, as $\EF_k$ is an eigenvector of $M$ with eigenvalue $\mu_k$, we have
\[ (\EF_k(1) c(\sigma_u,+) + \EF_k(2)c(\sigma_u,-)) = \frac{2}{a+b} \frac{\mu_k}{\PHItwo} \langle \EF_k , Z^u_ 0  \rangle = \frac{\mu_k}{\rho} \langle \EF_k,Z^u_ 0 \rangle. \]
Together with \eqref{eq::L_u_children_first} this gives
\be \mathbb{E}_{\cF_t, Y_t^0} L_{k,\ell}^u =  \rho^{\ell - t -2} \mu_k^{t+1}  \rho_u^2 \frac{\PHItwo}{\PHI} \langle \EF_k,Z^u_0 \rangle.
\label{eq::L_u_children}
\ee
Summing over $u \in Y^o_{t}$ using the last equation yields
\BA
\mathbb{E}_{\cF_t} \sum_{ u \in Y^o_{t}} L_{k,\ell}^u 
&= \mathbb{E}_{\cF_t} \sum_{ u \in Y^o_{t}} \mathbb{E}_{\cF_t, Y_t^0} L_{k,\ell}^u \\
&= \rho^{\ell - t -2} \mu_k^{t+1} \frac{\PHItwo}{\PHI} \mathbb{E}_{\cF_t}  \sum_{ u \in Y^o_{t}}  \rho_u^2 \langle \EF_k,Z^u_0 \rangle \\
&= \rho^{\ell - t -2} \mu_k^{t+1}  \langle \EF_k,Z_t \rangle \lr{\frac{a+b}{2}}^2 \PHItwo \cdot \left\{ \begin{array}{ll}
          \psi_o \PHItwo  & \mbox{if $t =  0$};\\
           \PHIthree  & \mbox{if $t > 0$.}\end{array} \right. \\
\label{eq::66}  
\EA

We leave it to the reader to verify that the same inequality holds for $l = t+1$.

We continue by bounding the variance of $L_{k,\ell}^u$: 
\BA
\VAR_{\cF_t} L_{k,\ell}^u
&\leq \mathbb{E}_{\cF_t} ( L_{k,\ell}^u ) ^2 \\
&= \mathbb{E}_{\cF_t} \sum_{ (v,w) \in Y_1 ^u , v \ne w} \sum_{ (v',w') \in Y_1 ^u , v' \ne w'} S^w_{\ell - t -1} S^{w'}_{\ell - t -1} \langle \EF_k,\Psi_t^v \rangle \langle \EF_k,\Psi_t^{v'} \rangle \\
&\leq \mathbb{E}_{\cF_t} |Y_1^u|^2 \mathbb{E}_{\infty} S^2_{\ell - t -1} \mathbb{E}_{\infty} \langle \EF_k,\Psi_t^v \rangle^2,
\label{eq::Var_L_u}
\EA
where $\mathbb{E}_{\infty} [ \cdot ] = \max_{\tau' \in \spm}  \mathbb{E}[ \cdot | \phi_o = \bdphi , \sigma_o = \tau' ]$.
Now, $\mathbb{E}_{\cF_t} |Y_1^u|^2 \leq c_0$. 
From Lemma \ref{lm::23}, we know that $S_k \overset{d} \leq \text{Exp} \lr{c_1 \rho^k}$, hence $\mathbb{E}_{\infty} S^2_{\ell - t -1} \leq 2 c_1^2 \lr{\rho^{\ell - t -1} }^2.$ To bound $\mathbb{E}_{\infty} \langle \EF_k,\Psi_t^v \rangle^2$, recall from Theorem \ref{th::21_other_martingale} that
\[ \E{ \left.  \lr{\frac{\langle \phi_{k} , \Psi_t \rangle}{\mu^{t-1}_k} - \langle  \EF_k, \Psi_1  \rangle }^2 \right| Z_1 } \leq C_2 \| Z_1 \|_1.  \]
Consequently, as $\E{\| Z_1 \|_1}$ is bounded,
\[ \mathbb{E}_{\infty} \langle \EF_k,\Psi_t^v \rangle^2 \leq c_3\mu_k^{2t}. \]

 Returning to \eqref{eq::Var_L_u}, we have
\be \Var_{\cF_t} \sum_{ u \in Y^o_{t}} L_{k,\ell}^u \leq c_4\mu_k^{2t} \rho^{2(\ell - t)} S_t. \label{eq::Var_sum_Lklu} \ee
We have
\BA
\bar Q_{k,\ell} 
&= \sum_{t=0} ^{\ell-1}   \mathbb{E}_{\cF_t}  \sum_{ u \in Y^o_{t}} L_{k,\ell}^u \\
&= \rho^{\ell} \mu_k \langle \EF_k,Z_0 \rangle   \psi_o  
+ 
\sum_{t=1} ^{\ell-1}  \rho^{\ell - t} \mu_k^{t+1} \langle \EF_k,Z_t \rangle \frac{  \PHIthree}{\PHItwo} \\
&=  \rho^{\ell} \mu_k \langle \EF_k,Z_0 \rangle   \psi_o
+ 
\frac{\PHIthree}{\PHItwo} \sum_{t=1} ^{\ell-1}  \rho^{\ell - t} \mu_k^{2t} \mu_{k,\psi_o} Y_{k,\psi_o}(t), 
\label{eq::Q_bar}
\EA
where $Y_{k,\psi_o}(t)$ is defined in Corollary \ref{CO::21}.

We consider 
\BA
\frac{\bar Q_{k,\ell}}{\mu_k^{2 \ell}} = o(1) + \frac{  \PHIthree}{\PHItwo} \sum_{t=1} ^{\ell-1}  \lr{\frac{\mu_k^2}{\rho}}^{t-\ell} \mu_{k,\psi_o} Y_{k,\psi_o}(t),
\EA
and verify our claim \eqref{eq::Conv_Q}. 
To do so, split for arbitrary \emph{fixed} $\epsilon > 0$,
\[ \s{t=1}{\ell - 1} r^{t - \ell} Y_k(t) =  \s{t=1}{T_{\epsilon}-1} r^{t - \ell} Y_k(t) + \s{t=T_{\epsilon}}{\ell - 1} r^{t - \ell} Y_k(t),\]
where $r = \frac{\mu_k^2}{\rho}$, $Y_k$ is shorthand notation for $Y_{k,\psi_o}$, and where
\[T_{\epsilon} = \min \{ t: \forall s \geq t, | Y_k(\infty) - Y_k(s) | \leq \epsilon \}. \]
Then,  
\[ \s{t=1}{T_{\epsilon}-1} r^{t - \ell} Y_k(t) \leq | \sup_t Y_k(t) | r^{-\ell} r^{T_{\epsilon}} T_{\epsilon} \overset{\text{a.s.}} \to 0, \]
as $\ell \to \infty$, since $(Y_k(t))_t$ is convergent (uniformly in $\psi_o$) and hence bounded. Further,
\BA \s{t=T_{\epsilon}}{\ell - 1} r^{t - \ell} Y_k(t) 
&= \s{u=1}{\ell - T_{\epsilon}} (Y_k(\infty) + \bigO(\epsilon)) \\
&\overset{\text{a.s.}}{\to} \s{u=1}{\infty} r^{-u}  (Y_k(\infty) + \bigO(\epsilon)) \\
&= \frac{1}{r-1} (Y_k(\infty) + \bigO(\epsilon)),
\EA
where the limit is taken for $\ell \to \infty$. Since $\epsilon > 0$ was arbitrary, \eqref{eq::Conv_Q} follows.

$L^2$-convergence follows from \cite{BoLeMa15}  (this convergence takes place uniformly for all $\psi_o$  due to Theorem \ref{th::21}). 

Further, that $\| \bar Q_{k,\ell} - Q_{k,\ell} \| = o(\mu_k^{2\ell})$ can be established by following the proof in \cite{BoLeMa15}. Indeed, from the latter proof we know that, for some constant $c_6$ independent of $\psi_o$,
\BA
\| Q_{k,\ell} - \bar Q_{k,\ell} \|_2 & 
\leq \sum_{t=0}^{\ell-1} \NRM{ \lr{ \VAR_{\cF_t} \lr{\sum_{ u \in Y^o_{t}} L_{k,\ell}^u}}^{1/2} }_2 \\
&\leq c_5 \sum_{t=0}^{\ell}  \mu_k ^{t} \rho^{\ell - t} \| \sqrt {S_t}\|_2  \\
&\leq c_6 \mu_k ^{ \ell} \rho^{\ell/2}, 
\label{eq::Q_min_Q_bar}
\EA
due to the variance bound in \eqref{eq::Var_sum_Lklu} and Lemma \ref{lm::23}.

Finally, combining the uniform bounds \eqref{eq::Conv_Q} and \eqref{eq::Q_min_Q_bar}, entails that
\[ \NRM{ \frac{ Q_{k,\ell}}{\mu_k^{2 \ell}} - \frac{  \PHIthree}{\PHItwo} \frac{\rho}{\mu_k^2 - \rho} \mu_{k,\psi_o} Y_{k,\psi_o}(\infty) }_2 \to 0, \]
uniformly for all $\psi_o$.
\end{proof}

\begin{proof}[Proof of Theorem \ref{th::25_special}]
Using \eqref{eq::Var_L_u} and Theorem \ref{th::24_special}, we have 
$$ \VAR_{\cF_t} L_{k,\ell}^u \leq c_1 \rho^{2(\ell -t)} t^3 \rho^t. $$
Plugging this bound, together with \eqref{eq::66} here, into $(66)$ in \cite{BoLeMa15} establishes the claim. 
\end{proof}

\begin{proof}[Proof of Theorem \ref{th::28}]
Recall the explicit expressions for $Q_{1,\ell}$ and $Q_{2,\ell}$ from \eqref{eq::Q1_l}, respectively \eqref{eq::Q2_l}. Now, conditional on $\cT$ and the weights (denoted by $\cT_{ \phi}$), $\cP_{2 \ell + 1}$ is deterministic, hence
\[ \E{Q_{1,\ell}Q_{2,\ell}| \cT, \cT_{ \phi}} =  Q_{1,\ell} \sum_{(u_0,\ldots,u_{2\ell+1})\in{\mathcal P}_{2\ell+1}} \phi_{u_{2 \ell +1}} \E{\sigma (u_{2 \ell +1})| \cT} = 0,\]
because,  $\E{\sigma (u)|\cT, \sigma_o} = \lr{\frac{a-b}{a+b}}^{|u|} \sigma_o$, for a vertex $u$ at distance $|u|$ from the root, by construction of the branching process.
\end{proof}

\section{Proofs of Section \ref{sec::coupling}}
\label{App::coupling}
\begin{proof}[Proof of Proposition \ref{prop::31}]
The second statement follows from the first after recalling that $(G,e)_\ell = (G',e_2)_\ell$, where $G'$ is the graph $G$ with edge $\{e_1,e_2\}$ removed. Since $e \in \vec E$, $e_2$ then has a biased weight governed by $\NUstar$. 

In  \cite{GuLeMa15}, we established a coupling between the branching process and the DC-SBM where the spins are drawn \emph{uniformly} from $\spm$, with error probability $n^{- \frac{1}{2} \log (4/e)}$.

Thus, we are done if we couple the neighbourhoods in the latter graph to the DC-SBM with \emph{deterministic} spins under consideration here.

Now, with probability at least $1 - e^{- \Omega( n^{-1/2} )}$ we can couple the graphs such that at most $c_1 n^{\frac{3}{4} \vee (1-\gamma) }$ have unequal spins (call the corresponding set of vertices $S$) and \emph{all} weights are equal. Further, we may assume that the subgraphs obtained after removing $S$ are identical.

The $\ell$-neighbourhoods in both graphs are exactly the same if they are both disjoint with $S$. Conditional on $|S|$ and $|G_\ell|$, this happens with probability at least $1 - c_2 \frac{|G_\ell||S|}{n}$.

From \cite{GuLeMa15}, we know that with probability $1 - n^{-\log (4/e)},$
$|G_\ell| < n^{\frac{1}{8} \wedge \frac{\gamma}{2}}$.

Thus, conditional on the bounds for $|S|$ and $|G_\ell|$, the neighbourhoods are the same with probability at least $1 - c_3 n^{-(\frac{1}{8} \vee \frac{\gamma}{2})}.$

All together, $\P{(G,v)_\ell =(T,o)_\ell} \geq 1 - c_4 n^{-\lr{ \frac{1}{8} \wedge \frac{\gamma}{2}} \wedge \lr{\frac{1}{2}\log (4/e)} }.$  
\end{proof}

\begin{proof}[Proof of Corollary \ref{cor::32}]
This proof follows the proof of Corollary 32 in \cite{BoLeMa15}. Indeed (although with a slightly different probability) the graph neighbourhood $(Y_t(e) )_{0 \leq t \leq \ell}$ and branching process $(Z_t)_{0 \leq t \leq \ell}$ coincide again, and moreover, the weights are equal in both processes.  
\end{proof}

\begin{proof}[Proof of Lemma \ref{lm::29}]
As observed in \cite{BoLeMa15}, the second statement follows from the first.

Adapting our paper \cite{GuLeMa15}, at step $m$ in the exploration process, the weights of the vertices in $\cU(m)$ are independent, and those with spin $\tau$ have weight governed by $\nu^{(m)}_\tau$, where
\[ \mathrm{d} \nu^{(m)}_\tau (\psi) = \frac{g_\tau(\psi)}{\int_{\PHImin}^{\PHImax} g_\tau(\psi') \mathrm{d} \nu(\psi') } \mathrm{d} \nu(\psi), \] 
where $g_{\tau}(\cdot)  = \prod_{i=1}^m \lr{1 - \frac{\kappa(x_i, \tau \cdot)}{n}},$ with $x_u = \sigma_u \phi_u$ the types of the already explored vertices and $\kappa(x,y) = |xy|(\indicator{xy>0} a + \indicator{xy<0} b)$.

We claim that variables following $\nu^{(m)}_\tau$ are stochastically dominated by variables governed by $\nu$. Indeed,   
  use that for any non-decreasing $f,h: \mathbb{R} \to \mathbb{R}$ and any random variable $X$ we have $\E{f(X)h(X)} \geq \E{f(X)}\E{h(X)}$. Then, for $\psi \geq 0$, 
\[ \nu_\tau^{(m)}([0,\psi]) = \frac{\E{-g_\tau(\phi) \cdot - \IND{\phi \leq \psi}}}{\E{g_\tau(\phi)}} \geq \frac{\E{g_\tau(\phi)} \E{ \IND{\phi \leq \psi}}}{\E{g_\tau(\phi)}} = \nu([0,\psi]),\]
with $\phi \sim \nu$.

Secondly, we claim that the weight of a vertex when it is just discovered is stochastically dominated by variables governed by $\NUstar$. To prove this, let $m \geq 0$ and assume the claim to hold for all $l \leq m$. Consider vertex $v$ explored in step $m+1$ (itself discovered in step, say, $l \leq m$) with weight $\phi_v^{*(l)}$. Its children are selected from the set $\mathcal{U}^{(m)}$ in which they have independent weights $(\phi_u^{(m)})_{u \in \mathcal{U}^{(m)}}$ all stochastically dominated by $\nu$. 
We compare this to a setting $\mathcal{S}$ where a particle with weight $\phi^* \sim \NUstar$ has its children selected following the same rules from a reservoir of $|\mathcal{U}^{(m)}|$ particles with  spins as in $\cU^{(m)}$ and i.i.d. weights 
$(\phi_u)_{u \in \mathcal{U}^{(m)}} \sim \nu$. 
Due to the assumed stochastic domination, there exists a coupling of the exploration process and the setting $\mathcal{S}$, such that pointwise $\phi_v^{*(l)} \leq \phi^*$ and $\phi_u^{(m)} \leq \phi_u$ for all $u$. 
To decide whether $u \in \mathcal{U}^{(m)}$ is selected as a child, we can draw uniformly from $[0,1]$ a number $U_u$ and include $u$ in the exploration process exactly when 
$\frac{(\IND{\sigma_u = \sigma_v }a + \IND{\sigma_u = -\sigma_v }b )\phi_v^{*(l)} \phi_u^{(m)}}{n} \geq U_u$
and in the setting $\cS$ exactly when $\frac{(\IND{\sigma_u = \sigma_v }a + \IND{\sigma_u = -\sigma_v }b )\phi^{*} \phi_u}{n} \geq U_u.$
Since by assumption $\phi^{*} \phi_u \geq \phi_v^{*(l)} \phi_u^{(m)}$, for each $u$, we conclude that the newly selected particles are also stochastically dominated.

Denote the vertices in $S_t$ by $1, \ldots, S_t$ and their weights by 
$(\widehat{\phi}_v^*)_{v \in S_t}$.
We shall use the same strategy as in Lemma \ref{lm::23} to bound
\[ S_{t+1} = \s{v=1}{S_t} \widehat{D}_v^*, \]
where
$\widehat{D}_v^*$ is the offspring-size of $v$. In particular, to use large deviation theory as in \eqref{eq::15}, we shall calculate for $\theta \geq 0$, 
$ \E{ \left. e^{\theta \s{v=1}{S_t} \widehat{D}_v^*} \right| S_t }. $
Caution is needed here as the variables $(\widehat{D}_v^*)_{v \in S_t}$ are \emph{not} independent. Let $\mathcal{F}_m$ be the sigma-algebra generated by the exploration process upto step $m$ (included). If vertex $v$ is explored in step $m+1$, then,  
\[ \widehat{D}_v^* = \sum_{u \in \mathcal{U}^{(m)}}  \text{Ber} \lr{ (\IND{\sigma_u = \sigma_v }a + \IND{\sigma_u = -\sigma_v }b )\frac{\widehat{\phi}_v^{*} \phi_u^{(m)}}{n}  }, \]
where we recall that conditioned on $\mathcal{F}_m$, $\phi_u^{(m)}$ is stochastically dominated by $\nu$ and $\widehat{\phi}_v^{*}$ by $\nu^*$. Hence, using that $1 + y \leq e^y$ for all $y \in \mathbb{R}$,
\BA
\E{\left.e^{\theta \widehat{D}_v^*} \right|\mathcal{F}_m,\widehat{\phi}_v^{*}} 
&\leq \E{ \left. \prod_{u} \lr{1 + \frac{\widehat{\phi}_v^{*} \phi_u^{(m)}}{n} (\IND{\sigma_u = \sigma_v }a + \IND{\sigma_u = -\sigma_v }b ) (e^{\theta} -1)} \right|\mathcal{F}_m,\widehat{\phi}_v^{*}}  \\
&\leq  \lr{1 + a\frac{\widehat{\phi}_v^{*} \PHI}{n} (e^{\theta} -1)}^{n_{\sigma_v}} 
\lr{1 + b\frac{\widehat{\phi}_v^{*} \PHI}{n} (e^{\theta} -1)}^{n_{-\sigma_v}} \\
&\leq e^{r_n \widehat{\phi}_v^{*}\PHI  (e^{\theta} -1)},
\EA
where $r_n = \max \{ \frac{n_{+}a + n_{-}b}{n}, \frac{n_{-}a + n_{+}b}{n} \}$. Thus, if $\phi^{*}$ has law $\NUstar$,
\be \E{\left.e^{\theta \widehat{D}_v^*} \right|\mathcal{F}_m} \leq \E{e^{r_n \phi^{*}\PHI  (e^{\theta} -1)}}, \label{eq::Size_EP_Individual}\ee
since for $t \geq 0$, $\E{e^{tX}} \leq \E{e^{tY}}$ if $X \overset{d} \leq Y.$
Iterating \eqref{eq::Size_EP_Individual}, we obtain 
\[ \E{ \left. e^{\theta \s{v=1}{S_t} \widehat{D}_v^*} \right| S_t} \leq  \lr{\E{e^{r_n \phi^{*}\PHI  (e^{\theta} -1)}}}^{S_t} = \E{\left. e^{r_n \s{v=1}{S_t} \phi_v^{*}\PHI  (e^{\theta} -1)}\right| S_t}, \]
where $\{\phi_v^{*}\}_v$ are i.i.d. with law $\NUstar$. Thus, we have
\[ \E{e^{\theta \s{v=1}{S_t} \widehat{D}_v^*}} \leq \E{e^{\theta \text{Poi} \lr{\s{v=1}{S_t} r_n \phi_v^{*}\PHI} }}, \]
compare this to \eqref{eq::S_k_plus_1}: the characteristic function of $\s{v=1}{S_t} \widehat{D}_v^*$ is dominated by the characteristic function of the Poisson-mixture in $\eqref{eq::S_k_plus_1}$ if we replace $\frac{a+b}{2}$ with $r_n$. Hence we can repeat the proof of Lemma \ref{lm::23}, with  $\rho_n:= r_n \PHItwo$ instead of $\rho$.  
 \end{proof}

\begin{proof}[Proof of Lemma \ref{lm::30}]
Fix a vertex $v$. Let $m \geq 0$ be the smallest integer such that all vertices within distance $R$ of $v$ have been revealed at step $m$ of the exploration process. Now, the exploration process constructs a spanning tree $\mathcal{T}_m$ for $G_R(v)$. However, edges between vertices in $\partial G_r$ $(r \leq \ell)$ are not inspected, and neither is it verified whether two vertices in $\partial G_r$ share a common neighbour in $\partial G_{r+1}$ $(r \leq R-1)$. The number of those uninspected edges is bounded by $|G_r|^2$. Hence, among them at most $\text{Bin}(|G_r|^2, \frac{c_1}{n})$ are actually present in $G_r$. Thus, using twice Markov's inequality in conjunction with Lemma \ref{lm::29}, for some $c_2 > 0$,
\[ \P{G_r(v) \text{ is not a tree}} \leq \E{|G_r|^2} \frac{c_1}{n} \leq \frac{c_3 \rho^{2\ell}}{n},  \]
and, 
\[ \P{\sum_v \IND{G_r(v) \text{ is not a tree}} \geq \rho^{2\ell} \log(n)} \leq \frac{c_4}{\log(n)}.  \]

For the other claim, if the graph is tangled, then there is a vertex such that among its uninspected edges in the exploration process at step $m$, at least two are in fact present. Now,
\[ \P{ \text{Bin} \lr{ |G_r|^2, \frac{c_1}{n} } \geq 2 } \leq \lr{\frac{c_1}{n}}^4 \E{|G_r|^4} \leq \frac{c_5 \rho^{4\ell}}{n^4}. \]
A union bound over all vertices then gives
\[ \P{G \text{ tangled}} \leq \frac{c_6 \rho^{4\ell}}{n^3} = o(1). \]
\end{proof}

\begin{proof}[Proof of Proposition \ref{prop::33}]
$(i)$ follows from Lemma \ref{lm::29} and Corollary \ref{cor::32}.

To prove $(ii)$, recall that $B^r_{\vec e \vec g}$ is the number of non-backtracking paths of length $r$ (i.e., containing $r+1$ edges) between $\vec e$ and $\vec g$. Further, if $G_r(e_2)$ is a tree, then there is exactly one path between $e$ and any edge $g$ on the tree. Hence
\[ \langle B^r \chi_k , \delta_e \rangle  = \langle \EF_k , \Psi_r(e) \rangle. \]
An appeal to Corollary \ref{cor::32} then establishes $(ii)$.

Further, $(iii)$ follows from the fact that $G$ is $\ell$-tangle-free with high probability, so that there are at most two non-backtracking walks of length $r$ between any edges $\vec e$ and $\vec f$.
Thus,
\[|\langle B^r \chi_k , \delta_e \rangle |   \leq  2 \|\EF_k\|_{\infty} \PHImax S_t(e) \leq  \log^2 (n)   \rho^{r},
\] 
with probability at least $1 - e^{-\Omega(n)}$, due to Lemma \ref{lm::29}.
\end{proof}

\begin{proof}[Proof of Corollary \ref{cor::34}]
We start with the case $\mu_2^2 > \rho$. Using that $\langle B^\ell \chi_k , x \rangle = 0$ and Proposition \ref{prop::33} (iii),  we write,
\BA
| \langle B^r \chi_k , x \rangle  |
&=  | \sum_{e \in \vec E_{\ell} } x_e \langle B^r \chi_k , \delta_e \rangle +  \sum_{e \notin \vec E_{\ell} } x_e \langle B^r \chi_k , \delta_e \rangle \\
&\quad - \mu_k^{r-\ell} \lr{ \sum_{e \in \vec E_{\ell} } x_e \langle B^\ell \chi_k , \delta_e \rangle +  \sum_{e \notin \vec E_{\ell} } x_e \langle B^\ell \chi_k , \delta_e \rangle} | \\
&\leq (\log n )^2  \rho^{r} \sqrt{|\vec E_\ell|} + \sum_{e \notin \vec E_{\ell} } |x_e| \| \langle B^r \chi_k , \delta_e \rangle - \mu_k^{r-\ell} \langle B^\ell \chi_k , \delta_e \rangle \| \\ 
&\quad+ \mu_k^{r-\ell} \log(n)^2 \rho^\ell \sqrt{|\vec E_\ell|}.
\EA
Now, $|\mu_k| > 1$ and for $e \notin \vec E_{\ell}$, bound $(ii)$ in Proposition \ref{prop::33} applies, so that w.h.p.
\BA
| \langle B^r \chi_k , x \rangle | 
&\leq 2 \rho^{\ell} (\log n )^2  \sqrt{|\vec E_\ell|} +  \rho^{r/2} (\log n )^4  \sqrt{|E|} \\
&\leq  \rho^{\ell} (\log n )^3 n^{\frac{1}{2} - \frac{\gamma}{4} \wedge \frac{1}{80}} + \rho^{r/2} (\log n )^{\frac{9}{2}} n^{\frac{1}{2}} \\
&\leq   \rho^{r/2}  (\log n )^5 n^{1/2},
\EA
since $\rho^\ell = n^{C}\ll n^{\frac{\gamma}{4} \wedge \frac{1}{80}}.$

In case $\mu_2^2 \leq \rho$, redefine  $\vec E_{\ell}$  as the set of oriented edges such that  $(G,e_2)_{\ell}$ is not a tree or $|  \langle \EF_1 , \Psi_t (e) \rangle - \rho^{t - \ell}  \langle \EF_1 , \Psi_\ell (e) \rangle | > (\log n)^4 \rho^{t/2}$ or $|  \langle \EF_2 , \Psi_t (e) \rangle | > (\log n)^4 \rho^{t/2}$. Note that $|\vec E_{\ell}|$ can now by bounded with the same arguments as in the proof of Corollary \ref{cor::32}.

 Write $\langle B^r \chi_k , x \rangle = \sum_{e \in \vec E_{\ell} } x_e \langle B^r \chi_k , \delta_e \rangle +  \sum_{e \notin \vec E_{\ell} } x_e \langle B^r \chi_k , \delta_e \rangle$,
  To bound the sum over $E_{\ell}$, use Cauchy-Schwartz inequality and Proposition \ref{prop::33} $(iii)$, which also holds if $\mu_2^2 \leq \rho$.  For the second sum, use that, if $e \notin \vec E_{\ell}$, then 
$|\langle B^r \chi_k , \delta_e \rangle| \leq (\log n)^4 \rho^{r/2}$, as follows from Theorem \ref{th::24_special} and the coupling result for local neighbourhoods.
\end{proof}

\section{Proofs of Section $7$} 
\label{App::law_large_numbers}
\begin{proof}[Proof of Proposition \ref{prop::35}]
We start by using the law of total variance for $Y = \sum_{v=1}^n  \tau ( G,v)$:
\[ \Var \lr{Y} = \E{\Var\lr{Y|\phi_1, \ldots, \phi_n}} + \Var \lr{ \E{Y | \phi_1, \ldots, \phi_n} }, \]
and shall apply Efron-Stein's inequality on both terms. 

Define the function $h$ for $(\psi_1, \ldots, \psi_n) \in [\PHImin,\PHImax]^n $ as \\
$ h(\psi_1, \ldots, \psi_n) = \E{Y | \phi_1 = \psi_1, \ldots, \phi_n =  \psi_n}. $
We need to bound \\
$ |h(\psi_1, \ldots \psi_{k-1}, \psi_k, \psi_{k+1},\ldots, \psi_n) - h(\psi_1, \ldots \psi_{k-1}, \psi'_k, \psi_{k+1},\ldots, \psi_n)|^2 $
for arbitrary $\psi'_k \in [\PHImin,\PHImax]$. 
Denote by $G_{\psi_1, \ldots, \psi_k,\ldots, \psi_n}$ the random graph $G$, conditional on $\phi_1 = \psi_1, \ldots, \phi_n =  \psi_n$. Assume without loss of generality that $\psi_k \geq \psi'_k$. Then, there exists a coupling of $G_{\psi_1, \ldots, \psi_k,\ldots, \psi_n}$ and $G_{\psi_1, \ldots, \psi'_k,\ldots, \psi_n}$ such that $G_{\psi_1, \ldots, \psi'_k,\ldots, \psi_n}$ is a subgraph of $G_{\psi_1, \ldots, \psi_k,\ldots, \psi_n}$ obtained after removing some edges between $k$ and its neighbours in the latter graph. For this coupling, 
$ |\tau(G_{\psi_1, \ldots, \psi_k,\ldots, \psi_n},u) - \tau(G_{\psi_1, \ldots, \psi'_k,\ldots, \psi_n},u)|  $
is nonzero only if $u \in V(G_{\psi_1, \ldots, \psi_k,\ldots, \psi_n},k)_{\ell}$, and it is bounded by $\max_{v} \varphi( G_{\psi_1, \ldots, \psi_k,\ldots, \psi_n},v ) + \max_{v} \varphi( G_{\psi_1, \ldots, \psi'_k,\ldots, \psi_n},v )$.
Consequently,
\BA &|h(\psi_1, \ldots \psi_{k-1}, \psi_k, \psi_{k+1},\ldots, \psi_n) - h(\psi_1, \ldots \psi_{k-1}, \psi'_k, \psi_{k+1},\ldots, \psi_n)|^2  \\
&\leq \E{ | V(G_{\psi_1, \ldots, \psi_k,\ldots, \psi_n},k)_{\ell} | \lr{\max_{v} \varphi( G_{\psi_1, \ldots, \psi_k,\ldots, \psi_n},v )
+ \max_{v} \varphi( G_{\psi_1, \ldots, \psi'_k,\ldots, \psi_n},v }}^2 \\
&\leq \E{ |V(G_{k, \infty}, k)|^2 | \phi_1 = \psi_1, \ldots, \phi_{k-1} = \psi_{k-1}, \phi_{k+1} =  \psi_{k+1},\ldots, \phi_n = \psi_n  } \\
&\quad\quad \cdot 3  \E{ \left.  \max_{v}  \varphi^2( G,v ) \right| \phi_1 = \psi_1, \ldots, \phi_k = \psi_k, \ldots,  \phi_n = \psi_n} \\
&\quad + \E{ |V(G_{k, \infty}, k)|^2 | \phi_1 = \psi_1, \ldots, \phi_{k-1} = \psi_{k-1}, \phi_{k+1} =  \psi_{k+1},\ldots, \phi_n = \psi_n  } \\
&\quad\quad \cdot 3 \E{ \left.  \max_{v}  \varphi^2( G,v ) \right| \phi_1 = \psi_1, \ldots, \phi_k = \psi'_k, \ldots,  \phi_n = \psi_n}
\EA
where $G_{k, \infty}$ is the random graph $G$ conditioned on $\phi_k = \PHImax$, and where we used H\"older's inequality and the fact that $(x+y)^2 \leq 3(x^2+y^2)$ for any $x,y \in \mathbb{R}$.  
Hence, using again H\"older's inequality, Efron-Stein's inequality becomes
\BA
\Var \lr{ \E{Y | \phi_1, \ldots, \phi_n} } &\leq \frac{1}{2} \s{k=1}{n} \E{ |h(\phi_1, \ldots, \phi_k,\ldots, \phi_n) - h(\phi_1, \ldots  \phi'_k,\ldots, \phi_n)|^2 } \\
&\leq 3 \s{k=1}{n} \sqrt{\E{ |V(G_{k, \infty}, k)|^4_{\ell}}    }  \sqrt{\E{  \max_{v}  \varphi^4( G,v ) }},
\EA
where $(\phi'_k)_k$ is an i.i.d. copy of $(\phi_k)_k$.
Now, due to \ref{lm::29},  $\E{ |V(G_{k, \infty}, k)|^4_{\ell}    } \leq \frac{c_1}{2} \rho^{4 \ell}$. Thus,
\[ \Var \lr{ \E{Y | \phi_1, \ldots, \phi_n} } \leq c_2 n \rho^{2 \ell}  \sqrt{\E{  \max_{v}  \varphi^4( G,v ) }}.\]

To bound $\Var\lr{Y|\phi_1=\psi_1, \ldots, \phi_n=\psi_n}$ we use again Efron-Stein's inequality. Define for $1 \leq k \leq n$,  $X_k = \{ 1 \leq v \leq k : \{v, k \} \in E \}$, where $E$ is the edge set of $G$. Then, conditioned on the weights ($\phi_u = \psi_u$), $\{X_k\}_{k}$ are independent. Let $\{X'_k\}_{k}$ be an independent copy of $\{X_k\}_{k}$ and define $G_k$ as the graph on vertex set $V$ with edge set $\cup_{v \neq k} X_v \cup X'_k$. Thus, conditional on the weights, $G_k$ equals $G$ except for the edges in $\{ 1 \leq v \leq k \}$ which are redrawn independently.

 Now, for some function $F_{\psi_1, \ldots, \psi_u}$,
\[ \sum_{v=1}^n  \tau ( G,v)  = F_{\psi_1, \ldots, \psi_n}(X_1, \ldots, X_k, \ldots, X_n), \]
and hence,
\[ \sum_{v=1}^n  \tau ( G_k,v)  = F_{\psi_1, \ldots, \psi_n}(X_1, \ldots, X'_k, \ldots, X_n). \]
Proceeding as above, we obtain
\BA
&\Var\lr{Y|\phi_1=\psi_1, \ldots, \phi_n=\psi_n} \\
&\leq \frac{1}{2} \s{k=1}{n} \E{ |F_{\psi_1, \ldots, \psi_n}(X_1, \ldots, X_k,\ldots, X_n) - F_{\psi_1, \ldots, \psi_n}(X_1, \ldots  X'_k,\ldots, X_n)|^2 }   \\
&\leq \frac{1}{2} \s{k=1}{n} \sqrt{\E{ |V(G, k)|^4_{\ell} \cap |V(G_{k}, k)|^4_{\ell} }    }  \sqrt{\E{  \lr{ \max_{v}  \varphi( G,v ) + \max_{v}  \varphi( G_k,v )}^4 }} \\
&\leq c_3 n \rho^{2 \ell}  \sqrt{\E{  \max_{v}  \varphi^4( G,v ) }}.
\EA
\end{proof}

\begin{proof}[Proof of Proposition \ref{prop::36}]
We recall that the coupling between neighbourhoods and branching processes is such that, in case of success, the weights are equal in both processes. Therefore, as in the proof of Proposition 36 in \cite{BoLeMa15}, we obtain
\[\E{\frac{1}{n} \sum_{v=1}^n  \tau ( G,v)} = \E{\tau(T,o)} + \epsilon(n),\]
where \[\epsilon(n) = \bigO(n^{-\gamma}) + c_1 n^{-\lr{ \frac{\gamma}{2} \wedge \frac{1}{40} }} \sqrt{ \E{ \max_{v \in [n]}  \varphi^2 (G,v) } \vee \E{ \varphi^2 (T,o) }}.\]
This error stems from the probability for the coupling to fail.

 Hence, 
\[ \E{ \left| \frac{1}{n} \sum_{v=1}^n  \tau ( G,v) - \E{\tau(T,o)} \right|}   \leq \sqrt{ \Var \lr{\frac{1}{n} \sum_{v=1}^n \tau(G,v)} } +\epsilon(n).\]
An appeal to Proposition \ref{prop::35} then finishes the proof.
\end{proof}

\begin{proof}[Proof of Proposition \ref{prop::37}]
We give the key steps used to prove Proposition 37 in \cite{BoLeMa15} together with the main differences in the current setting. For $(i)$, consider the branching process defined in Section \ref{sec::branching}, which we denote again by $Z_t(\pm)$. We denote the associated random rooted tree by $(T,o)$. 

Put $\tau(G,v) = \sum_{e \in \vec E, e_1=v}  \frac{\langle \EF_k , \Psi_{\ell} (e) \rangle^2}{ \mu_k ^{2 \ell}}.$ Then, $\frac{1}{n} \sum_v \tau(G,v) = \frac{1}{n}\sum_{e \in \vec E}  \frac{\langle \EF_k , \Psi_{\ell} (e) \rangle^2}{ \mu_k ^{2 \ell}}$ and $\tau(G,v) \leq \varphi(G,v) := \PHImax^2 \frac{S_{\ell}^2(v)}{\rho^{\ell}}$. It follows from Lemma \ref{lm::29} that $\E{ \max_{v \in [n]}  \varphi^4 (G,v) } = \bigO \lr{ (\log n)^{8} \rho^{4 \ell} }$. 

We have $\tau(T,o) = \sum_{v \in Z_1^o}  \frac{\langle \EF_k , \Psi_{\ell}^v \rangle^2}{ \mu_k ^{2 \ell}}$. Theorem \ref{th::21_other_martingale} says that $\lr{ \frac{\langle \EF_k , \Psi_{t} \rangle }{\mu_k^{t-1}} }_{t \geq 1}$ converges in $L^2$ and so does it conditional on $||Z_1^o|| = 1$. Hence, $\E{\tau(T,o)}$ converges.

An appeal to Proposition \ref{prop::36} in conjunction with the triangle inequality then establishes that $\frac{1}{n} \sum_v \tau(G,v)$ converges to a constant, say $c'_k$. 

Statement $(ii)$ follows similarly. 

The statements $(iii) - (v)$ follow after properly choosing local functionals. We further use that $\E{\phi_u \phi_v \EF_1(\sigma_u) \EF_2(\sigma_v) | \cT } = \E{ \phi_u \phi_v \frac{1}{2} \sigma_v | \cT } =0,$ for any two nodes $u,v$. Further, on the branching process,  $\E{\langle \EF_k , \Psi_{2\ell} \rangle    \langle \EF_j , \Psi_{\ell}  \rangle | \Psi_\ell} = \langle \EF_j , \Psi_{\ell} \rangle \langle \EF_k , M^\ell \Psi_{\ell} \rangle = \mu_k^\ell \langle \EF_k , \Psi_{\ell} \rangle    \langle \EF_j , \Psi_{\ell}  \rangle$.
\end{proof}

\begin{proof}[Proof of Proposition \ref{prop::38}]
Starting with $(i)$, we define the local function  $\tau$ as \\ $\tau(G,v) = \sum_{e \in \vec E, e_1 = v} P^2_{k,\ell} (e)  \mu^{-4 \ell}_k,$ for a rooted graph $(G,v)$. Let $$M (v) = \max_{0 \leq t \leq \ell} \max_{u \in (G,v)_t } \max_{s \leq 2 \ell - t}  (S_s (u) / \rho^s).$$
By monotonicity, the statement of Lemma \ref{lm::29} holds also for $\tilde S_{\ell - t -1}(h)$ and $\tilde S_t(g)$. We use this fact to bound powers of $M(v)$ in the following calculation:
\BA
\tau(G,v) & \leq  \rho^{-2 \ell}  \sum_{e \in \vec E, e_1 = v} \lr{ \sum_{t=0}^{\ell -1} \sum_{f \in \cY_t(e)} \| \EF_k \|_{\infty} \PHImax \tilde S_{t +1}(f) \tilde S_{\ell - t}(f)   }^2 \\
&  \leq c_1 \rho^{-2 \ell}  \sum_{e \in \vec E, e_1 = v} \lr{ \sum_{t=0}^{\ell -1} \sum_{f \in \cY_t(e)}   M^2 (v) \rho^{t+1} \rho^{\ell -t} }^2 \nonumber \\
& = c_1 \lr{   M^2 (v) \rho }^2    \sum_{e \in \vec E, e_1 = v} \lr{ \sum_{t=0}^{\ell -1} S_t(e) }^2 \\
& = c_2 \lr{   M^2 (v) \rho }^2    \sum_{e \in \vec E, e_1 = v} \lr{ M(v) \rho^{\ell} }^2 \label{eq:taugv2} \\
&\leq  c_2   M^7(v) \rho^{2 \ell}. 
\EA
We put $\varphi(G,v) = c_2   M^7(v) \rho^{2 \ell}$. Then, $\E{ \max_v \varphi( G,v)^4 } = O ( ( \log n)^{28}  \rho^{8\ell}),$ and the same bound holds for $\varphi(T,o)$. From Proposition \ref{prop::36}, we then know that
\be
\E{ \left|  \frac 1 { n} \sum_{e \in \vec E}  \frac{P^2_{k,\ell} (e) }{ \mu_k ^{4 \ell}}  - \E{\tau(T,o)}   \right|} \leq c_3 n^{-(\frac{\gamma}{2} \wedge \frac{1}{40})} (\log n)^7\rho^{2\ell},
\label{eq::conv_P_k_R}
\ee
where
\BA \tau(T,o)
&= \frac{1}{\mu_k ^{4 \ell}} \sum_{v \in Y_1^o} P^2_{k,\ell}(o \to v) \\
&= \frac{1}{\mu_k ^{4 \ell}} \sum_{v \in Y_1^o} \lr{Q_{k, \ell}^v}^2
, \label{eq::tau_T_o}\EA
where $Q_{k, \ell}^v$ is equal to $Q_{k, \ell}$ defined on the subtree of all vertices with common ancestor $v$.

We need to show that the expectation of $\tau(T,o)$ converges for $\ell \to \infty$. Conditional on $\sigma_o$, and $|Y_1^o|$, $\{Q_{k, \ell}^v \}_{v \in Y_1^o}$ are independent copies of $Q_{k,\ell}$ defined on the branching process in Section \ref{sec::branching} where the root has spin $\sigma_o$ with probability $\frac{a}{a+b}$ and random weight governed by the biased law $\NUstar$. The uniform $L^2$ convergence in Theorem \ref{th::25} establishes the claim.

We now prove $(ii)$. Put $\tau(G,v) =  \sum_{e \in \vec E,e_1 = v} ( P_{1,\ell} (e) + S_{1,\ell} (e) ) ( P_{2,\ell}(e) + S_{2,\ell} (e)).$ We claim that $\E{\tau(T,o)} = 0$. Consider $\tau(T,o) = \sum_{v \in Z_1^o} ( P_{1,\ell} (o \to v) + S_{1,\ell} (o \to v) ) ( P_{2,\ell}(o \to v) + S_{2,\ell} (o \to v))$. Firstly, for $k \in \{1,2\}$, $P_{k,\ell}(o \to v) = Q_{k, \ell}^v.$ Now, it follows from Theorem \ref{th::28},  that 
$\E{Q_{1,\ell}^v Q_{2,\ell}^v } = 0,$ since $\sigma_v$ is drawn uniformly from $\spm$.

Secondly, $S_{1,\ell} (o \to v) S_{2,\ell} (o \to v) = \frac{1}{2} \phi_o^2 \sigma_o S_{\ell}^2 (o \to v)$ has also zero expectation.

Thirdly, 
\be Q_{1, \ell}^v S_{2,\ell} (o \to v) = \frac{1}{2} \sum_{(u_0,\ldots,u_{2\ell+1})\in\cP_{2 \ell +1}^v}  \phi_{u_{2 \ell +1}} \phi_o \sigma_o S_{\ell} (o \to v), \label{eq::QS}\ee where $\cP_{2 \ell +1}^v$ is $\cP_{2 \ell +1}$ (from \eqref{eq::Q}) defined on the subtree of all vertices with common ancestor $v$. The expectation of $Q_{1, \ell}^v S_{2,\ell} (o \to v)$ is thus zero since $\sigma_o$ is independent of all other terms in \eqref{eq::QS}.

Lastly, $Q_{2, \ell}^v = \sum_{ (u_1, \ldots, u_{2\ell+1}) \in \cP_{2 \ell +1}^v} \sigma_{u_{2 \ell +1}},$ is seen to have zero expectation. 

Those four statements combined establish $\E{\tau(T,o)} = 0$. As above, we calculate $\E{ \max_v \varphi( G,v)^4 } = O ( ( \log n)^{28}  \rho^{16 \ell}).$
\end{proof} 

\begin{proof}[Proof of Proposition \ref{prop::37_special}]
Put $\tau$ as in Proposition \ref{prop::37} $(i)$, then
$\tau(T,o) = \sum_{v \in Z_1^o}  \frac{\langle \EF_2 , \Psi_{\ell}^v \rangle^2}{ \mu_2 ^{2 \ell}}$. Now,
\[ \E{\langle \EF_2 , \Psi_{\ell}^v \rangle^2} = \E{\langle \EF_2 , \frac{\PHItwo}{\PHI} Z_{\ell}^v \rangle^2} + \E{\langle \EF_2 , \Psi_{\ell}^v - \frac{\PHItwo}{\PHI} Z_{\ell}^v \rangle^2} \geq \lr{\frac{\PHItwo}{\PHI}}^2 \E{\langle \EF_2 , Z_{\ell}^v \rangle^2}. \]
Now, Theorem $2.4$ in \cite{KeSt66} says that for some random variable $X$ with strictly positive variance, weakly, $\frac{\langle \EF_2 , Z_{\ell}^v \rangle}{\rho^{\ell/2}} \to X,$ as $\ell \to \infty$. Because of the weak convergence, we have for any $\theta > 0$, $\E{\lr{\frac{\langle \EF_2 , Z_{\ell}^v \rangle}{\rho^{\ell/2}}}^2 \wedge \theta}\to \E{X^2 \wedge \theta}$, as $\ell \to \infty$. Now by Lebesque's dominated convergence theorem, $\E{X^2 \wedge \theta} \to \E{X^2} > 0,$ as $\theta \to \infty$. 
\end{proof}
\begin{proof}[Proof of Proposition \ref{prop::38_special}]
Use $\tau$ from Proposition \ref{prop::38} $(i)$, together with the bound $\E{Q_{2,\ell}^2} \leq C \rho^{2\ell} \ell^5$ from Theorem \ref{th::25_special}. 
\end{proof}

\section{Proof of Proposition \ref{prop::14}}
\label{App::norms}
\subsection{Bound on $\| \Delta^{(k)}\|$}
We set 
\[
m = \left\lfloor  \frac{ \log n }{ 13 \log (\log n)} \right\rfloor.
\]
We bound the norm of $\| \Delta^{(k)}\|$ by using the trace method. Following $(36)$ in \cite{BoLeMa15} (which remains true for the DC-SBM), we obtain
\be
\| \Delta^{(k-1)}  \| ^{2 m} \leq \sum_{\gamma \in W_{k,m} }   \prod_{i=1}^{2m}  \prod_{s=1}^{k} \underline  A_{\gamma_{i,s-1}  \gamma_{i,s}},
\label{eq::Delta_trace}
 \ee 
 where $W_{k,m}$ is the collection containing all sequences of paths $\gamma = ( \gamma_1, \ldots, \gamma_{2m})$ such that for all $i$:
\begin{itemize}
\item $\gamma_i = (\gamma_{i,0}, \cdots, \gamma_{i,k})\in V^{k+1}$ is a non-backtracking tangle-free path of length $k$, and,
\item $(\gamma_{i, k-1},\gamma_{i, k}) = (\gamma_{i+1,1}, \gamma_{i+1,0} )$,
\end{itemize} 
where we put $\gamma_0 = \gamma_{2m}$.

 Recall the notation $G(\gamma) = ( V(\gamma), E(\gamma) ) $. Further introduce the notation $\mathbb{E}_{\phi}\lr{\cdot} = \E{\cdot|\phi_1, \ldots, \phi_n}$. We bound, for a given $\gamma \in W_{k,m}$,
\be 
\ba
\mathbb{E}_{\phi} \lr{ \prod_{i=1}^{2m}  \prod_{s=1}^{k} \underline  A_{\gamma_{i,s-1}  \gamma_{i,s}} } 
&= \prod_{e \in E(\gamma)}  \mathbb{E}_{\phi} \lr{ \underline  A_{e_1 e_2}^{p_{e_1 e_2}^{(\gamma)}}},
\ea
\label{eq::Eprod_paths}
\ee
where for $e \in E(\gamma)$, $p_{e_1 e_2}^{(\gamma)}$ denotes the number of times the edge $e$ is traversed on the walk $\gamma$. In \eqref{eq::Eprod_paths} we used that $\underline{A}$ is symmetric and that, conditional on the weights, edges are independently present. Note that for any edge $uw$, and integer $p$, 
\[ \mathbb{E}_{\phi} \underline  A_{uw}^p \leq \phi_u \phi_w \frac{W_{\sigma(u) \sigma(w)}}{n}. \]

 Below in Lemma \ref{lm::bound_overshoot}, we construct a spanning tree $T(\gamma) = (V(\gamma),E_T(\gamma))$ of $\gamma$.
In particular, for the $e - (v-1)$ edges not present in $T$, we have 
$\phi_u \phi_w \frac{W_{\sigma(u) \sigma(w)}}{n} \leq \frac{c_1}{n}$, with $c_1 = \PHImax^2 (a \vee b). $ Putting this into \eqref{eq::Eprod_paths}, we get
\be
\ba 
\prod_{e \in E(\gamma)}  \mathbb{E}_{\phi} \lr{ \underline  A_{e_1 e_2}^{p_{e_1 e_2}^{(\gamma)}}} 
&\leq (c_1/n)^{e-v+1} \prod_{e \in E_T(\gamma)}  \phi_{e_1} \phi_{e_2} \frac{W_{\sigma(e_1) \sigma(e_2)}}{n} \\
&= (c_1/n)^{e-v+1} \prod_{u \in V(\gamma)}  \phi_{u}^{d_u} \prod_{e \in E_T(\gamma)}   \frac{W_{\sigma(e_1) \sigma(e_2)}}{n},
\ea
\label{eq::Eprod_paths_2}
\ee
where $d_u$ is the degree of $u$ in the \emph{spanning} tree. Consequently, 
\be
\mathbb{E} \lr{ \prod_{e \in E(\gamma)}    \underline  A_{e_1 e_2}^{p_{e_1 e_2}^{(\gamma)}} } \leq (c/n)^{e-v+1} \prod_{u \in V(\gamma)}  \Phi^{(d_u)} \prod_{e \in E_T(\gamma)}   \frac{W_{\sigma(e_1) \sigma(e_2)}}{n}.
\label{eq::Egamma_bound}
\ee
Let $\tau: [v(\gamma)] \mapsto V(\gamma)$ be the bijection describing the order the vertices are visited for the first time. I.e., for $1 \leq u \leq v(\gamma)-1$, $\tau(u)$ is seen for the first time, before $\tau(u+1)$. 

We shall say that a path $\gamma_c$ is canonical if $V(\gamma_c) = [v(\gamma_c)]$ and the vertices are first visited in the order $1, \ldots, v(\gamma_c)$. With every path $\gamma$
there corresponds (through the bijection $\tau$) a canonical path $\gamma_c$. Consequently, if $\c W_{k,m}(v,e)$ denotes the set of canonical paths in $W_{k,m}$ with $v$ vertices and $e$ edges, and $I_{\gamma_c}$ the set of all injections from $[v(\gamma_c)]$ to $[n]$,
\be
\E{ \sum_{\gamma \in W_{k,m} }   \prod_{i=1}^{2m}  \prod_{s=1}^{k} \underline  A_{\gamma_{i,s-1}  \gamma_{i,s}} }
  \leq \sum_{v=3}^{k m +1} \sum_{e = v - 1} ^{ km } \sum_{\gamma_c \in \cW_{k,m}(v,e)} \sum_{\tau \in I_{\gamma_c}} \E{ \prod_{e \in E(\gamma_c)}    \underline  A_{\tau(e_1) \tau(e_2)}^{p_{e_1 e_2}^{(\gamma_c)}} },
\label{eq::Eprod_paths_injections}
 \ee
because any non-backtracking path has at least $3$ vertices, and $v-1 \leq e \leq km$, since \eqref{eq::Eprod_paths} is non-zero only if each edge is traversed at least twice. 

We now bound the term $\sum_{\tau \in I_{\gamma_c}} \E{ \prod_{e \in E(\gamma_c)}    \underline  A_{\tau(e_1) \tau(e_2)}^{p_{e_1 e_2}^{(\gamma_c)}} }$ in \eqref{eq::Eprod_paths_injections}. Using \eqref{eq::Egamma_bound}, we have, 
\be
\ba 
\sum_{\tau \in I_{\gamma_c}} \E{ \prod_{e \in E(\gamma_c)}    \underline  A_{\tau(e_1) \tau(e_2)}^{p_{e_1 e_2}^{(\gamma_c)}} } 
\leq (c_1/n)^{e-v+1} \prod_{u=1}^{v(\gamma_c)}  \Phi^{(d_u)}  \sum_{\tau \in I_{\gamma_c}} \prod_{e \in E_T(\gamma_c)}   \frac{W_{\sigma(\tau(e_1)) \sigma(\tau(e_2))}}{n}.
\ea
\ee
Our objective is to compare $\prod_{u=1}^{v(\gamma_c)}  \Phi^{(d_u)}  \sum_{\tau \in I_{\gamma_c}} \prod_{e \in E_T(\gamma_c)}   \frac{W_{\sigma(\tau(e_1)) \sigma(\tau(e_2))}}{n}$ with $n  \rho^{(v-1)}$. We start by analysing the term containing the spins:
\begin{lemma}
For any canonical path $\gamma_c \in \cW_{k,m}$, 
\be 
\sum_{\tau \in I_{\gamma_c}} \prod_{e \in E_T(\gamma_c)}   \frac{W_{\sigma(\tau(e_1)) \sigma(\tau(e_2))}}{n} \leq (1+o(1)) n \lr{\frac{a+b}{2}}^{v-1}.
\label{eq::Bound_injections_spins}
\ee
\label{lm::Bound_injections_spins}
\end{lemma}
\begin{proof}
Let $l$ be any leaf on the tree with unique neighbour $g$. Then, writing $\tau_u = \tau(u)$ for $u \in \{1,\ldots,v\}$,
\[ \sum_{\tau \in I_{\gamma_c}} \prod_{e \in E_T(\gamma_c)}   \frac{W_{\sigma(\tau(e_1)) \sigma(\tau(e_2))}}{n} 
\leq \sum_{\tau_1 = 1}^n \cdots \sum_{\tau_v = 1}^n \prod_{e \in E_T(\gamma_c)}   \frac{W_{\sigma(\tau_{e_1}) \sigma(\tau_{e_2})}}{n}.  \]
Keeping $\tau_g$ fixed,
\[ \ba 
 \sum_{\tau_l = 1}^n \prod_{e \in E_T(\gamma_c)}   \frac{W_{\sigma(\tau_{e_1}) \sigma(\tau_{e_2})}}{n} 
 &=  \prod_{e \in E_T(\gamma_c) \setminus \{ g,l \} } \frac{W_{\sigma(\tau_{e_1}) \sigma(\tau_{e_2})}}{n} \sum_{\tau_l = 1}^n    \frac{W_{\sigma(\tau_{g}) \sigma(\tau_{l})}}{n} \\
 &= \prod_{e \in E_T(\gamma_c) \setminus \{ g,l \} } \frac{W_{\sigma(\tau_{e_1}) \sigma(\tau_{e_2})}}{n} \lr{\frac{a+b}{2} + \bigO(n^{-\gamma})},
\ea \]
due to assumption \eqref{eq::gamma}.  
\\ Repeating inductively this procedure (by removing leaves from the tree) proves the assertion.
\end{proof}
It remains to bound $\prod_{u=1}^{v(\gamma_c)}  \Phi^{(d_u)}$. To do so, we note that, since the weights are assumed to be bounded, 
\[ \Phi^{(d_u)} \leq C_2^{d_u-2} \Phi^{(2)} \lr{ \Phi^{(1)} }^{d_u - 2} \]
if $d_u \geq 2,$
with $C_2 = \frac{\bdphi}{\Phi^{(1)}} > 1$. Consequently,
\be 
\ba 
\prod_{u=1}^{v(\gamma_c)}  \Phi^{(d_u)} 
&\leq C_2^{\sum_{u: d_u > 2}(d_u-2)}  \prod_{u: d_u > 2} \Phi^{(2)} \lr{ \Phi^{(1)} }^{d_u - 2}   \prod_{u: d_u \leq 2} \Phi^{(d_u)} \\
&\leq C_2^{\sum_{u: d_u > 2}(d_u-2)} \lr{\Phi^{(2)}}^{\frac{1}{2}\sum_{u=1}^v d_u} \\
& = C_2^{\sum_{u: d_u > 2}(d_u-2)} \lr{\Phi^{(2)}}^{v-1},  
\ea
\label{eq::prod_phi}
\ee
where we used that by Jensen's inequality $\lr{ \Phi^{(1)} }^2 \leq  \Phi^{(2)}.$ 

Now, the sum $\sum_{u: d_u > 2}(d_u -2)$ is small for a tree spanning a path in $W_{k,m}$:

\begin{lemma}
For any $\gamma \in W_{k,m}$, with $v$ vertices and $e$ edges, there exists a tree spanning $\gamma$ with degrees $(d_u)_{u=1}^v$ such that:
\be 
\sum_{u: d_u > 2}(d_u-2) \leq e - (v-1) + 2m. \label{eq::bound_overshoot}
\ee
\label{lm::bound_overshoot}
\end{lemma}
\begin{proof}
We construct a spanning tree, while traversing $\gamma$.  We denote by $p(t)$ the graph constructed at step $t \geq 0$. Put $p(0) = \{\gamma_{1,0} ,\emptyset \}$ and $r=s=0$ (the meaning of these two \emph{counters} becomes clear in the algorithm below). Consider edge $f$ traversed in step $t+1$ of the walk: If $f$ or $\check  f$ has already been traversed, then continue with step $t + 2$. Otherwise, if both $f$ and $\check  f$ have not yet been traversed, distinguish between the following cases:
\begin{enumerate}[1.]
\item $f_1$ is a leaf of $p(t)$ and 
 \begin{enumerate}[a)]
 \item $p(t)$ contains a cycle, then if $f_2 \notin  p(t)$, put $p(t+1) = p(t) \cup f$, otherwise, if $f_2 \in  p(t)$, put $p(t+1) = p(t)$;
 \item $p(t)$ does not contain a cycle, then
 put $p(t+1) = p(t) \cup f$. If $f_2 \in  p(t)$, then put $e_c = f$; 
 \end{enumerate}
 \item $f_1$ is \emph{not} a leaf of $p(t)$ and
  \begin{enumerate}[a)]
  \item $p(t)$ contains a cycle, then put $p(t+1) = (p(t) \setminus e_c) \cup f$. If $f_2 \in p(t)$, put $e_c = f$. Increase the value of $r$ with one.
  \item $p(t)$ does not contain a cycle, then put $p(t+1) = p(t) \cup f$. If $f_2 \in p(t)$, put $e_c = f$. Otherwise, if $f_2 \notin  p(t)$, increase the value of $s$ with one.
 \end{enumerate}
\end{enumerate}
Once the path is completely traversed, remove $e_c$ to obtain a spanning tree. 

 Note that at each stage of the construction, the graph contains at most one cycle and in this case, removing $e_c$ will make the graph into a tree. 

 Further, cases $1.a$ and $1.b$ do not contribute to $\sum_{u: d_u > 2}(d_u-2)$, since the leave in $p(t)$ becomes a vertex with degree at most $2$ in $p(t+1)$. A cycle formed in step $t+1$ will \emph{temporarily} increase the degree of the vertex that is merged by the leaf, however this edge $e_c$ will later be removed. 

In case $2.a$, the degree of vertex $f_1$ increases with one, however, at the same time an edge is removed. The number of times $2.a$ happens, $r$,  is thus bounded by the number of times an edge is removed: $r \leq e - (v-1)$. 

In case $2.b$, we need only to consider the case where no cycle is formed. But, before arriving at such a vertex considered in $2.b$, the path must have made a backtrack. Hence $s \leq 2m$.

(In fact, between two subsequent occurrences of event $2$, the walk should at least either make a backtrack or 'get back to the tree' by forming a cycle: giving the same bound for $s +r$). 

 All together, 
\[ \sum_{u: d_u > 2}(d_u-2) \leq r + s \leq e - (v-1) + 2m. \] 
\end{proof}

Finally, we recall the bound on the cardinality of $\cW_{k,m}$ from \cite{BoLeMa15}:
\begin{lemma}[Lemma $17$ in \cite{BoLeMa15}]\label{lm::17}
Let $\cW_{k,m} (v,e) $ be the set of canonical paths with $v(\gamma) = v$ and $e(\gamma) = e$. We have 
\be
| \cW _{k,m} (v,e) | \leq  k^{2m} (2km)^{ 6 m ( e -v +1)}.
\label{eq::bound_cW}
\ee
\end{lemma}
Hence, combining \eqref{eq::Delta_trace}, \eqref{eq::Eprod_paths_injections} -  \eqref{eq::bound_cW},
\BA
\E{\| \Delta^{(k-1)} \| ^{2 m}} 
&\leq  \sum_{v=3}^{k m +1} \sum_{e = v - 1} ^{ km } | \cW _{k,m} (v,e) | \lr{ \frac{ c}{n} }^{e - (v-1)} n C^{e - (v-1) + 2m} \rho^{v-1} \\
&\leq n c_5^m \rho^{km} \sum_{v=3}^{k m +1} \sum_{e = v - 1} ^{ km } \ell^{2m} \lr{ \frac{c_7  ( 2 \ell m )^{ 6 m  } }{n} }^{ e - (v - 1)}\\
&\leq n c_5^m \rho^{km} \ell^{2m} \ell m \sum_{s = 0} ^{ \infty }  \lr{ \frac{c_7  ( 2 \ell m )^{ 6 m  } }{n} }^{ s} \\
&\leq n (c_8 \log n)^m \log^2 n \rho^{km} \\
&\leq (c_9 \log n)^{16m} \rho^{km},
\label{eq::Delta_k_m}
\EA
where we used the bound on $m$, in particular to derive convergence of the series, and the fact that $n^{1/m} = o(\log n)^{14}$.

We finish by using Markov's inequality.

\subsection{Bound on $\| \Delta^{(k)} \chi_i \|$}
We point out the differences with bound $(31)$ in \cite{BoLeMa15}: 
Here, we have
\BA
\E{  \|  \Delta^{(k-1)} \chi_i  \| ^2 } 
& = \E{ \sum_{e,f,g}  \Delta^{(k-1)}_{ef} \Delta^{(k-1)}_{eg} \xi_i(f) \xi_i(g) } \\
& \leq \bdphi^2  \E{ \sum_{e,f,g}  \Delta^{(k-1)}_{ef} \Delta^{(k-1)}_{eg} } \\
& \leq \bdphi^2 \sum_{\gamma \in W''_{k,1} } \E{ \prod_{i=1}^2 \prod_{s=1}^k \underline A_{\gamma_{i,s-1}, \gamma_{i,s}} }, 
\EA
where $ W''_{k,1} $ is defined in \cite{BoLeMa15}. In the latter paper it is also shown that the same bound, Lemma \ref{lm::17} holds for the cardinality of $ W''_{k,1} $). Hence, using the penultimate line of \eqref{eq::Delta_k_m} with $m=1$, gives
\[ \E{  \|  \Delta^{(k-1)} \chi_i  \| ^2 } \leq c_1 n \log^3 (n) \rho^k.  \]

\subsection{Bound on $\|R^{(\ell)}_k\|$}
Put
\[
m = \left\lfloor  \frac{ \log n }{ 25 \log (\log n)} \right\rfloor.
\]

We apply the same strategy as above: for $0 \leq k \leq \ell-1$, we have the bound
\be
\ba 
\| R^{(\ell-1)}_k  \| ^{2 m} &\leq  \text{tr} \left\{ \lr{  R^{(\ell-1)}_k { R^{(\ell-1)}_k }^*}^{m}  \right\} \\
&= \sum_{\gamma \in T'_{\ell,m,k} }   \prod_{i=1}^{2m}  \prod_{s=1}^{k} \underline  A_{\gamma_{i,s-1}  \gamma_{i,s}} \phi_{\gamma_{i,k}} \phi_{\gamma_{i,k+1}} W_{\sigma(\gamma_{i,k})\sigma(\gamma_{i,k+1})} \prod_{s=k+2}^{\ell} A_{\gamma_{i,s-1}  \gamma_{i,s}} \\
&\leq c_1^m \sum_{\gamma \in T'_{\ell,m,k} }   \prod_{i=1}^{2m}  \prod_{s=1}^{k} \underline  A_{\gamma_{i,s-1}  \gamma_{i,s}}  \prod_{s=k+2}^{\ell} A_{\gamma_{i,s-1}  \gamma_{i,s}},
\ea
\label{eq::R_k_m}
\ee
where $c_1 = \bdphi^4 (a \vee b)^2,$ and where 
$T'_{\ell,m,k}$ is the collection containing all sequences of paths 
$\gamma = ( \gamma_1, \ldots, \gamma_{2m})$ such that
\begin{itemize}
\item for all $i$: $\gamma_i = (\gamma^1 _i, \gamma^2_i)$, where $\gamma^1_i = (\gamma_{i,0}, \cdots, \gamma_{i,k})$ and $\gamma^2_i = (\gamma_{i,k+1}, \cdots, \gamma_{i,\ell})$ are non-backtracking tangle-free;
\item for all odd $i$: $(\gamma_{i,0}, \gamma_{i,1} ) = (\gamma_{i-1, 0},\gamma_{i-1,1}) $ and $ (\gamma_{i,\ell-1}, \gamma_{i,\ell} ) = (\gamma_{i+1, \ell-1},\gamma_{i+1, \ell})$,
\end{itemize}
with the convention that $\gamma_0 = \gamma_{2m}$.

 To calculate the expectation of $\| R^{(\ell-1)}_k  \| ^{2 m}$, we note that
\[ \E{\prod_{i=1}^{2m}  \prod_{s=1}^{k} \underline  A_{\gamma_{i,s-1}  \gamma_{i,s}}  \prod_{s=k+2}^{\ell} A_{\gamma_{i,s-1}  \gamma_{i,s}} } \]
is non-zero only if (for $i$ fixed) each edge $ \{\gamma_{i,s-1} , \gamma_{i,s}\} $ for $1 \leq s \leq k$ appears more than once in the $2(\ell - 1)m$ pairs $\{ \{\gamma_{j,s-1} , \gamma_{j,s}\} \}_{j=1, s \neq k+1}^{j = 2m}$. 
Hence,
\be \E{\| R^{(\ell-1)}_k  \| ^{2 m}} \leq c_1^m \sum_{\gamma \in T_{\ell,m,k} }   \E{ \prod_{i=1}^{2m}  \prod_{s=1}^{k} \underline  A_{\gamma_{i,s-1}  \gamma_{i,s}}  \prod_{s=k+2}^{\ell} A_{\gamma_{i,s-1}  \gamma_{i,s}} },  \ee
where
\be 
T_{\ell, m, k} = \{ \gamma \in T'_{\ell,m,k} \ | \ v(\gamma) \leq e(\gamma) \leq km + 2m(\ell - 1 - k)  \}.
\label{eq::T_l_m_k}
 \ee

Similarly as in establishing the bound on $\| \Delta^{(k)} \|,$
we say that a path $\gamma_c$ is \emph{canonical} if $V(\gamma_c) = [v(\gamma_c)]$ and the vertices are first visited in order. We denote by $\cT_{\ell,m,k} (v,e) $ the set of canonical paths in $T_{\ell, m, k}$ with $v$ vertices and $e$ edges. Then:
\be \ba &\E{\| R^{(\ell-1)}_k  \| ^{2 m}} \\
&\leq c_1^m \sum_{v=1}^{m(2 \ell - 2 - k)} \sum_{e = v } ^{ m(2 \ell - 2 - k) } \sum_{\gamma_c \in \cT_{\ell, m, k}(v,e)} \sum_{\tau \in I_{\gamma_c}} \E{ \prod_{e \in E(\gamma_c)}    \underline  A_{\tau(e_1) \tau(e_2)}^{\underline{p}_{e_1 e_2}^{(\gamma_c)}}  A_{\tau(e_1) \tau(e_2)}^{p_{e_1 e_2}^{(\gamma_c)}}}, 
\label{eq::E_R_ell_m} \ea \ee
where $I_{\gamma_c}$ is defined as above, $\underline{p}_{e_1 e_2}^{(\gamma_c)}$ is the number of times edge $\{ e_1, e_2 \}$ occurs in 
\\$\{ \{\gamma_{j,s-1} , \gamma_{j,s}\} \}_{s=1, j=1}^{s = k,j = 2m}$ and  ${p}_{e_1 e_2}^{(\gamma_c)}$ denotes the number of times edge $\{ e_1, e_2 \}$ occurs in the remainder of the collection of edges, $\{ \{\gamma_{j,s-1} , \gamma_{j,s}\} \}_{s=k+2, j=1}^{s = \ell,j = 2m}$.

 Now, again,
\[ \E{ \underline A_{\tau(e_1) \tau(e_2)}^{\underline{p}_{e_1 e_2}^{(\gamma_c)}}  A_{\tau(e_1) \tau(e_2)}^{p_{e_1 e_2}^{(\gamma_c)}} } \leq \phi_{\tau(e_1)} \phi_{\tau(e_2)} \frac{W_{\sigma(\tau(e_1)) \sigma(\tau(e_2))}}{n}. \]

Below we  construct a spanning forest $F = (V(\gamma),E_F(\gamma))$ of $\gamma$ (i.e., $F$ is the disjoint union of trees, each spanning another component of $G(\gamma)$). 

Let $n_{\cC} \leq m$ denote the number of components of $G(\gamma)$. Then,
\be 
\E{ \prod_{e \in E(\gamma)}  \underline  A_{e_1 e_2}^{\underline{p}_{e_1 e_2}^{(\gamma)}}  A_{e_1 e_2}^{p_{e_1 e_2}^{(\gamma)}} } 
\leq (c/n)^{e-(v -n_{\cC})} \prod_{u \in V(\gamma)}  \Phi^{(d_u)} \prod_{e \in E_F(\gamma)}   \frac{W_{\sigma(e_1) \sigma(e_2)}}{n}
\label{eq::Egamma_bound_R},
\ee
with $d_u$ the degree of vertex $u$ in the \emph{forest} $F$, compare to \eqref{eq::Egamma_bound}.

 Now, this time,
\begin{lemma}
For any canonical path $\gamma_c \in \cT_{\ell,m,k} (v,e)$, 
\be 
\sum_{\tau \in I_{\gamma_c}} \prod_{e \in E_F(\gamma_c)}   \frac{W_{\sigma(\tau(e_1)) \sigma(\tau(e_2))}}{n} \leq (1+o(1)) n^{n_{\cC}} \lr{\frac{a+b}{2}}^{v-n_{\cC}}.
\label{eq::Bound_injections_spins_R}
\ee
\begin{proof}
Apply Lemma \ref{lm::Bound_injections_spins} subsequently to the different components of $F$.
\end{proof}
\end{lemma}
Further, applying \eqref{eq::prod_phi} to different components in $F$ gives
\be 
\ba 
\prod_{u=1}^{v(\gamma)}  \Phi^{(d_u)} 
&\leq C_2^{\sum_{u: d_u > 2}(d_u-2)} \lr{\Phi^{(2)}}^{v-n_{\cC}}.  
\ea
\label{eq::prod_phi_R}
\ee
Together, 
\be 
\sum_{\tau \in I_{\gamma_c}} \E{ \prod_{e \in E(\gamma_c)}    \underline  A_{\tau(e_1) \tau(e_2)}^{\underline{p}_{e_1 e_2}^{(\gamma_c)}}  A_{\tau(e_1) \tau(e_2)}^{p_{e_1 e_2}^{(\gamma_c)}}} \leq (c / n)^{e-v} C_2^{\sum_{u: d_u > 2}(d_u-2)}  \rho^{ v-n_{\cC} }.
\label{eq::sum_tau}
\ee
Again, we bound $\sum_{u: d_u > 2}(d_u-2)$:
\begin{lemma}
For any $\gamma \in \cT_{\ell,m,k}$, with $v$ vertices and $e$ edges, there exists a forest spanning $\gamma$ with degrees $(d_u)_{u=1}^v$ such that:
\be 
\sum_{u: d_u > 2}(d_u-2) \leq 18m + e - (v-n_{\cC}).
\label{eq::bound_overshoot_degrees}
\ee
\label{lm::bound_overshoot_forest}
\end{lemma}
\begin{proof}
As in Lemma \ref{lm::bound_overshoot}, we construct the spanning forest, while traversing $\gamma$.  Again $p(t)$ denotes the graph constructed at step $t \geq 0$, with $p(0) = \{\gamma_{1,0} ,\emptyset \}$. Further, we introduce three counters: $r=s=q=0$, together with $e_c = \emptyset$ (below, $e_c$ is either equal to $\emptyset$ or it is an edge such that $p(t)$ contains one cycle, but $p(t) \setminus e_C$ is a forest). At any step $t$, we let $C_1, \ldots, C_{\#\text{components}}$ be the components of $p(t)$. 

Consider step $t+1$ of the walk: if the step consists in jumping to a vertex $w$, then put $p(t+1) = (p(t) \setminus e_C) \cup \{w\}$.

Else, if the step consists in traversing an edge $f = f_1 f_2$, then: 
If $f$ or $\check  f$ has already been traversed,  continue with step $t + 2$. Otherwise, if both $f$ and $\check  f$ have not yet been traversed, distinguish between the following cases:

\begin{enumerate}[1.]
\item $f_1$ is a leave or an isolated vertex of component $C_i$ of $p(t)$ and 
 \begin{enumerate}[a.]
 \item $C_i$ does not contain a cycle, then
 put $p(t+1) = p(t) \cup f$. Further, distinguish between the following cases:
 \begin{enumerate}[i)]
 	\item $f_2 \notin  p(t)$;
 	\item $f_2 \in  C_i$, then put $e_c = f$; 
 	\item $f_2 \in  C_{j \neq i}$, then increase the value of $s$ with one. 
 \end{enumerate}
\item $C_i$ contains a cycle, then distinguish between the following
 cases:
\begin{enumerate}[i)]
 	\item $f_2 \notin  p(t)$, then put $p(t+1) = p(t) \cup f$;
 	\item $f_2 \in  C_i$, then put  $p(t+1) = p(t)$; 
 	\item $f_2 \in  C_{j \neq i}$, then put $p(t+1) = p(t) \cup f$ and increase the value of $s$ with one. 
 \end{enumerate}
  \end{enumerate}
 \item $f_1$ in component $C_i$ has degree at least $2$ in $p(t)$, then distinguish between the following
 cases:
  \begin{enumerate}[a.]
  \item $C_i$ does not contain a cycle, then put $p(t+1) = p(t) \cup f$. Further, distinguish between the following cases:
  \begin{enumerate}[i)]
 	\item $f_2 \notin  p(t)$, then increase the value of $q$ with one;
 	\item $f_2 \in  C_i$, then put $e_c = f$; 
 	\item $f_2 \in  C_{j \neq i}$, then increase the value of $s$ with two. 
 \end{enumerate}
 
  \item $C_i$ contains a cycle, then put $p(t+1) = (p(t) \setminus e_c) \cup f$. Further, distinguish between the following cases:
  \begin{enumerate}[i)]
 	\item $f_2 \notin  p(t)$, then increase the value of $r$ with one;
 	\item $f_2 \in  C_i$, then put  $e_c = f$; 
 	\item $f_2 \in  C_{j \neq i}$, then increase the value of $s$ with two. 
 \end{enumerate}
   
 \end{enumerate}
\end{enumerate}
Once the path is completely traversed, remove $e_c$ to obtain a spanning tree. 

The only cases that contribute to $\sum_{u: d_u > 2}(d_u-2)$ are $1.a.iii, 1.b.iii, 2.a.i, 2.a.iii, 2.b.i$ and $2.b.iii$. 

 Now, $s$ counts the contribution of $1.a.iii, 1.b.iii,  2.a.iii$ and $2.b.iii$. But, in all those $4$ cases, two components are merged, hence $s \leq 6\  \#$merges $\leq 12 m$.

By definition of the event $2.b.i$, $r$ is an upper bound for the number of edges that are removed: $r \leq e - (v-n_{\cC})$.

To bound $q$ (which counts the occurrence of $2.a.i$), note that between two subsequent occurrences of the event $2.a.i$, the walk makes at least one of the following: a backtrack, a jump or a merge. Hence $q \leq 2m + 2m + 2m = 6m$.

Adding the bounds for $r,q$ and $s$ establishes \eqref{eq::bound_overshoot_degrees}. 
\end{proof}
\noindent Returning to \eqref{eq::sum_tau}, we get, since $n_{\cC} \leq 2m$:
\be 
\ba
\sum_{\tau \in I_{\gamma_c}} \E{ \prod_{e \in E(\gamma_c)}    \underline  A_{\tau(e_1) \tau(e_2)}^{\underline{p}_{e_1 e_2}^{(\gamma_c)}}  A_{\tau(e_1) \tau(e_2)}^{p_{e_1 e_2}^{(\gamma_c)}}} &\leq (c_1 / n)^{e-v} C_2^{18m + e - v + n_{\cC}}  \rho^{v-n_{\cC}} \\
&\leq \lr{\frac{c_3}{n}}^{e-v} c_4^{m}  \rho^{v-n_{\cC}}.
\ea
\ee
Putting this into \eqref{eq::E_R_ell_m}, we obtain
\be \E{\| R^{(\ell-1)}_k  \| ^{2 m}} \leq c_5^m \sum_{v=1}^{m(2 \ell - 2 - k)} \sum_{e = v } ^{ m(2 \ell - 2 - k) } \sum_{\gamma_c \in \cT_{\ell, m, k}(v,e)}  \lr{\frac{c_3}{n}}^{e-v} \rho^{m(2 \ell  - k)}. 
 \ee
 Now the cardinality of $\cT_{\ell, m, k}(v,e)$ is bounded in the following lemma:
 \begin{lemma}[Lemma $18$ in \cite{BoLeMa15}]\label{lm::18}
Let $\cT_{\ell,m,k} (v,e) $ be the set of canonical paths in $T_{\ell,m,k}$ with $v(\gamma) = v$ and $e(\gamma) = e$. We have 
$$
| \cT _{\ell,m,k} (v,e) | \leq   (4 \ell m )^{ 12 m ( e - v +1) + 8 m }.
$$
\end{lemma}

\noindent Hence,
\be
\ba
 \E{\| R^{(\ell-1)}_k  \| ^{2 m}}
 &\leq c_5^m \rho^{m(2 \ell  - k)} 
 \sum_{v=1}^{m(2 \ell - 2 - k)} \sum_{e = v } ^{ m(2 \ell - 2 - k) } (4 \ell m )^{ 12 m ( e - v +1) + 8 m }  \lr{\frac{c_3}{n}}^{e-v}   \\
 &\leq \rho^{m(2 \ell  - k)} c_5^m (4 \ell m )^{ 20 m  }   
 \sum_{v=1}^{m(2 \ell - 2 - k)} \sum_{s = 0} ^{ \infty }  \lr{\frac{c_3 (4 \ell m)^{ 12 m } }{n}}^s \\
  &\leq \rho^{m(2 \ell  - k)} c_5^m (4 \ell m )^{ 20 m  }   
 2 \ell m \cdot \bigO(1) \\
 &\leq \rho^{m(2 \ell  - k)} (c_5 \text{log} (n))^{42m}.
 \ea
 \label{eq::E_R_ell_m_final}
 \ee
We used that, due to our choice of $m$, $(4 \ell m)^{ 12 m } \leq n^{24/25}$.

We use \eqref{eq::E_R_ell_m_final} together with Markov's inequality:
\be 
\ba 
\P{\| R^{(\ell)}_k  \| > (\text{log} (n))^{25} \rho^{ \ell  - k/2}  } 
&\leq \frac{\E{\| R^{(\ell)}_k  \| ^{2 m}}}{(\text{log} (n))^{50m} \rho^{ m(2\ell  - k)}} \\
&\leq (c_6 \text{log} (n))^{-8m} \to 0.
\ea
\ee

\subsection{Bound $\| KB^{(k)} \|$}
Put
\be
m = \left\lfloor  \frac{ \log n }{ 13 \log (\log n)} \right\rfloor.
\ee
We have, with the convention that $e_{2m + 1} = e_1$,
\BA
\| KB^{(k-2)}  \| ^{2 m} & \leq  \tr  \{  \lr{  KB^{(k-2)} { KB^{(k-2)} }^*}^{m}   \} \\
&= \sum_{e_1, \ldots, e_{2m}}\prod_{i=1}^{m}  (KB^{(k-2)} ) _{e_{2i-1} , e_{2 i}}(KB^{(k-2)} ) _{e_{2i+1} , e_{2 i}}. 
\EA
Now,
\BA
\lr{KB^{(k-2)}}_{ef} 
&= \sum_g K_{eg} B_{gf}^{(k-2)} \\
&= \sum_g  \IND{e \to  g} \phi_{e_1} \phi_{e_2} W_{\sigma(e_1) \sigma(e_2)} \sum_{\gamma \in F^{k-1} _{g f}} \prod_{s=0}^{k-2} A_{\gamma_{s} \gamma_{s+1}} \\
&\leq c_1 \sum_g  \IND{e \to  g} \sum_{\gamma \in F^{k-1} _{g f}} \prod_{s=0}^{k-2} A_{\gamma_{s} \gamma_{s+1}}.
\EA
Hence,
\BA
&\| KB^{(k-2)}  \| ^{2 m} \\
& \leq c_2^m \sum_{e_1, \ldots, e_{2m}}\prod_{i=1}^{m}  \lr{ \sum_g  \IND{e_{2i-1} \to  g}  \sum_{\gamma \in F^{k-1} _{g e_{2i}}} \prod_{s=0}^{k-2} A_{\gamma_{s} \gamma_{s+1} }} \lr{ \sum_g  \IND{e_{2i+1} \to  g}  \sum_{\gamma \in F^{k-1} _{g e_{2i}}} \prod_{s=0}^{k-2} A_{\gamma_{s} \gamma_{s+1} }} \\
&= c_2^m \sum_{\gamma \in \overline {W}_{k,m}}\prod_{i=1}^{m}  \prod_{s=2}^{k} A_{\gamma_{2i-1,s-1} \gamma_{2i-1,s} } \prod_{s=1}^{k-1} A_{\gamma_{2i,s-1} \gamma_{2i,s}, }
\EA
 where $\overline W_{k,m}$ is the collection containing all sequences of paths $\gamma = ( \gamma_1, \ldots, \gamma_{2m})$ with $\gamma_i = (\gamma_{i,0}, \cdots, \gamma_{i,k})\in V^{k+1}$ is  non-backtracking such that
\begin{itemize}
\item for all $i$: $(\gamma_{i, k-1},\gamma_{i, k}) = (\gamma_{i+1,1}, \gamma_{i+1,0} )$,
\item for all \emph{odd} $i$: $(\gamma_{i,1}, \cdots, \gamma_{i,k})$ is tangle-free,
\item for all \emph{even} $i$: $(\gamma_{i,0}, \cdots, \gamma_{i,k-1})$ is tangle-free,
\end{itemize} 
with the convention that $\gamma_{2m+1} = \gamma_1$.

Recall the definition of $W_{k,m}$ and note that $ W_{k,m} \subset \overline W_{k,m} $. Fix $\overline{\gamma} \in \overline  W_{k,m} \setminus W_{k,m}$ and let $S_{\overline \gamma}$ be the set of all $\hat  {\gamma} \in \overline  W_{k,m} \setminus W_{k,m}$ such that for all odd $ i: (\hat \gamma_{i,1}, \cdots, \hat \gamma_{i,k}) = (\overline \gamma_{i,1}, \cdots, \overline \gamma_{i,k})$ and for all  even $ i: (\hat \gamma_{i,0}, \cdots, \hat \gamma_{i,k-1}) = (\overline  \gamma_{i,0}, \cdots, \overline \gamma_{i,k-1})$.
Then $|S_{\overline \gamma}| \leq k^m$. Indeed, if for odd $i$, $\hat  {\gamma}_i$ is \emph{not} tangle-free then necessarily $\hat{\gamma}_{i,0} \in \{ \hat{\gamma}_{i,1}, \ldots, \hat{\gamma}_{i,k} \},$
i.e., $\hat{\gamma}_{i,0}$ can be chosen in at most $k$ different ways. A similar argument works in case $i$ is even. 

Now, there always exists $\gamma \in W_{k,m}$ such that for all odd  $i: (  \gamma_{i,1}, \cdots,   \gamma_{i,k}) = (\overline \gamma_{i,1}, \cdots, \overline \gamma_{i,k})$ and for all even $i: (  \gamma_{i,0}, \cdots,   \gamma_{i,k-1}) = (\overline  \gamma_{i,0}, \cdots, \overline \gamma_{i,k-1}).$

As a consequence of these two observations, we have 
\be
\| KB^{(k-2)}  \| ^{2 m}
 \leq  c_2^m (1 + k^m) \sum_{\gamma \in  {W}_{k,m}}\prod_{i=1}^{m}  \prod_{s=2}^{k} A_{\gamma_{2i-1,s-1} \gamma_{2i-1,s} } \prod_{s=1}^{k-1} A_{\gamma_{2i,s-1} \gamma_{2i,s}. } \label{eq::bound_KB}
\ee 

To proceed following the method used to bound $\Delta^{(k)}$, note that the product in \eqref{eq::bound_KB} is taken over a path, consisting of $2m$ non-backtracking tangle-free subpaths of length $k-1$, that makes at most $2m$ backtracks. Hence Lemma's \ref{lm::Bound_injections_spins} and \ref{lm::bound_overshoot} may be adapted to the current setting (for instance the right hand side of \eqref{eq::bound_overshoot} becomes $e - (v-m-1) + 2m$), entailing
\BA &\E{\| KB^{(k-2)} \| ^{2 m}} \\
&\leq c_2^m (1 + k^m) \sum_{v=3}^{2km + 1} \sum_{e = v -1} ^{ 2km } |\cW_{k,m}|  \lr{\frac{c_3}{n}}^{e-(v-1)-m} c_4^{e-(v-m-1)+2m} n  \rho^{v-1} \\ 
&\leq c_5^m (1 + k^m) n^{m+1} \sum_{v=3}^{2km + 1} \sum_{e = v -1} ^{ 2km } |\cW_{k,m}|  \lr{\frac{c_6}{n}}^{e-(v-1)}  \rho^{v-1} \\
&\leq c_7^m (1 + k^m) n^{m+1} \rho^{2km} \ell^{2m} \ell m \sum_{s = 0} ^{ \infty }  \lr{\frac{c_6 (2\ell m)^{ 6 m } }{n}}^s \\
&\leq c_8^m (\ell m)^2 \ell^{3m} n^{m+1} \rho^{2km}\\
&\leq (c_9 \log n)^{19m} n^m \rho^{2km},
\EA
where we used our choice for $m$ several times. An appeal to Markov's inequality finishes the proof.

\subsection{Bound on $\| S_k^{(l)} \|$} 
This proof follows almost line-to-line the proof used in \cite{BoLeMa15} to establish bound $(34)$ there. We restrict ourselves here to the differences:

Observe that 
$ L_{ef} = 0 $ unless $ e \stackrel{2}{\to}  f $ does not hold, that is $e = f$, $e \to f$, $f^ {-1} \to  e $ or $ e \to  f^{-1} $, in which cases $ L_{ef}  = - \phi_{e_2} \phi_{f_1} W_{\sigma(e_2) \sigma(f_1)} $. Hence, we have the decomposition 
$$
L=    - I^* -  K^*,
$$
where $(I^*)_{ef} = \IND{e=f} \phi_{e_1} \phi_{e_2} W_{\sigma(e_1) \sigma(e_2)}$, 
and where $(K^*)_{ef} = \phi_{e_2} \phi_{f_1} W_{\sigma(e_2) \sigma(f_1)}$ if $e \to f$, $f^ {-1} \to  e $ or $ e \to  f^{-1} $ and $(K^*)_{ef} = 0$ otherwise. 

Thus
\[ \| S^{(\ell)}_k \|   \leq  \bdphi^2 (a \vee b)  \lr{   \| \Delta^{(k-1)} \|  \| B^{(\ell-k -1)} \|  +  \| \Delta^{(\ell-1)} K' \| \|  B^{(\ell-k-1)} \| } ,\]
where $K'$ is defined in \cite{BoLeMa15}. The rest of the proof follows after applying the arguments used in \cite{BoLeMa15} and following the procedure set out above to obtain the bound on $KB^{(k)}$.

\section{Proofs of Section \ref{sec::detection}}
\label{App::detection}

\begin{proof}[Proof of Lemma \ref{lm::40}]
Since $\widehat{\sigma}(v) = +$ if and only if  $F(v) = 1$, it follows that 
\[
\frac 1 n \sum_{v = 1}^n \IND{\sigma(v)  = +}   \IND{\widehat \sigma (v) = \sigma(v) } 
= \frac 1 n \sum_{v = 1}^n \IND{\sigma(v)  = +}   F(v)
 \to \frac { f(+)}{2}, 
\]
and
\[
\frac 1 n \sum_{v = 1}^n \IND{\sigma(v)  = -}   \IND{\widehat \sigma (v) = \sigma(v) } 
= \frac 1 n \sum_{v = 1}^n \IND{\sigma(v)  = -}  (1 - F(v))
 \to \frac {1 - f(-)}{2}. 
\]
Consequently,
\[
\frac 1 n \sum_{v = 1}^n  \IND{\widehat \sigma (v) = \sigma(v) } 
 \to \frac { 1 + f(+) - f(-)}{2} > \frac{1}{2}, 
\]
because $f(+) > f(-)$ by assumption.
\end{proof}

\begin{proof}[Proof of Lemma \ref{lm::41}]
We use Proposition \ref{prop::36} with 
\[\tau(G,v) =  \IND{\sigma(v) = i}  \IND{  I_{\ell} (v) \mu_2 ^{- 2\ell}  -  \widehat{c} \EF_2(i)  \geq t} .\]

Denote by $(T,o)$ the branching process defined in Section \ref{sec::branching} where the root has spin $\sigma_o$ uniformly drawn from $\spm$. Denote the number of offspring of the root by $D$ and let $Q_{\ell}(v)$ be equal to  $Q_{2,\ell}$ defined on the tree $T^v$ obtained after removing the subtree attached to $v$ from $T$. Then,
\[ \tau(T,o) =  \IND{\sigma_o = i}  \IND{  J_{\ell} \mu_2 ^{- 2\ell}  -  \widehat{c} \EF_2(i)  \geq t}, \] 
where
\be
J_{\ell} = \sum_{v=1}^D Q_{\ell} (v)  = (D-1) Q_{2,\ell}  -  L^o_{2,\ell}, \label{eq::J_k_ell} \ee
with $L^o_{2,\ell}$ defined in \eqref{eq::L_k_ell}. 

We need to calculate $\lim_{\ell \to \infty} \E{\tau(T,o)}$. To this end, we first show that, conditional on $\sigma_o = i$,  $\frac{ J_{\ell} }{\mu_2^{2\ell}} - \widehat{c} \EF_2(i)$ converges in probability to some \emph{centered} random variable $\widehat{Y}_{i}$.

We first calculate $\Esub{ J_{\ell} |\phi_o }{i}$, where $\Esub{\cdot}{i} = \E{\cdot| \sigma_o = i}$. Put $r_o = \frac{a+b}{2} \PHI \phi_o$, then
\BA
\Esub{ J_{\ell} |\phi_o }{i} 
&= \s{n=0}{\infty} \Esub{ J_{\ell} |D = n, \phi_o }{i} \P{D = n | \phi_o} \\
&= \sum_{n=0}^\infty n \Esub{ Q_{2,\ell} | D = n-1, \phi_o}{i} \frac{r_o^n e^{-r_o} }{ n !}  \\
&= r_o \sum_{n=1}^\infty  \Esub{ Q_{2,\ell} | D = n-1, \phi_o}{i} \frac{r_o^{n-1} e^{-r_o} }{ (n-1) !} \\
&= r_o \Esub{ Q_{2,\ell} | \phi_o}{i}.
\EA
Recall from Theorem \ref{th::25}, that $\frac{ Q_{2,\ell} }{\mu_2^{2\ell}}$ converges in $L^2$ to some random variable $X$, with mean ... Therefore, \BA  \E{ \left. \left| \frac{ Q_{2,\ell} }{\mu_2^{2\ell}} - X \right|  \right|  	\phi_0 = \psi_o} 
&= \s{z=0}{\infty} \E{ \left. \left| \frac{ Q_{2,\ell} }{\mu_2^{2\ell}} - X \right| \right| \|Z_1\| = z} \ \P{\|Z_1\| = z |\phi_0 = \psi_o} \\
&\leq e^{\frac{a+b}{2}\PHI (\PHImax - \PHImin)} \E{\left. \left| \frac{ Q_{2,\ell} }{\mu_2^{2\ell}} - X \right|  \right| \phi_0 = \PHImax}  \EA

Recall from Theorem \ref{th::25} that, uniformly for all $\psi_o$,   
\[\Esub{\left. \frac{Q_{2,\ell}}{\mu_2 ^{2 \ell}} \right| \phi_o = \psi_o}{i} \to \frac{ \PHIthree}{\PHItwo} \frac{\rho}{\mu_2^2 - \rho} \mu_{2,\psi_o}  \EF_2 (i)\] as $n \to \infty$. Hence, $ \sup_{n, \psi_o} \Esub{ \left. \frac{ Q_{2,\ell} }{\mu_2^{2\ell}} \right| \phi_o = \psi_o}{i} < \infty$, so that we can apply Lebesque's dominated convergence theorem:
\be \frac{\Esub{ J_{2,\ell}}{i}}{\mu_2^{2\ell}} = \Esub{r_o \Esub{ \left. \frac{ Q_{2,\ell} }{\mu_2^{2\ell}} \right| \phi_o}{i}}{i} \to \widehat{c} \EF_2(i), \label{eq::J_exp_limit} \ee
as $n \to \infty$. 

We now combine the right hand side of \eqref{eq::J_k_ell}, \eqref{eq::J_exp_limit}, and Theorem \ref{th::25} (and in particular \eqref{eq::Var_L_u} which implies that $L^o_{2,\ell} / \mu_2^{2\ell} \to 0$ as $n \to \infty$) to establish the claim that, conditional on $\sigma_o = i$, 
 $\frac{ J_{\ell} }{\mu_2^{2\ell}} - \widehat{c} \EF_k(i)$ converges in probability to some \emph{centered} random variable $\widehat{Y}_{i}$.

In particular, conditional on $\sigma_o = i$, $\frac{ J_{\ell} }{\mu_2^{2\ell}} - \widehat{c} \EF_2(i)$ converges in distribution to $\widehat{Y}_{i}$. So that, for $t$ as in the statement, 
\[ \E{\tau(T,o)} = \frac{1}{2} \P{ \left. \frac{ J_{\ell} }{\mu_2^{2\ell}} - \widehat{c} \EF_2(i)  \geq t \right| \sigma_o = i } \to \frac{1}{2} \P{\widehat{Y}_{i} \geq t}, \]
as $n \to \infty$. 

Finally, noting that the error term in Proposition \ref{prop::36} is $\bigO \lr{ n^{- \lr{\frac{\gamma}{2} \wedge \frac{1}{40} }} } = o(1)$ finishes the proof.
\end{proof}

\begin{proof}[Proof of Lemma \ref{lm::42}]
This follows after repeating the proof in \cite{BoLeMa15}  in conjunction with Lemma \ref{lm::41} established here.
\end{proof}

\bibliographystyle{abbrv}
\bibliography{literature}

\end{document}